\documentclass[a4paper,oneside,10pt]{article}%
\usepackage{amsmath}
\usepackage{amsfonts}
\usepackage{amssymb}
\usepackage{graphicx}
\usepackage{color}
\usepackage[square,numbers,sort&compress]{natbib}%
\providecommand{\U}[1]{\protect \rule{.1in}{.1in}}

\newtheorem{theorem}{Theorem}[section]

\newtheorem{definition}[theorem]{Definition}

\newtheorem{example}[theorem]{Example}

\newtheorem{lemma}[theorem]{Lemma}

\newtheorem{proposition}[theorem]{Proposition}
\newtheorem{remark}[theorem]{Remark}

\newenvironment{proof}[1][Proof]{\noindent \textbf{#1.} }{\  $\Box$}
\numberwithin{equation}{section}

\begin{document}

\title{BSDEs driven by $G$-Brownian motion under degenerate case and its application
to the regularity of fully nonlinear PDEs}
\author{Mingshang Hu \thanks{Zhongtai Securities Institute for Financial Studies,
Shandong University, Jinan, Shandong 250100, PR China. humingshang@sdu.edu.cn.
Research supported by National Key R\&D Program of China (No. 2018YFA0703900)
and NSF (No. 11671231). }
\and Shaolin Ji\thanks{Zhongtai Securities Institute for Financial Studies,
Shandong University, Jinan, Shandong 250100, PR China. jsl@sdu.edu.cn.
Research supported by NSF (No. 11971263 and 11871458). }
\and Xiaojuan Li\thanks{Zhongtai Securities Institute for Financial Studies,
Shandong University, Jinan 250100, China. Email: lixiaojuan@mail.sdu.edu.cn.} }
\maketitle

\textbf{Abstract}. In this paper, we obtain the existence and uniqueness
theorem for backward stochastic differential equation driven by $G$-Brownian
motion ($G$-BSDE) under degenerate case. Moreover, we propose a new
probabilistic method based on the representation theorem of $G$-expectation
and weak convergence to obtain the regularity of fully nonlinear PDE
associated to $G$-BSDE.



\addcontentsline{toc}{section}{\hspace*{1.8em}Abstract}

\section{Introduction}

Motivated by volatility uncertainty in finance, Peng \cite{Peng2004, Peng2005,
P07a, P08a} introduced the notions of $G$-expectation $\mathbb{\hat{E}}%
[\cdot]$ and $G$-Brownian motion $B$ for each monotone and sublinear function
$G:\mathbb{S}_{d}\rightarrow \mathbb{R}$. The Ito's calculus with respect to
$G$-Brownian motion was constructed. Furthermore, he studied stochastic
differential equation driven by $G$-Brownian motion ($G$-SDE) and a special
type of backward stochastic differential equation (BSDE) containing only the
solution $Y$, and then established \ the relevant theory. Denis et al.
\cite{DHP11} (see also \cite{HP09}) obtained that the $G$-expectation can be
represented as an upper expectation over a family of weakly compact and
non-dominated probability measures $\mathcal{P}$, and gave the
characterizations of some spaces by inner capacity associated to $\mathcal{P}%
$. By quasi-surely stochastic analysis based on outer capacity, Denis and
Martini \cite{DenisMartini2006} made a great contribution to study
super-pricing of contingent claims under volatility uncertainty. The
relationship between these two capacities has been clearly explained in notes
and comments of Chapter 6 in \cite{P2019}.

Hu et al. \cite{HJPS1} studied the following BSDE driven by $G$-Brownian
motion ($G$-BSDE)%
\begin{equation}
Y_{t}=\xi+\int_{t}^{T}f(s,Y_{s},Z_{s})ds+\int_{t}^{T}g(s,Y_{s},Z_{s})d\langle
B\rangle_{s}-\int_{t}^{T}Z_{s}dB_{s}-(K_{T}-K_{t}) \label{e0-1}%
\end{equation}
under the non-degenerate $G$, i.e., there exists a constant $\underline
{\sigma}^{2}>0$ such that
\[
G(A)-G(B)\geq \frac{1}{2}\underline{\sigma}^{2}\mathrm{tr}[A-B]\text{ for
}A\geq B.
\]
They proved that the above $G$-BSDE has a unique solution $(Y,Z,K)$, where $K$
is a non-increasing $G$-martingale with $K_{0}=0$. Soner et al. \cite{STZ11}
studied a new type of fully nonlinear BSDE, called $2$BSDE, by different
formulation and method, and obtained the deep result of the existence and
uniqueness theorem for $2$BSDE. For recent advances in these two directions,
the reader may refer to \cite{DK, HJ0, HJ1, HYH, LPH, LRT, MPZ, PZ} and the
references therein.

The key step to obtain the solution of $G$-BSDE (\ref{e0-1}) under
non-degenerate $G$ is to use Krylov's regularity estimate for fully nonlinear
PDEs (see Appendix C.4 in \cite{P2019}). But under degenerate $G$, we have to
get round the difficulty that the regularity estimation condition (see
Definition C.4.3 in \cite{P2019}) is not satisfied. A natural idea is to
construct a family of non-degenerate $G_{\varepsilon}$ with $\varepsilon
\in(0,\varepsilon_{0}]$ such that $G_{\varepsilon}\uparrow G$ as
$\varepsilon \downarrow0$. The corresponding $G_{\varepsilon}$-expectation and
the set of probability measures are denoted by $\mathbb{\hat{E}}^{\varepsilon
}[\cdot]$ and $\mathcal{P}^{\varepsilon}$, respectively. By the definition of
$G$-expectation, we know that $\mathbb{\hat{E}}^{\varepsilon}[X]\uparrow
\mathbb{\hat{E}}[X]$ for $X\in L_{G}^{1}(\Omega_{T})$ and $\mathcal{P}$ is the
closure of $\mathcal{P}_{1}:=\cup_{\varepsilon>0}\mathcal{P}^{\varepsilon}$
under the topology of weak convergece. It is important to note that the
quasi-surely stochastic analysis with respect to $\mathcal{P}$ (i.e.
$\mathcal{P}$-q.s.) and $\mathcal{P}_{1}$ (i.e. $\mathcal{P}_{1}$-q.s.) are
different (see \cite{HWZ}). Following the method proposed in the proof of
Proposition A.1. in \cite{STZ}, we can get a process $Z$ in the $\mathcal{P}%
_{1}$-q.s. sence such that%
\[
\inf_{\eta \in M^{0}(0,T)}\mathbb{\hat{E}}^{\varepsilon}\left[  \left(
\int_{0}^{T}|Z_{s}-\eta_{s}|^{2}d\langle B\rangle_{s}\right)  ^{p/2}\right]
=0\text{ for }\varepsilon>0\text{, }p>1\text{,}%
\]
where the definition of $M^{0}(0,T)$ can be found in Section 3. At this point,
there is a natural misconception that $Z\in H_{G}^{2,p}(0,T;\langle B\rangle)$
holds. But we notice Sion's minimax theorem can not be used to obtain%
\[
\inf_{\eta \in M^{0}(0,T)}\sup_{\varepsilon \in(0,\varepsilon_{0}]}%
\mathbb{\hat{E}}^{\varepsilon}\left[  \left(  \int_{0}^{T}|Z_{s}-\eta_{s}%
|^{2}d\langle B\rangle_{s}\right)  ^{p/2}\right]  =\sup_{\varepsilon
\in(0,\varepsilon_{0}]}\inf_{\eta \in M^{0}(0,T)}\mathbb{\hat{E}}^{\varepsilon
}\left[  \left(  \int_{0}^{T}|Z_{s}-\eta_{s}|^{2}d\langle B\rangle_{s}\right)
^{p/2}\right]  ,
\]
because $(0,\varepsilon_{0}]$ is not compact. Therefore, whether $Z$ belongs
to $H_{G}^{2,p}(0,T;\langle B\rangle)$ remains unsolved, even for
$G$-martingale representation theorem, which is a special case of $G$-BSDE
(\ref{e0-1}), i.e., $f=g=0$. Thus, one purpose of this paper is to investigate
the existence and uniqueness theorem for $G$-BSDE (\ref{e0-1}) under
degenerate $G$.

It is well known that the theory of classical BSDEs provides a tool to study
the regularity of quasilinear PDEs (see \cite{P-P92}). However, we all know
that this classical tool is not suitable for the regularity of fully nonlinear
PDEs, and up to our knowledge there is no result on this field. So, the other
purpose of this paper is to establish the regularity of fully nonlinear PDEs
by $G$-BSDEs.

In this paper, we introduce a quite different method to study the existence
and uniqueness theorem for a type of well-posed $G$-BSDEs under degenerate $G$
(see (\ref{new-e2-4})), which has two major contributions. The first one is to
obtain the soution $(Y,Z,K)$ for $G$-BSDE under degenerate $G$ in the extended
$\tilde{G}$-expectation space, which is essential to show that $K$ is a
$G$-martingale in the key Lemma \ref{pro2-6}. The second one is to propose a
new probabilistic method based on the representation theorem of $G$%
-expectation and weak convergence to obtain the uniform lower bound for
$\partial_{xx}^{2}u_{\varepsilon}$ with $\varepsilon>0$, where $u_{\varepsilon
}$ is a soution to a fully nonlinear PDE associated to a $G_{\varepsilon}%
$-BSDE under non-degenerate $G_{\varepsilon}$ (see (\ref{new-e2-11}) and
(\ref{e2-26})). This uniform lower bound for $\partial_{xx}^{2}u_{\varepsilon
}$ plays a key role in proving $Z\in H_{G}^{2,p}(0,T;\langle B\rangle)$ in
Lemma \ref{pro2-6}, and up to our knowledge, it is completely new in the
literature because it does not depend on $\varepsilon$ as the bound by
Krylov's regularity estimate for fully nonlinear PDEs. Finally, we use the
above probabilistic method to obtain the regularity of fully nonlinear PDE
associated to $G$-BSDE under degenerate $G$.

The paper is organized as follows. In Section 2, we present some basic results
of $G$-expectations. The existence and uniqueness theorem for $G$-BSDE under
degenerate case is established in Section 3. In Section 4, we obtain the
regularity of fully nonlinear PDE associated to $G$-BSDE under degenerate $G$. \ 

\section{Preliminaries}

We recall some basic results of $G$-expectations. The readers may refer to
Peng's book \cite{P2019} for more details.

Let $T>0$ be given and let $\Omega_{T}=C_{0}([0,T];\mathbb{R}^{d})$ be the
space of $\mathbb{R}^{d}$-valued continuous functions on $[0,T]$ with
$\omega_{0}=0$. The canonical process $B_{t}(\omega):=\omega_{t}$, for
$\omega \in \Omega_{T}$ and $t\in \lbrack0,T]$. For any fixed $t\leq T$, set%
\[
Lip(\Omega_{t}):=\{ \varphi(B_{t_{1}},B_{t_{2}}-B_{t_{1}},\ldots,B_{t_{N}%
}-B_{t_{N-1}}):N\geq1,t_{1}<\cdots<t_{N}\leq t,\varphi \in C_{b.Lip}%
(\mathbb{R}^{d\times N})\},
\]
where $C_{b.Lip}(\mathbb{R}^{d\times N})$ denotes the space of bounded
Lipschitz functions on $\mathbb{R}^{d\times N}$.

Let $G:\mathbb{S}_{d}\rightarrow \mathbb{R}$ be a given monotonic and sublinear
function, where $\mathbb{S}_{d}$ denotes the set of $d\times d$ symmetric
matrices. Then there exists a unique bounded, convex and closed set
$\Sigma \subset \mathbb{S}_{d}^{+}$ such that%
\begin{equation}
G(A)=\frac{1}{2}\sup_{\gamma \in \Sigma}\mathrm{tr}[A\gamma]\text{ for }%
A\in \mathbb{S}_{d}, \label{new-e1-1}%
\end{equation}
where $\mathbb{S}_{d}^{+}$ denotes the set of $d\times d$ nonnegative
matrices. If there exists a $\underline{\sigma}^{2}>0$ such that $\gamma
\geq \underline{\sigma}^{2}I_{d}$ for any $\gamma \in \Sigma$, $G$ is called
non-degenerate. Otherwise, $G$ is called degenerate.

Peng \cite{P07a, P08a} constructed the $G$-expectation $\mathbb{\hat{E}%
}:Lip(\Omega_{T})\rightarrow \mathbb{R}$ and the conditional $G$-expectation
$\mathbb{\hat{E}}_{t}:Lip(\Omega_{T})\rightarrow Lip(\Omega_{t})$ as follows:

\begin{description}
\item[(i)] For each $s_{1}\leq s_{2}\leq T$ and $\varphi \in C_{b.Lip}%
(\mathbb{R}^{d})$, define $\mathbb{\hat{E}}[\varphi(B_{s_{2}}-B_{s_{1}%
})]=u(s_{2}-s_{1},0)$, where $u$ is the viscosity solution (see \cite{CIP}) of
the following $G$-heat equation:%
\[
\partial_{t}u-G(D_{x}^{2}u)=0,\ u(0,x)=\varphi(x).
\]

\item[(ii)] For each $X=\varphi(B_{t_{1}},B_{t_{2}}-B_{t_{1}},\ldots,B_{t_{N}%
}-B_{t_{N-1}})\in Lip(\Omega_{T})$, define
\[
\mathbb{\hat{E}}_{t_{i}}[X]=\varphi_{i}(B_{t_{1}},\ldots,B_{t_{i}}-B_{t_{i-1}%
})\text{ for }i=N-1,\ldots,1\text{ and }\mathbb{\hat{E}}[X]=\mathbb{\hat{E}%
}[\varphi_{1}(B_{t_{1}})],
\]
where $\varphi_{N-1}(x_{1},\ldots,x_{N-1}):=\mathbb{\hat{E}}[\varphi
(x_{1},\ldots,x_{N-1},B_{t_{N}}-B_{t_{N-1}})]$ for $(x_{1},\ldots,x_{N-1}%
)\in \mathbb{R}^{d\times(N-1)}$ and
\[
\varphi_{i}(x_{1},\ldots,x_{i}):=\mathbb{\hat{E}}[\varphi_{i+1}(x_{1}%
,\ldots,x_{i},B_{t_{i+1}}-B_{t_{i}})]\text{ for }i=N-2,\ldots,1.
\]

\end{description}

The space $(\Omega_{T},Lip(\Omega_{T}),\mathbb{\hat{E}},(\mathbb{\hat{E}}%
_{t})_{t\in \lbrack0,T]})$ is a consistent sublinear expectation space, where
$\mathbb{\hat{E}}_{0}=\mathbb{\hat{E}}$. The canonical process $(B_{t}%
)_{t\in \lbrack0,T]}$ is called the $G$-Brownian motion under $\mathbb{\hat{E}%
}$.

For each $t\in \lbrack0,T]$, denote by $L_{G}^{p}(\Omega_{t})$ the completion
of $Lip(\Omega_{t})$ under the norm $||X||_{L_{G}^{p}}:=(\mathbb{\hat{E}%
}[|X|^{p}])^{1/p}$ for $p\geq1$. It is clear that $\mathbb{\hat{E}}_{t}$ can
be continuously extended to $L_{G}^{1}(\Omega_{T})$ under the norm
$||\cdot||_{L_{G}^{1}}$.

The following theorem is the representation theorem of $G$-expectation.

\begin{theorem}
(\cite{DHP11, HP09}) There exists a unique weakly compact and convex set of
probability measures $\mathcal{P}$ on $(\Omega_{T},\mathcal{B}(\Omega_{T}))$
such that%
\[
\mathbb{\hat{E}}[X]=\sup_{P\in \mathcal{P}}E_{P}[X]\text{ for all }X\in
L_{G}^{1}(\Omega_{T}),
\]
where $\mathcal{B}(\Omega_{T})=\sigma(B_{s}:s\leq T)$.
\end{theorem}

For this $\mathcal{P}$, define%
\[
\mathbb{L}^{p}(\Omega_{t}):=\left \{  X\in \mathcal{B}(\Omega_{t}):\sup
_{P\in \mathcal{P}}E_{P}[|X|^{p}]<\infty \right \}  \text{ for }p\geq1.
\]
It is easy to check that $L_{G}^{p}(\Omega_{t})\subset \mathbb{L}^{p}%
(\Omega_{t})$. For each $X\in \mathbb{L}^{1}(\Omega_{T})$,
\[
\mathbb{\hat{E}}[X]:=\sup_{P\in \mathcal{P}}E_{P}[X]
\]
is still called the $G$-expectation.

The capacity associated to $\mathcal{P}$ is defined by
\[
c(A):=\sup_{P\in \mathcal{P}}P(A)\text{ for }A\in \mathcal{B}(\Omega_{T}).
\]
A set $A\in \mathcal{B}(\Omega_{T})$ is polar if $c(A)=0$. A property holds
\textquotedblleft quasi-surely" (q.s. for short) if it holds outside a polar
set. In the following, we do not distinguish two random variables $X$ and $Y$
if $X=Y$ q.s.

\begin{definition}
A process $(M_{t})_{t\leq T}$ is called a $G$-martingale if $M_{t}\in
L_{G}^{1}(\Omega_{t})$ and $\mathbb{\hat{E}}_{s}[M_{t}]=M_{s}$ for any $0\leq
s\leq t\leq T$.
\end{definition}

The following Doob's inequality for $G$-martingale can be found in \cite{STZ,
Song11}. The following proof is based on \cite{HJL, STZ}.

\begin{theorem}
\label{th1-1}Let $1\leq p<p^{\prime}$ and $\xi \in L_{G}^{p^{\prime}}%
(\Omega_{T})$. Then
\begin{equation}
\left(  \hat{\mathbb{E}}\left[  \sup_{t\leq T}\left(  \hat{\mathbb{E}}%
_{t}[|\xi|]\right)  ^{p}\right]  \right)  ^{1/p}\leq \left(  \hat{\mathbb{E}%
}\left[  \sup_{t\leq T}\hat{\mathbb{E}}_{t}[|\xi|^{p}]\right]  \right)
^{1/p}\leq C\left(  \hat{\mathbb{E}}[|\xi|^{p^{\prime}}]\right)
^{1/p^{\prime}}, \label{e1-2}%
\end{equation}
where%
\[
C=\left(  1+\frac{p}{p^{\prime}-p}\right)  ^{1/p}.
\]

\end{theorem}

\begin{proof}
By the definition of $L_{G}^{p^{\prime}}(\Omega_{T})$, we only need to prove
the inequality for $\xi \in Lip(\Omega_{T})$. Define%
\[
M_{t}=\hat{\mathbb{E}}_{t}[|\xi|]\text{ for }t\leq T.
\]
For each fixed $\lambda>0$ and integer $n\geq1$, define a stopping time%
\[
\tau=\inf \{t_{i}:M_{t_{i}}\geq \lambda,i=0,\ldots,n\},
\]
where $t_{i}=iT/n$ and $\inf \emptyset=\infty$. It is easy to check that%
\[
\{ \tau=t_{i}\} \in \mathcal{B}(\Omega_{t_{i}}),\text{ }\{ \tau=\infty \}
\in \mathcal{B}(\Omega_{T})\text{ and }\{ \tau=t_{i}\} \cap \{ \tau
=t_{j}\}=\emptyset \text{ for }i\not =j\text{.}%
\]
By Proposition 3.9 in \cite{HJL}, we have%
\[
\hat{\mathbb{E}}\left[  \sum_{i=0}^{n}|\xi|I_{\{ \tau=t_{i}\}}+0I_{\{
\tau=\infty \}}\right]  =\hat{\mathbb{E}}\left[  \sum_{i=0}^{n}\hat{\mathbb{E}%
}_{t_{i}}[|\xi|]I_{\{ \tau=t_{i}\}}+\hat{\mathbb{E}}_{T}[0]I_{\{ \tau
=\infty \}}\right]  ,
\]
which implies%
\[
\hat{\mathbb{E}}\left[  |\xi|I_{\{ \tau \leq t_{n}\}}\right]  =\hat{\mathbb{E}%
}\left[  \sum_{i=0}^{n}M_{t_{i}}I_{\{ \tau=t_{i}\}}\right]  \geq \lambda
\hat{\mathbb{E}}\left[  I_{\{ \tau \leq t_{n}\}}\right]  .
\]
Note that $\{ \tau \leq t_{n}\}=\{ \sup_{i}M_{t_{i}}\geq \lambda \}$, then we
have%
\[
\lambda \hat{\mathbb{E}}\left[  I_{\{ \sup_{i}M_{t_{i}}\geq \lambda \}}\right]
\leq \hat{\mathbb{E}}\left[  |\xi|I_{\{ \sup_{i}M_{t_{i}}\geq \lambda \}}\right]
\leq \left(  \hat{\mathbb{E}}[|\xi|^{p^{\prime}}]\right)  ^{1/p^{\prime}%
}\left(  \hat{\mathbb{E}}\left[  I_{\{ \sup_{i}M_{t_{i}}\geq \lambda \}}\right]
\right)  ^{1/q^{\prime}},
\]
where $1/p^{\prime}+1/q^{\prime}=1$. Thus,%
\[
\hat{\mathbb{E}}\left[  I_{\{ \sup_{i}M_{t_{i}}\geq \lambda \}}\right]
\leq \frac{1}{\lambda^{p^{\prime}}}\hat{\mathbb{E}}\left[  |\xi|^{p^{\prime}%
}\right]  \text{ for each }\lambda>0\text{.}%
\]
For each fixed $\lambda_{0}>0$, we have%
\begin{align*}
\hat{\mathbb{E}}\left[  \sup_{i}M_{t_{i}}^{p}\right]   &  =\sup_{P\in
\mathcal{P}}E_{P}\left[  \sup_{i}M_{t_{i}}^{p}\right] \\
&  =\sup_{P\in \mathcal{P}}p\int_{0}^{\infty}P(\sup_{i}M_{t_{i}}\geq
\lambda)\lambda^{p-1}d\lambda \\
&  \leq \int_{0}^{\lambda_{0}}p\lambda^{p-1}d\lambda+\int_{\lambda_{0}}%
^{\infty}p\lambda^{p-1-p^{\prime}}\hat{\mathbb{E}}[|\xi|^{p^{\prime}}%
]d\lambda \\
&  =(\lambda_{0})^{p}+\frac{p\lambda_{0}^{p-p^{\prime}}}{p^{\prime}-p}%
\hat{\mathbb{E}}[|\xi|^{p^{\prime}}].
\end{align*}
Taking $\lambda_{0}=\left(  \hat{\mathbb{E}}[|\xi|^{p^{\prime}}]\right)
^{1/p^{\prime}}$, we get%
\[
\hat{\mathbb{E}}\left[  \sup_{i}M_{t_{i}}^{p}\right]  \leq \left(  1+\frac
{p}{p^{\prime}-p}\right)  \left(  \hat{\mathbb{E}}[|\xi|^{p^{\prime}}]\right)
^{p/p^{\prime}}.
\]
Since $|\xi|\in L_{ip}(\Omega_{T})$, we have%
\[
\sup_{i}M_{t_{i}}^{p}\uparrow \sup_{t\leq T}M_{t}^{p}.
\]
Then we obtain%
\begin{equation}
\hat{\mathbb{E}}\left[  \sup_{t\leq T}\left(  \hat{\mathbb{E}}_{t}%
[|\xi|]\right)  ^{p}\right]  \leq \left(  1+\frac{p}{p^{\prime}-p}\right)
\left(  \hat{\mathbb{E}}[|\xi|^{p^{\prime}}]\right)  ^{p/p^{\prime}}.
\label{e1-1}%
\end{equation}

It is obvious that $\left(  \hat{\mathbb{E}}_{t}[|\xi|]\right)  ^{p}\leq
\hat{\mathbb{E}}_{t}[|\xi|^{p}]$. Since inequality (\ref{e1-1}) holds for
$|\xi|^{p}\in L_{ip}(\Omega_{T})$ and $1<p^{\prime}/p$, we have%
\[
\hat{\mathbb{E}}\left[  \sup_{t\leq T}\hat{\mathbb{E}}_{t}[|\xi|^{p}]\right]
\leq \left(  1+\frac{1}{p^{\prime}/p-1}\right)  \left(  \hat{\mathbb{E}}%
[|\xi|^{p^{\prime}}]\right)  ^{p/p^{\prime}}=\left(  1+\frac{p}{p^{\prime}%
-p}\right)  \left(  \hat{\mathbb{E}}[|\xi|^{p^{\prime}}]\right)
^{p/p^{\prime}}.
\]
Thus we obtain (\ref{e1-2}).
\end{proof}

\section{BSDEs driven by $G$-Brownian motion under degenerate case}

Let $B_{t}=(B_{t}^{1},\ldots,B_{t}^{d})^{T}$ be a $d$-dimensional $G$-Brownian
motion satisfying%
\begin{equation}
G(A)=G^{\prime}(A^{\prime})+\frac{1}{2}\sum_{i=d^{\prime}+1}^{d}\bar{\sigma
}_{i}^{2}a_{i}^{+}, \label{new-e2-2}%
\end{equation}
where $d^{\prime}<d$, $A^{\prime}\in \mathbb{S}_{d^{\prime}}$, $a_{i}%
\in \mathbb{R}$ for $d^{\prime}<i\leq d$,
\[
A=\left(
\begin{array}
[c]{cccc}%
A^{\prime} & \cdots & \cdots & \cdots \\
\cdots & a_{d^{\prime}+1} & \cdots & \cdots \\
\vdots & \vdots & \ddots & \vdots \\
\cdots & \cdots & \cdots & a_{d}%
\end{array}
\right)  \in \mathbb{S}_{d},
\]
$G^{\prime}:\mathbb{S}_{d^{\prime}}\rightarrow \mathbb{R}$ is non-degenerate,
$\bar{\sigma}_{i}>0$ for $i=d^{\prime}+1,\ldots,d$. By Corollary 3.5.8 in Peng
\cite{P2019}, we know that%
\begin{equation}
(\langle B^{i},B^{j}\rangle_{t+s}-\langle B^{i},B^{j}\rangle_{t})_{i,j=1}%
^{d}\in s\Sigma \text{ for any }t\text{, }s\geq0, \label{new-e2-3}%
\end{equation}
where $\langle B^{i},B^{j}\rangle$ is the mutual variation process of $B^{i}$
and $B^{j}$, and $\Sigma \subset \mathbb{S}_{d}^{+}$ is the unique bounded,
convex and closed set satisfying (\ref{new-e1-1}). It follows from
(\ref{new-e2-2}) and (\ref{new-e2-3}) that, for any $t$, $s\geq0$,%
\begin{equation}
cs\leq \langle B^{i}\rangle_{t+s}-\langle B^{i}\rangle_{t}\leq Cs\text{ for
}i\leq d^{\prime},\text{ }\langle B^{i}\rangle_{t+s}-\langle B^{i}\rangle
_{t}\leq \bar{\sigma}_{i}^{2}s\text{ for }d^{\prime}<i\leq d, \label{new-e2-5}%
\end{equation}%
\[
\langle B^{i},B^{j}\rangle_{t}=0\text{ for }i\leq d\text{, }d^{\prime}<j\leq
d\text{, }i\not =j\text{,}%
\]
where $\langle B^{i}\rangle=\langle B^{i},B^{i}\rangle$, $0<c\leq C<\infty$.
We consider the following type of $G$-BSDE under degenerate case:%
\begin{equation}%
\begin{array}
[c]{rl}%
Y_{t}= & \xi+\int_{t}^{T}f(s,Y_{s},Z_{s}^{\prime})ds+\sum_{i,j=1}^{d^{\prime}%
}\int_{t}^{T}g_{ij}(s,Y_{s},Z_{s}^{\prime})d\langle B^{i},B^{j}\rangle_{s}\\
& +\sum_{l=d^{\prime}+1}^{d}\int_{t}^{T}g_{l}(s,Y_{s},Z_{s}^{\prime},Z_{s}%
^{l})d\langle B^{l}\rangle_{s}-\sum_{k=1}^{d}\int_{t}^{T}Z_{s}^{k}dB_{s}%
^{k}-(K_{T}-K_{t}),
\end{array}
\label{new-e2-4}%
\end{equation}
where $Z_{s}^{\prime}=(Z_{s}^{1},\ldots,Z_{s}^{d^{\prime}})^{T}$,%
\[
f,g_{ij}:[0,T]\times \Omega_{T}\times \mathbb{R}\times \mathbb{R}^{d^{\prime}%
}\rightarrow \mathbb{R}\text{, }g_{l}:[0,T]\times \Omega_{T}\times
\mathbb{R}\times \mathbb{R}^{d^{\prime}}\times \mathbb{R}\rightarrow
\mathbb{R}\text{.}%
\]
The following spaces and norms are needed to define the solution of the above
$G$-BSDE.

\begin{itemize}
\item $M^{0}(0,T):=\left \{  \eta_{t}=\sum_{k=0}^{N-1}\xi_{k}I_{[t_{k}%
,t_{k+1})}(t):N\in \mathbb{N}\text{, }0=t_{0}<\cdots<t_{N}=T,\text{ }\xi_{k}\in
Lip(\Omega_{t_{k}})\right \}  $;

\item $||\eta||_{M_{G}^{p,\bar{p}}(0,T)}:=\left(  \mathbb{\hat{E}}\left[
\left(  \int_{0}^{T}|\eta_{t}|^{p}dt\right)  ^{\bar{p}/p}\right]  \right)
^{1/\bar{p}}$, $||\eta||_{H_{G}^{p,\bar{p}}(0,T;\langle B^{i}\rangle
)}:=\left(  \mathbb{\hat{E}}\left[  \left(  \int_{0}^{T}|\eta_{t}|^{p}d\langle
B^{i}\rangle_{t}\right)  ^{\bar{p}/p}\right]  \right)  ^{1/\bar{p}}$;

\item $M_{G}^{p,\bar{p}}(0,T):=\left \{  \text{the completion of }%
M^{0}(0,T)\text{ under the norm }||\cdot||_{M_{G}^{p,\bar{p}}(0,T)}\right \}  $
for $p$, $\bar{p}\geq1$;

\item $H_{G}^{p,\bar{p}}(0,T;\langle B^{i}\rangle):=\left \{  \text{the
completion of }M^{0}(0,T)\text{ under the norm }||\cdot||_{H_{G}^{p,\bar{p}%
}(0,T;\langle B^{i}\rangle)}\right \}  $ for $p$, $\bar{p}\geq1$;

\item $M_{G}^{p}(0,T):=M_{G}^{p,p}(0,T)$, $H_{G}^{p}(0,T;\langle B^{i}%
\rangle):=H_{G}^{p,p}(0,T;\langle B^{i}\rangle)$;

\item $S^{0}(0,T):=\left \{  h(t,B_{t_{1}\wedge t},\ldots,B_{t_{N}\wedge
t}):N\in \mathbb{N}\text{, }0<t_{1}<\cdots<t_{N}=T,\text{ }h\in C_{b.Lip}%
(\mathbb{R}^{N+1})\right \}  $;

\item $||\eta||_{S_{G}^{p}(0,T)}:=\left(  \mathbb{\hat{E}}\left[  \sup_{t\leq
T}|\eta_{t}|^{p}\right]  \right)  ^{1/p}$;

\item $S_{G}^{p}(0,T):=\left \{  \text{the completion of }S^{0}(0,T)\text{
under the norm }||\cdot||_{S_{G}^{p}(0,T)}\right \}  $ for $p\geq1$.
\end{itemize}

By (\ref{new-e2-5}), we know that%
\[
c^{1/p}||\eta||_{M_{G}^{p,\bar{p}}(0,T)}\leq||\eta||_{H_{G}^{p,\bar{p}%
}(0,T;\langle B^{i}\rangle)}\leq C^{1/p}||\eta||_{M_{G}^{p,\bar{p}}%
(0,T)}\text{ for }i\leq d^{\prime}%
\]
and%
\[
||\eta||_{H_{G}^{p,\bar{p}}(0,T;\langle B^{i}\rangle)}\leq \bar{\sigma}%
_{i}^{2/p}||\eta||_{M_{G}^{p,\bar{p}}(0,T)}\text{ for }d^{\prime}<i\leq d.
\]
Thus $M_{G}^{p,\bar{p}}(0,T)=H_{G}^{p,\bar{p}}(0,T;\langle B^{i}\rangle)$ for
$i\leq d^{\prime}$ and $M_{G}^{p,\bar{p}}(0,T)\subset H_{G}^{p,\bar{p}%
}(0,T;\langle B^{i}\rangle)$ for $d^{\prime}<i\leq d$.

Throughout the paper, we use the following assumptions:

\begin{description}
\item[(H1)] There exists a $\bar{p}>1$ such that $\xi \in L_{G}^{\bar{p}%
}(\Omega_{T})$, $f(\cdot,y,z^{\prime})$, $g_{ij}(\cdot,y,z^{\prime})\in
M_{G}^{1,\bar{p}}(0,T)$ and $g_{l}(\cdot,y,z^{\prime},z)\in H_{G}^{1,\bar{p}%
}(0,T;\langle B^{l}\rangle)$ for any $y$, $z\in \mathbb{R}$, $z^{\prime}%
\in \mathbb{R}^{d^{\prime}}$, $i$, $j\leq d^{\prime}$, $d^{\prime}<l\leq d$;

\item[(H2)] There exists a constant $L>0$ such that, for any $(t,\omega
)\in \lbrack0,T]\times \Omega_{T}$, $(y,z^{\prime},z)$, $(\bar{y},\bar
{z}^{\prime},\bar{z})\in$ $\mathbb{R}\times \mathbb{R}^{d^{\prime}}%
\times \mathbb{R}$,
\[%
\begin{array}
[c]{l}%
|f(t,\omega,y,z^{\prime})-f(t,\omega,\bar{y},\bar{z}^{\prime})|+\sum
_{i,j=1}^{d^{\prime}}|g_{ij}(t,\omega,y,z^{\prime})-g_{ij}(t,\omega,\bar
{y},\bar{z}^{\prime})|\\
+\sum_{l=d^{\prime}+1}^{d}|g_{l}(t,\omega,y,z^{\prime},z)-g_{l}(t,\omega
,\bar{y},\bar{z}^{\prime},\bar{z})|\leq L(|y-\bar{y}|+|z^{\prime}-\bar
{z}^{\prime}|+|z-\bar{z}|).
\end{array}
\]

\end{description}

Now we give the $L^{p}$-solution of $G$-BSDE (\ref{new-e2-4}) for $p\in
(1,\bar{p})$.

\begin{definition}
$(Y,Z^{1},\ldots,Z^{d},K)$ is called an $L^{p}$-solution of $G$-BSDE
(\ref{new-e2-4}) if the following properties hold:

\begin{description}
\item[(i)] $Y\in S_{G}^{p}(0,T)$, $Z^{i}\in H_{G}^{2,p}(0,T;\langle
B^{i}\rangle)$ for $i\leq d$, $K$ is a non-increasing $G$-martingale with
$K_{0}=0$ and $K_{T}\in L_{G}^{p}(\Omega_{T})$;

\item[(ii)]
\[%
\begin{array}
[c]{rl}%
Y_{t}= & \xi+\int_{t}^{T}f(s,Y_{s},Z_{s}^{\prime})ds+\sum_{i,j=1}^{d^{\prime}%
}\int_{t}^{T}g_{ij}(s,Y_{s},Z_{s}^{\prime})d\langle B^{i},B^{j}\rangle_{s}\\
& +\sum_{l=d^{\prime}+1}^{d}\int_{t}^{T}g_{l}(s,Y_{s},Z_{s}^{\prime},Z_{s}%
^{l})d\langle B^{l}\rangle_{s}-\sum_{k=1}^{d}\int_{t}^{T}Z_{s}^{k}dB_{s}%
^{k}-(K_{T}-K_{t}),
\end{array}
\]
where $Z_{s}^{\prime}=(Z_{s}^{1},\ldots,Z_{s}^{d^{\prime}})^{T}$ and $t\leq T$.
\end{description}
\end{definition}

For simplicity of representation, we only give the proof for the following
$G$-BSDE:%
\begin{equation}
Y_{t}=\xi+\int_{t}^{T}f(s,Y_{s})ds+\int_{t}^{T}g(s,Y_{s},Z_{s})d\langle
B\rangle_{s}-\int_{t}^{T}Z_{s}dB_{s}-(K_{T}-K_{t}), \label{e2-1}%
\end{equation}
where $B$ is a $1$-dimensional $G$-Brownian motion, $G(a):=\frac{1}{2}%
\bar{\sigma}^{2}a^{+}$ for $a\in \mathbb{R}$ with $\bar{\sigma}>0$. The results
still hold for $G$-BSDE (\ref{new-e2-4}), and will be given at the end of this
section. In the following, the constant $C$ will change from line to line for simplicity.

\subsection{Prior estimates of $G$-BSDEs}

In this subsection, we give some useful prior estimates of $G$-BSDE
(\ref{e2-1}).

\begin{proposition}
\label{pro2-1}Suppose that $\xi_{i}$, $f_{i}$ and $g_{i}$ satisfy (H1) and
(H2) for $i=1$, $2$. Let $(Y^{i},Z^{i},K^{i})$ be the $L^{p}$-solution of
$G$-BSDE (\ref{e2-1}) corresponding to $\xi_{i}$, $f_{i}$ and $g_{i}$ for some
$p\in(1,\bar{p})$. Then there exists a positive constant $C$ depending on $p$,
$\bar{\sigma}$, $L$ and $T$ satisfying%
\begin{equation}
|\hat{Y}_{t}|^{p}\leq C\hat{\mathbb{E}}_{t}\left[  |\hat{\xi}|^{p}+\left(
\int_{t}^{T}|\hat{f}_{s}|ds\right)  ^{p}+\left(  \int_{t}^{T}|\hat{g}%
_{s}|d\langle B\rangle_{s}\right)  ^{p}\right]  , \label{e2-2}%
\end{equation}%
\begin{equation}
|Y_{t}^{i}|^{p}\leq C\hat{\mathbb{E}}_{t}\left[  |\xi_{i}|^{p}+\left(
\int_{t}^{T}|f_{i}(s,0)|ds\right)  ^{p}+\left(  \int_{t}^{T}|g_{i}%
(s,0,0)|d\langle B\rangle_{s}\right)  ^{p}\right]  \text{ for }i=1,2,
\label{e2-6}%
\end{equation}%
\begin{equation}
\hat{\mathbb{E}}\left[  \left(  \int_{0}^{T}|Z_{s}^{i}|^{2}d\langle
B\rangle_{s}\right)  ^{p/2}\right]  +\hat{\mathbb{E}}\left[  |K_{T}^{i}%
|^{p}\right]  \leq C\Lambda_{i}\text{ for }i=1,2, \label{e2-7}%
\end{equation}%
\begin{equation}
\hat{\mathbb{E}}\left[  \left(  \int_{0}^{T}|\hat{Z}_{s}|^{2}d\langle
B\rangle_{s}\right)  ^{p/2}\right]  \leq C\left \{  \mathbb{\hat{E}}\left[
\sup_{t\leq T}|\hat{Y}_{t}|^{p}\right]  +(\Lambda_{1}+\Lambda_{2}%
)^{1/2}\left(  \mathbb{\hat{E}}\left[  \sup_{t\leq T}|\hat{Y}_{t}|^{p}\right]
\right)  ^{1/2}\right \}  , \label{e2-8}%
\end{equation}
where%
\[
\Lambda_{i}=\mathbb{\hat{E}}\left[  \sup_{t\leq T}|Y_{t}^{i}|^{p}\right]
+\mathbb{\hat{E}}\left[  \left(  \int_{0}^{T}|f_{i}(s,0)|ds\right)
^{p}\right]  +\mathbb{\hat{E}}\left[  \left(  \int_{0}^{T}|g_{i}%
(s,0,0)|d\langle B\rangle_{s}\right)  ^{p}\right]  \text{ for }i=1,2,
\]
$\hat{Y}_{t}=Y_{t}^{1}-Y_{t}^{2}$, $\hat{\xi}=\xi_{1}-\xi_{2}$, $\hat{f}%
_{s}=f_{1}(s,Y_{s}^{2})-f_{2}(s,Y_{s}^{2})$, $\hat{g}_{s}=g_{1}(s,Y_{s}%
^{2},Z_{s}^{2})-g_{2}(s,Y_{s}^{2},Z_{s}^{2})$, $\hat{Z}_{t}=Z_{t}^{1}%
-Z_{t}^{2}$.
\end{proposition}

\begin{proof}
The method is the same as that in \cite{HJPS1}. For convenience of the reader,
we sketch the proof.

For each given $t<T$, consider the following SDE for $r\in \lbrack t,T]$%
\[
X_{r}=\int_{t}^{r}(f_{1}(s,Y_{s}^{2}-X_{s})-f_{2}(s,Y_{s}^{2}))ds+\int_{t}%
^{r}(g_{1}(s,Y_{s}^{2}-X_{s},Z_{s}^{2})-g_{2}(s,Y_{s}^{2},Z_{s}^{2}))d\langle
B\rangle_{s}.
\]
Noting that $\langle B\rangle_{t+s}-\langle B\rangle_{t}\leq \bar{\sigma}^{2}s$
for any $t$, $s\geq0$, we obtain%
\[
|X_{r}|\leq \int_{t}^{T}|\hat{f}_{s}|ds+\int_{t}^{T}|\hat{g}_{s}|d\langle
B\rangle_{s}+L(1+\bar{\sigma}^{2})\int_{t}^{r}|X_{s}|ds\text{ for }r\in \lbrack
t,T].
\]
By the Gronwall inequality, we have%
\begin{equation}
|X_{T}|\leq C\left(  \int_{t}^{T}|\hat{f}_{s}|ds+\int_{t}^{T}|\hat{g}%
_{s}|d\langle B\rangle_{s}\right)  , \label{e2-4}%
\end{equation}
where $C$ depends on $\bar{\sigma}$, $L$ and $T$. For each $\varepsilon>0$,
noting that%
\[
p(|x|^{2}+\varepsilon)^{(p/2)-1}+p(p-2)(|x|^{2}+\varepsilon)^{(p/2)-2}%
|x|^{2}\geq p((p-1)\wedge1)(|x|^{2}+\varepsilon)^{(p/2)-1}%
\]
for $x\in \mathbb{R}$ and taking $\lambda=pL(1+\bar{\sigma}^{2})+pL^{2}%
\bar{\sigma}^{2}2^{-1}[(p-1)^{-1}\vee1]$, we get by applying It\^{o}'s formula
to $(|\hat{Y}_{r}+X_{r}|^{2}+\varepsilon)^{p/2}e^{\lambda r}$ on $[t,T]$ that%
\begin{equation}
(|\hat{Y}_{t}|^{2}+\varepsilon)^{p/2}e^{\lambda t}+M_{T}-M_{t}\leq(|\hat{\xi
}+X_{T}|^{2}+\varepsilon)^{p/2}e^{\lambda T}, \label{e2-5}%
\end{equation}
where $M_{T}-M_{t}=\int_{t}^{T}p(|\hat{Y}_{s}+X_{s}|^{2}+\varepsilon
)^{(p/2)-1}e^{\lambda s}[(\hat{Y}_{s}+X_{s})\hat{Z}_{s}dB_{s}+(\hat{Y}%
_{s}+X_{s})^{+}dK_{s}^{1}+(\hat{Y}_{s}+X_{s})^{-}dK_{s}^{2}]$. By Lemma 3.4 in
\cite{HJPS1}, we know that $\hat{\mathbb{E}}_{t}[M_{T}-M_{t}]=0$. Taking
$\hat{\mathbb{E}}_{t}$ on both sides of (\ref{e2-5}) and letting
$\varepsilon \downarrow0$, we get (\ref{e2-2}) by (\ref{e2-4}).

Taking $\xi_{j}=f_{j}=g_{j}=0$ for $j\not =i$, we have $(Y^{j},Z^{j},K^{j}%
)=0$. Thus we obtain (\ref{e2-6}) by (\ref{e2-2}).

Applying It\^{o}'s formula to $|Y_{t}^{i}|^{2}$ on $[0,T]$, by the B-D-G
inequality, we get%
\begin{equation}
\hat{\mathbb{E}}\left[  \left(  \int_{0}^{T}|Z_{s}^{i}|^{2}d\langle
B\rangle_{s}\right)  ^{p/2}\right]  \leq C\left \{  \Lambda_{i}+\left(
\mathbb{\hat{E}}\left[  \sup_{t\leq T}|Y_{t}^{i}|^{p}\right]  \right)
^{1/2}\left(  \hat{\mathbb{E}}\left[  |K_{T}^{i}|^{p}\right]  \right)
^{1/2}\right \}  , \label{e2-9}%
\end{equation}
where $C$ depends on $p$, $\bar{\sigma}$, $L$ and $T$. It follows from
$G$-BSDE (\ref{e2-1}) and the B-D-G inequality that%
\begin{equation}
\hat{\mathbb{E}}\left[  |K_{T}^{i}|^{p}\right]  \leq C\left \{  \Lambda
_{i}+\hat{\mathbb{E}}\left[  \left(  \int_{0}^{T}|Z_{s}^{i}|^{2}d\langle
B\rangle_{s}\right)  ^{p/2}\right]  \right \}  , \label{e2-10}%
\end{equation}
where $C$ depends on $p$, $\bar{\sigma}$, $L$ and $T$. Then we deduce
(\ref{e2-7}) by (\ref{e2-9}) and (\ref{e2-10}).

Applying It\^{o}'s formula to $|\hat{Y}_{t}|^{2}$ on $[0,T]$, by the B-D-G
inequality, we get%
\begin{equation}
\hat{\mathbb{E}}\left[  \left(  \int_{0}^{T}|\hat{Z}_{s}|^{2}d\langle
B\rangle_{s}\right)  ^{p/2}\right]  \leq C\left \{  \mathbb{\hat{E}}\left[
\sup_{t\leq T}|\hat{Y}_{t}|^{p}\right]  +(\tilde{\Lambda}_{1}+\tilde{\Lambda
}_{2})^{1/2}\left(  \mathbb{\hat{E}}\left[  \sup_{t\leq T}|\hat{Y}_{t}%
|^{p}\right]  \right)  ^{1/2}\right \}  , \label{e2-11}%
\end{equation}
where $C$ depends on $p$, $\bar{\sigma}$, $L$ and $T$,
\[
\tilde{\Lambda}_{i}=\Lambda_{i}+\hat{\mathbb{E}}\left[  \left(  \int_{0}%
^{T}|Z_{s}^{i}|^{2}d\langle B\rangle_{s}\right)  ^{p/2}\right]  +\hat
{\mathbb{E}}\left[  |K_{T}^{i}|^{p}\right]  \text{ for }i=1,2.\text{ }%
\]
Thus we obtain (\ref{e2-8}) by (\ref{e2-7}) and (\ref{e2-11}).
\end{proof}

\subsection{Solution in the extended $\tilde{G}$-expectation space}

Following \cite{HJPS1}, the key point to obtain the solution of $G$-BSDE
(\ref{e2-1}) is to study the following type of $G$-BSDE:%
\begin{equation}
Y_{t}=\varphi(B_{T})+\int_{t}^{T}h(Y_{s},Z_{s})d\langle B\rangle_{s}-\int
_{t}^{T}Z_{s}dB_{s}-(K_{T}-K_{t}), \label{e2-12}%
\end{equation}
where $\varphi \in C_{0}^{\infty}(\mathbb{R})$, $h\in C_{0}^{\infty}%
(\mathbb{R}^{2})$.

In order to obtain the solution of $G$-BSDE (\ref{e2-12}), we introduce the
extended $\tilde{G}$-expectation space. Set $\tilde{\Omega}_{T}=C_{0}%
([0,T];\mathbb{R}^{2})$ and the canonical process is denoted by $(B,\tilde
{B})$. For each $a_{11}$, $a_{12}$, $a_{22}\in \mathbb{R}$, define%
\[
\tilde{G}\left(  \left(
\begin{array}
[c]{cc}%
a_{11} & a_{12}\\
a_{12} & a_{22}%
\end{array}
\right)  \right)  =G(a_{11})+\frac{1}{2}a_{22}=\frac{1}{2}\sup_{\gamma
\in \tilde{\Sigma}}\mathrm{tr}\left[  \left(
\begin{array}
[c]{cc}%
a_{11} & a_{12}\\
a_{12} & a_{22}%
\end{array}
\right)  \gamma \right]  ,
\]
where
\[
\tilde{\Sigma}=\left \{  \left(
\begin{array}
[c]{cc}%
\sigma^{2} & 0\\
0 & 1
\end{array}
\right)  :\sigma \in \lbrack0,\bar{\sigma}]\right \}  .
\]
The $\tilde{G}$-expectation is denoted by $\mathbb{\tilde{E}}$, and the
related spaces are denoted by%
\[
Lip(\tilde{\Omega}_{t})\text{, }L_{\tilde{G}}^{p}(\tilde{\Omega}_{t})\text{,
}\tilde{M}^{0}(0,T)\text{, }M_{\tilde{G}}^{p,\bar{p}}(0,T),\text{ }%
H_{\tilde{G}}^{p,\bar{p}}(0,T;\langle B\rangle)\text{, }S_{\tilde{G}}%
^{p}(0,T)\text{.}%
\]

For each $\mathbf{a}=(a_{1},a_{2})^{T}\in \mathbb{R}^{2}$, by Proposition 3.1.5
in Peng \cite{P2019}, we know that $B^{\mathbf{a}}:=a_{1}B+a_{2}\tilde{B}$ is
a $G_{\mathbf{a}}$-Brownian motion, where $G_{\mathbf{a}}(b)=\frac{1}{2}%
[(\bar{\sigma}^{2}|a_{1}|^{2}+|a_{2}|^{2})b^{+}-|a_{2}|^{2}b^{-}]$ for
$b\in \mathbb{R}$. In particular, $B$ is a $G$-Browinian motion and $\tilde{B}$
is a classical Browinian motion. Thus $\mathbb{\tilde{E}}|_{Lip(\Omega_{T}%
)}=\mathbb{\hat{E}}$, which implies that the completion of $M^{0}(0,T)$ (resp.
$S^{0}(0,T)$) under the norm $||\cdot||_{H_{\tilde{G}}^{p,\bar{p}}(0,T;\langle
B\rangle)}$ (resp. $||\cdot||_{S_{\tilde{G}}^{p}(0,T)}$) is $H_{G}^{p,\bar{p}%
}(0,T;\langle B\rangle)$ (resp. $S_{G}^{p}(0,T)$). Similar to (\ref{new-e2-3}%
), we know that $\langle B,\tilde{B}\rangle_{t}=0$ and $\langle \tilde
{B}\rangle_{t}=t$ in the $\tilde{G}$-expectation space.

\begin{lemma}
\label{pro2-2}Let $\varphi \in C_{0}^{\infty}(\mathbb{R})$ and $h\in
C_{0}^{\infty}(\mathbb{R}^{2})$. Then, for each given $p>1$, $G$-BSDE
(\ref{e2-12}) has a unique $L^{p}$-solution $(Y,Z,K)$ in the extended
$\tilde{G}$-expectation space such that $Y\in S_{G}^{p}(0,T)$, $Z\in
H_{\tilde{G}}^{2,p}(0,T;\langle B\rangle)$ and $K_{T}\in L_{\tilde{G}}%
^{p}(\tilde{\Omega}_{T})$.
\end{lemma}

\begin{proof}
The uniqueness is due to (\ref{e2-2}) and (\ref{e2-8}) in Proposition
\ref{pro2-1}. The proof of existence is divided into two parts.

Part 1. The purpose of this part is to find a solution $(Y,Z,K)$ in the
extended $\tilde{G}$-expectation space such that $Y\in S_{\tilde{G}}^{p}(0,T)$
and $Z\in H_{\tilde{G}}^{2,p}(0,T;\langle B\rangle)$.

For each fixed $\varepsilon \in(0,\bar{\sigma})$, define%
\[
B_{t}^{\varepsilon}=B_{t}+\varepsilon \tilde{B}_{t}\text{ for }t\in
\lbrack0,T].
\]
Then $(B_{t}^{\varepsilon})_{t\in \lbrack0,T]}$ is the $G_{\varepsilon}%
$-Brownian motion under $\mathbb{\tilde{E}}$, where
\begin{equation}
G_{\varepsilon}(a)=\frac{1}{2}[(\bar{\sigma}^{2}+\varepsilon^{2}%
)a^{+}-\varepsilon^{2}a^{-}]\text{ for }a\in \mathbb{R}. \label{new-e2-10}%
\end{equation}
Let $u_{\varepsilon}$ be the viscosity solution of the following PDE%
\begin{equation}
\partial_{t}u+G_{\varepsilon}(\partial_{xx}^{2}u+2h(u,\partial_{x}u))=0\text{,
}u(T,x)=\varphi(x). \label{new-e2-11}%
\end{equation}
By Theorem 6.4.3 in Krylov \cite{Kr} (see also Theorem C.4.4 in Peng
\cite{P2019}), there exists a constant $\alpha \in(0,1)$ satisfying%
\begin{equation}
||u_{\varepsilon}||_{C^{1+\alpha/2,2+\alpha}([0,T-\delta]\times \mathbb{R}%
)}<\infty \text{ for any }\delta>0\text{.} \label{new-e2-12}%
\end{equation}

Applying It\^{o}'s formula to $u_{\varepsilon}(t,B_{t}^{\varepsilon})$ on
$[0,T-\delta]$, we obtain%
\begin{equation}
Y_{t}^{\varepsilon}=Y_{T-\delta}^{\varepsilon}+\int_{t}^{T-\delta}%
h(Y_{s}^{\varepsilon},Z_{s}^{\varepsilon})d\langle B^{\varepsilon}\rangle
_{s}-\int_{t}^{T-\delta}Z_{s}^{\varepsilon}dB_{s}^{\varepsilon}-(K_{T-\delta
}^{\varepsilon}-K_{t}^{\varepsilon}), \label{e2-13}%
\end{equation}
where $Y_{t}^{\varepsilon}=u_{\varepsilon}(t,B_{t}^{\varepsilon})$,
$Z_{t}^{\varepsilon}=\partial_{x}u_{\varepsilon}(t,B_{t}^{\varepsilon})$ and%
\[
K_{t}^{\varepsilon}=\int_{0}^{t}\frac{1}{2}\left[  \partial_{xx}%
^{2}u_{\varepsilon}(s,B_{s}^{\varepsilon})+2h(Y_{s}^{\varepsilon}%
,Z_{s}^{\varepsilon})\right]  d\langle B^{\varepsilon}\rangle_{s}-\int_{0}%
^{t}G_{\varepsilon}\left(  \partial_{xx}^{2}u_{\varepsilon}(s,B_{s}%
^{\varepsilon})+2h(Y_{s}^{\varepsilon},Z_{s}^{\varepsilon})\right)  ds.
\]
By Lemma 4.2.1 in Peng \cite{P2019}, we obtain that $K^{\varepsilon}$ is
non-increasing and $K_{t}^{\varepsilon}=\mathbb{\tilde{E}}_{t}[K_{T-\delta
}^{\varepsilon}]$ for $t\leq T-\delta$.

The same analysis as in the proof of inequality (4.3) in \cite{HJPS1}, we get
that there exists a positive constant $C$ depending on $\varphi$, $h$,
$\bar{\sigma}$ and $T$ such that%
\[
|u_{\varepsilon}(t_{1},x_{1})-u_{\varepsilon}(t_{2},x_{2})|\leq C(\sqrt
{|t_{1}-t_{2}|}+|x_{1}-x_{2}|)\text{ for }\varepsilon \in(0,\bar{\sigma
})\text{, }t_{1},t_{2}\leq T,\text{ }x_{1},x_{2}\in \mathbb{R}.
\]
From this we can easily deduce that $\mathbb{\tilde{E}}[|Y_{T-\delta
}^{\varepsilon}-\varphi(B_{T}^{\varepsilon})|^{2}]\rightarrow0$ as
$\delta \downarrow0$ and%
\begin{equation}
|u_{\varepsilon}(t,x)|\leq|\varphi(x)|+C\sqrt{T},\text{ }|\partial
_{x}u_{\varepsilon}(t,x)|\leq C\text{ for }\varepsilon \in(0,\bar{\sigma
})\text{, }t\leq T,\text{ }x\in \mathbb{R}. \label{e2-14}%
\end{equation}
Taking $\delta \downarrow0$ in (\ref{e2-13}), we obtain%
\begin{equation}
Y_{t}^{\varepsilon}=\varphi(B_{T}^{\varepsilon})+\int_{t}^{T}h(Y_{s}%
^{\varepsilon},Z_{s}^{\varepsilon})d\langle B^{\varepsilon}\rangle_{s}%
-\int_{t}^{T}Z_{s}^{\varepsilon}dB_{s}^{\varepsilon}-(K_{T}^{\varepsilon
}-K_{t}^{\varepsilon}), \label{e2-15}%
\end{equation}
where $Y^{\varepsilon}$ and $Z^{\varepsilon}$ are uniformly bounded for
$\varepsilon \in(0,\bar{\sigma})$ by (\ref{e2-14}).

For each given $\varepsilon$, $\varepsilon^{\prime}\in(0,\bar{\sigma})$, set%
\[
\hat{Y}_{t}^{\varepsilon,\varepsilon^{\prime}}=Y_{t}^{\varepsilon}%
-Y_{t}^{\varepsilon^{\prime}},\text{ }\hat{Z}_{t}^{\varepsilon,\varepsilon
^{\prime}}=Z_{t}^{\varepsilon}-Z_{t}^{\varepsilon^{\prime}},\text{ }\hat
{K}_{t}^{\varepsilon,\varepsilon^{\prime}}=K_{t}^{\varepsilon}-K_{t}%
^{\varepsilon^{\prime}},\text{ }\hat{\xi}^{\varepsilon,\varepsilon^{\prime}%
}=\varphi(B_{T}^{\varepsilon})-\varphi(B_{T}^{\varepsilon^{\prime}}),
\]%
\[
\hat{h}_{t}^{\varepsilon,\varepsilon^{\prime}}=h(Y_{t}^{\varepsilon}%
,Z_{t}^{\varepsilon})-h(Y_{t}^{\varepsilon^{\prime}},Z_{t}^{\varepsilon
^{\prime}})\text{, }\bar{h}_{t}^{\varepsilon,\varepsilon^{\prime}}%
=\varepsilon^{2}h(Y_{t}^{\varepsilon},Z_{t}^{\varepsilon})-(\varepsilon
^{\prime})^{2}h(Y_{t}^{\varepsilon^{\prime}},Z_{t}^{\varepsilon^{\prime}%
}),\text{ }\bar{Z}_{t}^{\varepsilon,\varepsilon^{\prime}}=\varepsilon
Z_{t}^{\varepsilon}-\varepsilon^{\prime}Z_{t}^{\varepsilon^{\prime}}.
\]
Then, by (\ref{e2-15}) and $\langle B^{\varepsilon}\rangle_{s}=\langle
B\rangle_{s}+\varepsilon^{2}s$, we have%
\[
\hat{Y}_{t}^{\varepsilon,\varepsilon^{\prime}}=\hat{\xi}^{\varepsilon
,\varepsilon^{\prime}}+\int_{t}^{T}\hat{h}_{s}^{\varepsilon,\varepsilon
^{\prime}}d\langle B\rangle_{s}+\int_{t}^{T}\bar{h}_{s}^{\varepsilon
,\varepsilon^{\prime}}ds-\int_{t}^{T}\hat{Z}_{s}^{\varepsilon,\varepsilon
^{\prime}}dB_{s}-\int_{t}^{T}\bar{Z}_{s}^{\varepsilon,\varepsilon^{\prime}%
}d\tilde{B}_{s}-(\hat{K}_{T}^{\varepsilon,\varepsilon^{\prime}}-\hat{K}%
_{t}^{\varepsilon,\varepsilon^{\prime}}).
\]
Applying It\^{o}'s formula to $|\hat{Y}_{s}^{\varepsilon,\varepsilon^{\prime}%
}|^{2}e^{\lambda s}$ on $[t,T]$ for some positive constant $\lambda$, we
obtain%
\begin{equation}%
\begin{array}
[c]{l}%
|\hat{Y}_{t}^{\varepsilon,\varepsilon^{\prime}}|^{2}e^{\lambda t}+\lambda
\int_{t}^{T}e^{\lambda s}|\hat{Y}_{s}^{\varepsilon,\varepsilon^{\prime}}%
|^{2}ds+\int_{t}^{T}e^{\lambda s}|\hat{Z}_{s}^{\varepsilon,\varepsilon
^{\prime}}|^{2}d\langle B\rangle_{s}+M_{T}-M_{t}\\
\leq|\hat{\xi}^{\varepsilon,\varepsilon^{\prime}}|^{2}e^{\lambda T}+2\int
_{t}^{T}e^{\lambda s}|\hat{Y}_{s}^{\varepsilon,\varepsilon^{\prime}}||\hat
{h}_{s}^{\varepsilon,\varepsilon^{\prime}}|d\langle B\rangle_{s}+2\int_{t}%
^{T}e^{\lambda s}|\hat{Y}_{s}^{\varepsilon,\varepsilon^{\prime}}||\bar{h}%
_{s}^{\varepsilon,\varepsilon^{\prime}}|ds,
\end{array}
\label{e2-16}%
\end{equation}
where%
\[
M_{T}-M_{t}=2\int_{t}^{T}e^{\lambda s}\hat{Y}_{s}^{\varepsilon,\varepsilon
^{\prime}}[\hat{Z}_{s}^{\varepsilon,\varepsilon^{\prime}}dB_{s}+\bar{Z}%
_{s}^{\varepsilon,\varepsilon^{\prime}}d\tilde{B}_{s}]+2\int_{t}^{T}e^{\lambda
s}[(\hat{Y}_{s}^{\varepsilon,\varepsilon^{\prime}})^{+}dK_{s}^{\varepsilon
}+(\hat{Y}_{s}^{\varepsilon,\varepsilon^{\prime}})^{-}dK_{s}^{\varepsilon
^{\prime}}].
\]
Since%
\[
2|\hat{Y}_{s}^{\varepsilon,\varepsilon^{\prime}}||\hat{h}_{s}^{\varepsilon
,\varepsilon^{\prime}}|\leq2L_{1}|\hat{Y}_{s}^{\varepsilon,\varepsilon
^{\prime}}|(|\hat{Y}_{s}^{\varepsilon,\varepsilon^{\prime}}|+|\hat{Z}%
_{s}^{\varepsilon,\varepsilon^{\prime}}|)\leq(|L_{1}|^{2}+2L_{1})|\hat{Y}%
_{s}^{\varepsilon,\varepsilon^{\prime}}|^{2}+|\hat{Z}_{s}^{\varepsilon
,\varepsilon^{\prime}}|^{2},
\]%
\[
2|\hat{Y}_{s}^{\varepsilon,\varepsilon^{\prime}}||\bar{h}_{s}^{\varepsilon
,\varepsilon^{\prime}}|\leq|\hat{Y}_{s}^{\varepsilon,\varepsilon^{\prime}%
}|^{2}+|\bar{h}_{s}^{\varepsilon,\varepsilon^{\prime}}|^{2}\leq|\hat{Y}%
_{s}^{\varepsilon,\varepsilon^{\prime}}|^{2}+2|L_{2}|^{2}(\varepsilon
^{4}+(\varepsilon^{\prime})^{4}),
\]
where $L_{1}=\sup_{(y,z)\in \mathbb{R}^{2}}(|\partial_{y}h(y,z)|+|\partial
_{z}h(y,z)|)$ and $L_{2}=\sup_{(y,z)\in \mathbb{R}^{2}}|h(y,z)|$, we get by
taking $\lambda=(|L_{1}|^{2}+2L_{1})\bar{\sigma}^{2}+1$ in (\ref{e2-16}) that%
\begin{equation}
|\hat{Y}_{t}^{\varepsilon,\varepsilon^{\prime}}|^{2}e^{\lambda t}+M_{T}%
-M_{t}\leq|\hat{\xi}^{\varepsilon,\varepsilon^{\prime}}|^{2}e^{\lambda
T}+2|L_{2}|^{2}(\varepsilon^{4}+(\varepsilon^{\prime})^{4})Te^{\lambda T}.
\label{e2-17}%
\end{equation}
By Lemma 3.4 in \cite{HJPS1}, we know that $\mathbb{\tilde{E}}_{t}[M_{T}%
-M_{t}]=0$. Taking $\mathbb{\tilde{E}}_{t}$ on both sides of (\ref{e2-17}), we
obtain%
\begin{align*}
|\hat{Y}_{t}^{\varepsilon,\varepsilon^{\prime}}|^{2}  &  \leq C\left(
\mathbb{\tilde{E}}_{t}[|\hat{\xi}^{\varepsilon,\varepsilon^{\prime}}%
|^{2}]+\varepsilon^{4}+(\varepsilon^{\prime})^{4}\right) \\
&  \leq C\left(  L_{\varphi}^{2}|\varepsilon-\varepsilon^{\prime}%
|^{2}\mathbb{\tilde{E}}_{t}[|\tilde{B}_{T}|^{2}]+\varepsilon^{4}%
+(\varepsilon^{\prime})^{4}\right)  ,
\end{align*}
where $L_{\varphi}=\sup_{x\in \mathbb{R}}|\varphi^{\prime}(x)|$ and $C$ depends
on $\bar{\sigma}$, $h$ and $T$. Thus, for each given $p>1$, we obtain%
\begin{equation}
\mathbb{\tilde{E}}\left[  \sup_{t\leq T}|\hat{Y}_{t}^{\varepsilon
,\varepsilon^{\prime}}|^{p}\right]  \leq C\left(  |\varepsilon-\varepsilon
^{\prime}|^{p}+\varepsilon^{2p}+(\varepsilon^{\prime})^{2p}\right)
\rightarrow0\text{ as }\varepsilon,\text{ }\varepsilon^{\prime}\rightarrow0,
\label{e2-18}%
\end{equation}
where $C$ depends on $p$, $\bar{\sigma}$, $\varphi$, $h$ and $T$. Applying
It\^{o}'s formula to $|\hat{Y}_{t}^{\varepsilon,\varepsilon^{\prime}}|^{2}$ on
$[0,T]$, we get%
\begin{equation}%
\begin{array}
[c]{rl}%
\int_{0}^{T}|\hat{Z}_{t}^{\varepsilon,\varepsilon^{\prime}}|^{2}d\langle
B\rangle_{t}\leq & |\hat{\xi}^{\varepsilon,\varepsilon^{\prime}}|^{2}%
+2\int_{0}^{T}|\hat{Y}_{t}^{\varepsilon,\varepsilon^{\prime}}||\hat{h}%
_{t}^{\varepsilon,\varepsilon^{\prime}}|d\langle B\rangle_{t}-2\int_{0}%
^{T}\hat{Y}_{t}^{\varepsilon,\varepsilon^{\prime}}\hat{Z}_{t}^{\varepsilon
,\varepsilon^{\prime}}dB_{t}-2\int_{0}^{T}\hat{Y}_{t}^{\varepsilon
,\varepsilon^{\prime}}\bar{Z}_{t}^{\varepsilon,\varepsilon^{\prime}}d\tilde
{B}_{t}\\
& +2\int_{0}^{T}|\hat{Y}_{t}^{\varepsilon,\varepsilon^{\prime}}||\bar{h}%
_{t}^{\varepsilon,\varepsilon^{\prime}}|dt+2(|K_{T}^{\varepsilon}%
|+|K_{T}^{\varepsilon^{\prime}}|)\sup_{t\leq T}|\hat{Y}_{t}^{\varepsilon
,\varepsilon^{\prime}}|.
\end{array}
\label{e2-19}%
\end{equation}
By (\ref{e2-14}), (\ref{e2-15}), (\ref{e2-18}) and (\ref{e2-19}), we obtain%
\begin{equation}
\mathbb{\tilde{E}}\left[  \int_{0}^{T}|\hat{Z}_{t}^{\varepsilon,\varepsilon
^{\prime}}|^{2}d\langle B\rangle_{t}\right]  \leq C\left \{  \mathbb{\tilde{E}%
}\left[  \sup_{t\leq T}|\hat{Y}_{t}^{\varepsilon,\varepsilon^{\prime}}%
|^{2}\right]  +\left(  \mathbb{\tilde{E}}\left[  \sup_{t\leq T}|\hat{Y}%
_{t}^{\varepsilon,\varepsilon^{\prime}}|^{2}\right]  \right)  ^{1/2}\right \}
\rightarrow0\text{ as }\varepsilon,\text{ }\varepsilon^{\prime}\rightarrow0,
\label{e2-20}%
\end{equation}
where $C$ depends on $\bar{\sigma}$, $\varphi$, $h$ and $T$. Since
$Z^{\varepsilon}$ is uniformly bounded for $\varepsilon \in(0,\bar{\sigma})$,
we deduce from (\ref{e2-20}) that, for each given $p>1$,%
\begin{equation}
\mathbb{\tilde{E}}\left[  \left(  \int_{0}^{T}|\hat{Z}_{t}^{\varepsilon
,\varepsilon^{\prime}}|^{2}d\langle B\rangle_{t}\right)  ^{p/2}\right]
\rightarrow0\text{ as }\varepsilon,\text{ }\varepsilon^{\prime}\rightarrow0.
\label{e2-21}%
\end{equation}
Thus, for each given $p>1$, there exist $Y\in S_{\tilde{G}}^{p}(0,T)$ and
$Z\in H_{\tilde{G}}^{2,p}(0,T;\langle B\rangle)$ such that%
\begin{equation}
\mathbb{\tilde{E}}\left[  \sup_{t\leq T}|Y_{t}^{\varepsilon}-Y_{t}%
|^{p}+\left(  \int_{0}^{T}|Z_{t}^{\varepsilon}-Z_{t}|^{2}d\langle B\rangle
_{t}\right)  ^{p/2}\right]  \rightarrow0\text{ as }\varepsilon \rightarrow0.
\label{e2-22}%
\end{equation}
It follows from (\ref{e2-15}) and (\ref{e2-22}) that there exists a $K_{T}\in
L_{\tilde{G}}^{p}(\tilde{\Omega}_{T})$ such that $\mathbb{\tilde{E}}\left[
|K_{T}^{\varepsilon}-K_{T}|^{p}\right]  \rightarrow0$ as $\varepsilon
\rightarrow0$. Taking $\varepsilon \rightarrow0$ in (\ref{e2-15}), we obtain%
\begin{equation}
Y_{t}=\varphi(B_{T})+\int_{t}^{T}h(Y_{s},Z_{s})d\langle B\rangle_{s}-\int
_{t}^{T}Z_{s}dB_{s}-(K_{T}-K_{t}), \label{e2-23}%
\end{equation}
where $K$ is non-increasing and $K_{t}=\mathbb{\tilde{E}}_{t}[K_{T}]$ for
$t\leq T$.

Part 2. The purpose of this part is to prove that $Y\in S_{G}^{p}(0,T)$ for
each $p>1$.

Noting that $Y_{t}^{\varepsilon}=u_{\varepsilon}(t,B_{t}^{\varepsilon})$ and
(\ref{e2-14}), we have%
\[
\mathbb{\tilde{E}}\left[  \sup_{t\leq T}|Y_{t}^{\varepsilon}-u_{\varepsilon
}(t,B_{t})|^{p}\right]  \leq C\varepsilon^{p}\mathbb{\tilde{E}}\left[
\sup_{t\leq T}|\tilde{B}_{t}|^{p}\right]  \rightarrow0\text{ as }%
\varepsilon \rightarrow0,
\]
which implies%
\begin{equation}
\mathbb{\tilde{E}}\left[  \sup_{t\leq T}|u_{\varepsilon}(t,B_{t})-Y_{t}%
|^{p}\right]  \rightarrow0\text{ as }\varepsilon \rightarrow0. \label{e2-24}%
\end{equation}
Thus $Y\in S_{G}^{p}(0,T)$.
\end{proof}

\subsection{Estimates of partial derivatives of $u_{\varepsilon}$}

In order to show that $Z$ obtained in Lemma \ref{pro2-2} belongs to
$H_{G}^{2,p}(0,T;\langle B\rangle)$, we need to prove that $\partial_{xx}%
^{2}u_{\varepsilon}$ is uniformly bounded from below for $\varepsilon
\in(0,\bar{\sigma})$, where $u_{\varepsilon}$ is the solution of PDE
(\ref{new-e2-11}).

For each fixed $\varepsilon \in(0,\bar{\sigma})$, $G_{\varepsilon}$ is defined
in (\ref{new-e2-10}). Let $\hat{\mathbb{E}}^{\varepsilon}$ be the
$G_{\varepsilon}$-expectation on $(\Omega_{T},Lip(\Omega_{T}))$. The canonical
process $(B_{t})_{t\in \lbrack0,T]}$ is the $1$-dimensional $G_{\varepsilon}%
$-Brownian motion under $\hat{\mathbb{E}}^{\varepsilon}$. For each given
$(t,x)\in \lbrack0,T)\times \mathbb{R}$, denote
\[
B_{s}^{t,x}=x+B_{s}-B_{t}\text{ for }s\in \lbrack t,T].
\]
Similar to (\ref{e2-15}), applying It\^{o}'s formula to $u_{\varepsilon
}(s,B_{s}^{t,x})$ under $\hat{\mathbb{E}}^{\varepsilon}$, we obtain that the
following $G_{\varepsilon}$-BSDE
\begin{equation}
Y_{s}^{t,x}=\varphi(B_{T}^{t,x})+\int_{s}^{T}h(Y_{r}^{t,x},Z_{r}%
^{t,x})d\langle B\rangle_{r}-\int_{s}^{T}Z_{r}^{t,x}dB_{r}-(K_{T}^{t,x}%
-K_{s}^{t,x}) \label{e2-26}%
\end{equation}
has a unique solution $(Y_{s}^{t,x},Z_{s}^{t,x},K_{s}^{t,x})_{s\in \lbrack
t,T]}$ satisfying $Y_{s}^{t,x}=u_{\varepsilon}(s,B_{s}^{t,x})$, $Z_{s}%
^{t,x}=\partial_{x}u_{\varepsilon}(t,B_{s}^{t,x})$ and $K_{t}^{t,x}=0$.

Let $\mathcal{P}^{\varepsilon}$ be a weakly compact and convex set of
probability measures on $(\Omega_{T},\mathcal{B}(\Omega_{T}))$ such that%
\[
\hat{\mathbb{E}}^{\varepsilon}[X]=\sup_{P\in \mathcal{P}^{\varepsilon}}%
E_{P}[X]\text{ for all }X\in L_{G_{\varepsilon}}^{1}(\Omega_{T}).
\]
For each given $(t,x)\in \lbrack0,T)\times \mathbb{R}$, denote%
\[
\mathcal{P}_{t,x}^{\varepsilon}=\{P\in \mathcal{P}^{\varepsilon}:E_{P}%
[K_{T}^{t,x}]=0\}.
\]

The following estimates for $G_{\varepsilon}$-BSDE (\ref{e2-26}) are useful.

\begin{proposition}
\label{pro2-3} Suppose that $\varphi \in C_{0}^{\infty}(\mathbb{R})$ and $h\in
C_{0}^{\infty}(\mathbb{R}^{2})$. For each $(t,x,\Delta)\in \lbrack
0,T)\times \mathbb{R}\times \mathbb{R}$, let $(Y_{s}^{t,x},Z_{s}^{t,x}%
,K_{s}^{t,x})_{s\in \lbrack t,T]}$ and $(Y_{s}^{t,x+\Delta},Z_{s}^{t,x+\Delta
},K_{s}^{t,x+\Delta})_{s\in \lbrack t,T]}$ be two solutions of $G_{\varepsilon
}$-BSDE (\ref{e2-26}). Then, for each given $p>1$,
\begin{equation}
\sup_{s\in \lbrack t,T]}\left \vert Y_{s}^{t,x+\Delta}-Y_{s}^{t,x}\right \vert
^{p}\leq C|\Delta|^{p}, \label{e2-27}%
\end{equation}%
\begin{equation}
\hat{\mathbb{E}}^{\varepsilon}\left[  \sup_{s\in \lbrack t,T]}\left \vert
Y_{s}^{t,x}\right \vert ^{p}+\left(  \int_{t}^{T}\left \vert Z_{s}%
^{t,x}\right \vert ^{2}d\langle B\rangle_{s}\right)  ^{p/2}+\left \vert
K_{T}^{t,x}\right \vert ^{p}\right]  \leq C(1+|x|^{p}), \label{e2-28}%
\end{equation}%
\begin{equation}
E_{P}\left[  \left(  \int_{t}^{T}\left \vert Z_{s}^{t,x+\Delta}-Z_{s}%
^{t,x}\right \vert ^{2}d\langle B\rangle_{s}\right)  ^{p/2}+\left \vert
K_{T}^{t,x+\Delta}\right \vert ^{p}\right]  \leq C|\Delta|^{p}\text{ for }%
P\in \mathcal{P}_{t,x}^{\varepsilon}, \label{e2-29}%
\end{equation}%
\begin{equation}
E_{P^{\Delta}}\left[  \left(  \int_{t}^{T}\left \vert Z_{s}^{t,x+\Delta}%
-Z_{s}^{t,x}\right \vert ^{2}d\langle B\rangle_{s}\right)  ^{p/2}+\left \vert
K_{T}^{t,x}\right \vert ^{p}\right]  \leq C|\Delta|^{p}\text{ for }P^{\Delta
}\in \mathcal{P}_{t,x+\Delta}^{\varepsilon}, \label{e2-30}%
\end{equation}
where the constant $C>0$ depends on $p$, $\bar{\sigma}$, $\varphi$, $h$ and
$T$.
\end{proposition}

\begin{proof}
Similar to the proof of (\ref{e2-2}), (\ref{e2-6}) and (\ref{e2-7}), we obtain%
\[
\sup_{s\in \lbrack t,T]}\left \vert Y_{s}^{t,x+\Delta}-Y_{s}^{t,x}\right \vert
^{p}\leq C\sup_{s\in \lbrack t,T]}\hat{\mathbb{E}}_{s}^{\varepsilon}\left[
\left \vert \varphi(B_{T}^{t,x+\Delta})-\varphi(B_{T}^{t,x})\right \vert
^{p}\right]  \leq C|\Delta|^{p}%
\]
and%
\begin{align*}
&  \hat{\mathbb{E}}^{\varepsilon}\left[  \sup_{s\in \lbrack t,T]}\left \vert
Y_{s}^{t,x}\right \vert ^{p}+\left(  \int_{t}^{T}\left \vert Z_{s}%
^{t,x}\right \vert ^{2}d\langle B\rangle_{s}\right)  ^{p/2}+\left \vert
K_{T}^{t,x}\right \vert ^{p}\right] \\
&  \leq C\left(  1+\hat{\mathbb{E}}^{\varepsilon}\left[  \sup_{s\in \lbrack
t,T]}\hat{\mathbb{E}}_{s}^{\varepsilon}\left[  \left \vert \varphi(B_{T}%
^{t,x})\right \vert ^{p}\right]  \right]  \right) \\
&  \leq C\left(  1+|x|^{p}+\hat{\mathbb{E}}^{\varepsilon}\left[  \sup
_{s\in \lbrack t,T]}\hat{\mathbb{E}}_{s}^{\varepsilon}\left[  \left \vert
B_{T}-B_{t}\right \vert ^{p}\right]  \right]  \right) \\
&  \leq C(1+|x|^{p}),
\end{align*}
where the constant $C>0$ depends on $p$, $\bar{\sigma}$, $\varphi$, $h$ and
$T$.

Set $\hat{Y}_{s}^{\Delta}=Y_{s}^{t,x+\Delta}-Y_{s}^{t,x}$ and $\hat{Z}%
_{s}^{\Delta}=Z_{s}^{t,x+\Delta}-Z_{s}^{t,x}$ for $s\in \lbrack t,T]$. For each
given $P\in \mathcal{P}_{t,x}^{\varepsilon}$, we know that $K^{t,x}=0$ $P$-a.s.
by $E_{P}[K_{T}^{t,x}]=0$. Applying It\^{o}'s formula to $|\hat{Y}_{s}%
^{\Delta}|^{2}$ on $[t,T]$ under $P$, we obtain%
\begin{equation}
|\hat{Y}_{t}^{\Delta}|^{2}+\int_{t}^{T}|\hat{Z}_{r}^{\Delta}|^{2}d\langle
B\rangle_{r}=|\hat{Y}_{T}^{\Delta}|^{2}+2\int_{t}^{T}\hat{Y}_{r}^{\Delta}%
\hat{h}_{r}d\langle B\rangle_{r}-2\int_{t}^{T}\hat{Y}_{r}^{\Delta}\hat{Z}%
_{r}^{\Delta}dB_{r}-2\int_{t}^{T}\hat{Y}_{r}^{\Delta}dK_{r}^{t,x+\Delta},
\label{e2-31}%
\end{equation}
where%
\begin{equation}
|\hat{h}_{r}|=|h(Y_{r}^{t,x+\Delta},Z_{r}^{t,x+\Delta})-h(Y_{r}^{t,x}%
,Z_{r}^{t,x})|\leq \sup_{(y,z)\in \mathbb{R}^{2}}(|h_{y}^{\prime}(y,z)|+|h_{z}%
^{\prime}(y,z)|)(|\hat{Y}_{r}^{\Delta}|+|\hat{Z}_{r}^{\Delta}|). \label{e2-32}%
\end{equation}
Since $K^{t,x+\Delta}$ is non-increasing with $K_{t}^{t,x+\Delta}=0$ and
$d\langle B\rangle_{r}\leq(\bar{\sigma}^{2}+\varepsilon^{2})dr\leq2\bar
{\sigma}^{2}dr$ under $P$, we deduce by (\ref{e2-31}) and (\ref{e2-32}) that%
\begin{equation}
E_{P}\left[  \left(  \int_{t}^{T}|\hat{Z}_{r}^{\Delta}|^{2}d\langle
B\rangle_{r}\right)  ^{p/2}\right]  \leq CE_{P}\left[  \sup_{r\in \lbrack
t,T]}\left \vert \hat{Y}_{r}^{\Delta}\right \vert ^{p}+\left(  \sup_{r\in \lbrack
t,T]}\left \vert \hat{Y}_{r}^{\Delta}\right \vert ^{p/2}\right)  \left \vert
K_{T}^{t,x+\Delta}\right \vert ^{p/2}\right]  , \label{e2-33}%
\end{equation}
where $C>0$ depends on $p$, $\bar{\sigma}$, $h$ and $T$. Noting that%
\[
K_{T}^{t,x+\Delta}=\hat{Y}_{T}^{\Delta}-\hat{Y}_{t}^{\Delta}+\int_{t}^{T}%
\hat{h}_{r}d\langle B\rangle_{r}-\int_{t}^{T}\hat{Z}_{r}^{\Delta}dB_{r},\text{
}P\text{-a.s.,}%
\]
we get%
\begin{equation}
E_{P}\left[  \left \vert K_{T}^{t,x+\Delta}\right \vert ^{p}\right]  \leq
CE_{P}\left[  \sup_{r\in \lbrack t,T]}\left \vert \hat{Y}_{r}^{\Delta
}\right \vert ^{p}+\left(  \int_{t}^{T}|\hat{Z}_{r}^{\Delta}|^{2}d\langle
B\rangle_{r}\right)  ^{p/2}\right]  , \label{e2-34}%
\end{equation}
where $C>0$ depends on $p$, $\bar{\sigma}$, $h$ and $T$. Thus we obtain by
(\ref{e2-33}) and (\ref{e2-34}) that%
\begin{equation}
E_{P}\left[  \left(  \int_{t}^{T}|\hat{Z}_{r}^{\Delta}|^{2}d\langle
B\rangle_{r}\right)  ^{p/2}+\left \vert K_{T}^{t,x+\Delta}\right \vert
^{p}\right]  \leq CE_{P}\left[  \sup_{r\in \lbrack t,T]}\left \vert \hat{Y}%
_{r}^{\Delta}\right \vert ^{p}\right]  , \label{e2-35}%
\end{equation}
where $C>0$ depends on $p$, $\bar{\sigma}$, $h$ and $T$. By (\ref{e2-27}) and
(\ref{e2-35}), we obtain (\ref{e2-29}). By the same method, we obtain
(\ref{e2-30}).
\end{proof}

In the following theorem, we obtain the formula of $\partial_{x}%
u_{\varepsilon}$ based on $u_{\varepsilon}(t,x)=Y_{t}^{t,x}$.

\begin{theorem}
\label{th2-4}Suppose that $\varphi \in C_{0}^{\infty}(\mathbb{R})$ and $h\in
C_{0}^{\infty}(\mathbb{R}^{2})$. Let $u_{\varepsilon}$ be the solution of PDE
(\ref{new-e2-11}). Then, for each $(t,x)\in \lbrack0,T)\times \mathbb{R}$, we
have%
\begin{equation}
\partial_{x}u_{\varepsilon}(t,x)=E_{P}\left[  \Gamma_{T}^{t,x}\varphi^{\prime
}(B_{T}^{t,x})\right]  \text{ for any }P\in \mathcal{P}_{t,x}^{\varepsilon},
\label{e2-36}%
\end{equation}
where $(\Gamma_{s}^{t,x})_{s\in \lbrack t,T]}$ is the solution of the following
$G$-SDE:%
\begin{equation}
d\Gamma_{s}^{t,x}=h_{y}^{\prime}(Y_{s}^{t,x},Z_{s}^{t,x})\Gamma_{s}%
^{t,x}d\langle B\rangle_{s}+h_{z}^{\prime}(Y_{s}^{t,x},Z_{s}^{t,x})\Gamma
_{s}^{t,x}dB_{s},\text{ }\Gamma_{t}^{t,x}=1. \label{e2-37}%
\end{equation}

\end{theorem}

\begin{proof}
For each $\Delta \in \mathbb{R}$, we use the notations $(\hat{Y}_{s}^{\Delta
})_{s\in \lbrack t,T]}$ and $(\hat{Z}_{s}^{\Delta})_{s\in \lbrack t,T]}$ as in
the proof of Proposition \ref{pro2-3}. Then, for any given $P\in
\mathcal{P}_{t,x}^{\varepsilon}$, we have%
\[
\hat{Y}_{s}^{\Delta}=\hat{Y}_{T}^{\Delta}+\int_{s}^{T}\hat{h}_{r}d\langle
B\rangle_{r}-\int_{s}^{T}\hat{Z}_{r}^{\Delta}dB_{r}-\int_{s}^{T}%
dK_{r}^{t,x+\Delta},\text{ }P\text{-a.s.,}%
\]
where%
\begin{align*}
\hat{h}_{r}  &  =h(Y_{r}^{t,x+\Delta},Z_{r}^{t,x+\Delta})-h(Y_{r}^{t,x}%
,Z_{r}^{t,x})\\
&  =h_{y}^{\prime}(Y_{r}^{t,x},Z_{r}^{t,x})\hat{Y}_{r}^{\Delta}+h_{z}^{\prime
}(Y_{r}^{t,x},Z_{r}^{t,x})\hat{Z}_{r}^{\Delta}+I_{r}^{\Delta}.
\end{align*}
Since $h\in C_{0}^{\infty}(\mathbb{R}^{2})$, we get $|I_{r}^{\Delta}|\leq
C(|\hat{Y}_{r}^{\Delta}|^{2}+|\hat{Z}_{r}^{\Delta}|^{2})$, where $C>0$ depends
on $h$. Applying It\^{o}'s formula to $\hat{Y}_{s}^{\Delta}\Gamma_{s}^{t,x}$
on $[t,T]$ under $P$, we obtain%
\begin{equation}
\hat{Y}_{t}^{\Delta}=\hat{Y}_{T}^{\Delta}\Gamma_{T}^{t,x}+\int_{t}^{T}%
\Gamma_{r}^{t,x}I_{r}^{\Delta}d\langle B\rangle_{r}-\int_{t}^{T}(\Gamma
_{r}^{t,x}\hat{Z}_{r}^{\Delta}+h_{z}^{\prime}(Y_{r}^{t,x},Z_{r}^{t,x}%
)\Gamma_{r}^{t,x}\hat{Y}_{r}^{\Delta})dB_{r}-\int_{t}^{T}\Gamma_{r}%
^{t,x}dK_{r}^{t,x+\Delta}. \label{e2-38}%
\end{equation}
Noting that $\hat{Y}_{t}^{\Delta}=u_{\varepsilon}(t,x+\Delta)-u_{\varepsilon
}(t,x)$, we get%
\begin{equation}
\frac{u_{\varepsilon}(t,x+\Delta)-u_{\varepsilon}(t,x)}{\Delta}=\frac
{1}{\Delta}E_{P}\left[  \hat{Y}_{T}^{\Delta}\Gamma_{T}^{t,x}+\int_{t}%
^{T}\Gamma_{r}^{t,x}I_{r}^{\Delta}d\langle B\rangle_{r}-\int_{t}^{T}\Gamma
_{r}^{t,x}dK_{r}^{t,x+\Delta}\right]  . \label{e2-39}%
\end{equation}
By (\ref{e2-27}), (\ref{e2-29}), $\varphi \in C_{0}^{\infty}(\mathbb{R})$ and
$h\in C_{0}^{\infty}(\mathbb{R}^{2})$, we can easily deduce that%
\begin{equation}
\lim_{\Delta \rightarrow0}\frac{1}{\Delta}E_{P}\left[  \hat{Y}_{T}^{\Delta
}\Gamma_{T}^{t,x}+\int_{t}^{T}\Gamma_{r}^{t,x}I_{r}^{\Delta}d\langle
B\rangle_{r}\right]  =E_{P}\left[  \Gamma_{T}^{t,x}\varphi^{\prime}%
(B_{T}^{t,x})\right]  . \label{e2-40}%
\end{equation}
Since $\Gamma_{r}^{t,x}>0$, $dK_{r}^{t,x+\Delta}\leq0$ and $\partial
_{x}u_{\varepsilon}(t,x)$ exists, we obtain by (\ref{e2-39}) and (\ref{e2-40})
that%
\[
E_{P}\left[  \Gamma_{T}^{t,x}\varphi^{\prime}(B_{T}^{t,x})\right]
\leq \partial_{x+}u_{\varepsilon}(t,x)=\partial_{x}u_{\varepsilon
}(t,x)=\partial_{x-}u_{\varepsilon}(t,x)\leq E_{P}\left[  \Gamma_{T}%
^{t,x}\varphi^{\prime}(B_{T}^{t,x})\right]  ,
\]
which implies the desired result.
\end{proof}

Now we give the estimate for $\partial_{xx}^{2}u_{\varepsilon}$.

\begin{theorem}
\label{th2-5}Suppose that $\varphi \in C_{0}^{\infty}(\mathbb{R})$ and $h\in
C_{0}^{\infty}(\mathbb{R}^{2})$. Let $u_{\varepsilon}$ be the solution of PDE
(\ref{new-e2-11}). Then%
\[
\partial_{xx}^{2}u_{\varepsilon}(t,x)\geq-C\text{ for }(t,x)\in \lbrack
0,T)\times \mathbb{R},
\]
where the constant $C>0$ depends on $\bar{\sigma}$, $\varphi$, $h$ and $T$.
\end{theorem}

\begin{proof}
For each $(t,x,\Delta)\in \lbrack0,T)\times \mathbb{R}\times \mathbb{R}$, we use
the notations $(\hat{Y}_{s}^{\Delta})_{s\in \lbrack t,T]}$ and $(\hat{Z}%
_{s}^{\Delta})_{s\in \lbrack t,T]}$ as in the proof of Proposition
\ref{pro2-3}. For any given $P\in \mathcal{P}_{t,x}^{\varepsilon}$, we obtain
by (\ref{e2-38}) that%
\[
\hat{Y}_{t}^{\Delta}=E_{P}\left[  \hat{Y}_{T}^{\Delta}\Gamma_{T}^{t,x}%
+\int_{t}^{T}\Gamma_{r}^{t,x}I_{r}^{\Delta}d\langle B\rangle_{r}-\int_{t}%
^{T}\Gamma_{r}^{t,x}dK_{r}^{t,x+\Delta}\right]  .
\]
Since $\Gamma_{r}^{t,x}>0$ and $dK_{r}^{t,x+\Delta}\leq0$, we get%
\[
\hat{Y}_{t}^{\Delta}\geq E_{P}\left[  \hat{Y}_{T}^{\Delta}\Gamma_{T}%
^{t,x}+\int_{t}^{T}\Gamma_{r}^{t,x}I_{r}^{\Delta}d\langle B\rangle_{r}\right]
.
\]
Noting that $|\hat{Y}_{T}^{\Delta}-\varphi^{\prime}(B_{T}^{t,x})\Delta|\leq
C\Delta^{2}$ and $|I_{r}^{\Delta}|\leq C(|\hat{Y}_{r}^{\Delta}|^{2}+|\hat
{Z}_{r}^{\Delta}|^{2})$, where $C>0$ depends on $\varphi$ and $h$, we deduce
by Proposition \ref{pro2-3} that%
\begin{equation}
\hat{Y}_{t}^{\Delta}\geq E_{P}\left[  \hat{Y}_{T}^{\Delta}\Gamma_{T}%
^{t,x}+\int_{t}^{T}\Gamma_{r}^{t,x}I_{r}^{\Delta}d\langle B\rangle_{r}\right]
\geq E_{P}\left[  \Gamma_{T}^{t,x}\varphi^{\prime}(B_{T}^{t,x})\right]
\Delta-C\Delta^{2}, \label{e2-43}%
\end{equation}
where $C>0$ depends on $\bar{\sigma}$, $\varphi$, $h$ and $T$. For any given
$P^{\Delta}\in \mathcal{P}_{t,x+\Delta}^{\varepsilon}$, applying It\^{o}'s
formula to $\hat{Y}_{s}^{\Delta}\Gamma_{s}^{t,x+\Delta}$ on $[t,T]$ under
$P^{\Delta}$, we obtain%
\[
\hat{Y}_{t}^{\Delta}=E_{P^{\Delta}}\left[  \hat{Y}_{T}^{\Delta}\Gamma
_{T}^{t,x+\Delta}+\int_{t}^{T}\Gamma_{r}^{t,x+\Delta}\tilde{I}_{r}^{\Delta
}d\langle B\rangle_{r}+\int_{t}^{T}\Gamma_{r}^{t,x+\Delta}dK_{r}^{t,x}\right]
,
\]
where $\tilde{I}_{r}^{\Delta}=h(Y_{r}^{t,x+\Delta},Z_{r}^{t,x+\Delta}%
)-h(Y_{r}^{t,x},Z_{r}^{t,x})-h_{y}^{\prime}(Y_{r}^{t,x+\Delta},Z_{r}%
^{t,x+\Delta})\hat{Y}_{r}^{\Delta}-h_{z}^{\prime}(Y_{r}^{t,x+\Delta}%
,Z_{r}^{t,x+\Delta})\hat{Z}_{r}^{\Delta}$. Since $\Gamma_{r}^{t,x+\Delta}>0$
and $dK_{r}^{t,x}\leq0$, we get%
\[
\hat{Y}_{t}^{\Delta}\leq E_{P^{\Delta}}\left[  \hat{Y}_{T}^{\Delta}\Gamma
_{T}^{t,x+\Delta}+\int_{t}^{T}\Gamma_{r}^{t,x+\Delta}\tilde{I}_{r}^{\Delta
}d\langle B\rangle_{r}\right]  .
\]
Similar to the proof of (\ref{e2-43}), we have%
\begin{equation}
\hat{Y}_{t}^{\Delta}\leq E_{P^{\Delta}}\left[  \hat{Y}_{T}^{\Delta}\Gamma
_{T}^{t,x+\Delta}+\int_{t}^{T}\Gamma_{r}^{t,x+\Delta}\tilde{I}_{r}^{\Delta
}d\langle B\rangle_{r}\right]  \leq E_{P^{\Delta}}\left[  \Gamma
_{T}^{t,x+\Delta}\varphi^{\prime}(B_{T}^{t,x+\Delta})\right]  \Delta
+C\Delta^{2}, \label{e2-44}%
\end{equation}
where $C>0$ depends on $\bar{\sigma}$, $\varphi$, $h$ and $T$. By Theorem
\ref{th2-4}, (\ref{e2-43}) and (\ref{e2-44}), we obtain%
\[
\frac{\partial_{x}u_{\varepsilon}(t,x+\Delta)-\partial_{x}u_{\varepsilon
}(t,x)}{\Delta}=\frac{1}{\Delta^{2}}\left \{  E_{P^{\Delta}}\left[  \Gamma
_{T}^{t,x+\Delta}\varphi^{\prime}(B_{T}^{t,x+\Delta})\right]  \Delta
-E_{P}\left[  \Gamma_{T}^{t,x}\varphi^{\prime}(B_{T}^{t,x})\right]
\Delta \right \}  \geq-C,
\]
which implies the desired result.
\end{proof}

\begin{remark}
The constant $C$ in the above theorem is independent of $\varepsilon \in
(0,\bar{\sigma})$.
\end{remark}

\subsection{Existence and uniqueness of $G$-BSDEs}

We first give the following existence and uniqueness result of $G$-BSDE
(\ref{e2-12}).

\begin{lemma}
\label{pro2-6}Let $\varphi \in C_{0}^{\infty}(\mathbb{R})$ and $h\in
C_{0}^{\infty}(\mathbb{R}^{2})$. Then, for each given $p>1$, $G$-BSDE
(\ref{e2-12}) has a unique $L^{p}$-solution $(Y,Z,K)$ in the $G$-expectation space.
\end{lemma}

\begin{proof}
The uniqueness is due to (\ref{e2-2}) and (\ref{e2-8}) in Proposition
\ref{pro2-1}. In the following, we give the proof of existence.

By Lemma \ref{pro2-2}, for each given $p>1$, $G$-BSDE (\ref{e2-12}) has a
unique $L^{p}$-solution $(Y,Z,K)$ in the extended $\tilde{G}$-expectation
space, i.e.,
\begin{equation}
Y_{t}=\varphi(B_{T})+\int_{t}^{T}h(Y_{s},Z_{s})d\langle B\rangle_{s}-\int
_{t}^{T}Z_{s}dB_{s}-(K_{T}-K_{t}). \label{e2-45}%
\end{equation}
Let $u_{\varepsilon}$ be the solution of PDE (\ref{new-e2-11}) for
$\varepsilon \in(0,\bar{\sigma})$. Applying It\^{o}'s formula to
$u_{\varepsilon}(t,B_{t})$ under $\tilde{G}$-expectation, we get by
(\ref{e2-14}) that%
\begin{equation}%
\begin{array}
[c]{cl}%
\tilde{Y}_{t}^{\varepsilon}= & \varphi(B_{T})+\int_{t}^{T}h(\tilde{Y}%
_{s}^{\varepsilon},\tilde{Z}_{s}^{\varepsilon})d\langle B\rangle_{s}-\int
_{t}^{T}\frac{1}{2}\varepsilon^{2}\left(  \partial_{xx}^{2}u_{\varepsilon
}(s,B_{s})+2h(\tilde{Y}_{s}^{\varepsilon},\tilde{Z}_{s}^{\varepsilon})\right)
^{-}ds\\
& -\int_{t}^{T}\tilde{Z}_{s}^{\varepsilon}dB_{s}-(L_{T}^{\varepsilon}%
-L_{t}^{\varepsilon}),
\end{array}
\label{e2-25}%
\end{equation}
where $\tilde{Y}_{t}^{\varepsilon}=u_{\varepsilon}(t,B_{t})$, $\tilde{Z}%
_{t}^{\varepsilon}=\partial_{x}u_{\varepsilon}(t,B_{t})$ and%
\[%
\begin{array}
[c]{cl}%
L_{t}^{\varepsilon}= & \int_{0}^{t}\frac{1}{2}\left[  \partial_{xx}%
^{2}u_{\varepsilon}(s,B_{s})+2h(\tilde{Y}_{s}^{\varepsilon},\tilde{Z}%
_{s}^{\varepsilon})\right]  d\langle B\rangle_{s}-\int_{0}^{t}G\left(
\partial_{xx}^{2}u_{\varepsilon}(s,B_{s})+2h(\tilde{Y}_{s}^{\varepsilon
},\tilde{Z}_{s}^{\varepsilon})\right)  ds\\
& -\int_{0}^{t}\frac{1}{2}\varepsilon^{2}\left(  \partial_{xx}^{2}%
u_{\varepsilon}(s,B_{s})+2h(\tilde{Y}_{s}^{\varepsilon},\tilde{Z}%
_{s}^{\varepsilon})\right)  ^{+}ds.
\end{array}
\]
Since $0\leq d\langle B\rangle_{s}\leq \bar{\sigma}^{2}ds$ under
$\mathbb{\tilde{E}}$, we deduce that $L^{\varepsilon}$ is non-increasing with
$L_{0}^{\varepsilon}=0$ under $\mathbb{\tilde{E}}$.

In the proof of Lemma \ref{pro2-2}, we know that, for each given $p>1$,%
\begin{equation}
\mathbb{\tilde{E}}\left[  \sup_{t\leq T}|\tilde{Y}_{t}^{\varepsilon}%
-Y_{t}|^{p}+\left(  \int_{0}^{T}|\partial_{x}u_{\varepsilon}(t,B_{t}%
+\varepsilon \tilde{B}_{t})-Z_{t}|^{2}d\langle B\rangle_{t}\right)
^{p/2}\right]  \rightarrow0\text{ as }\varepsilon \rightarrow0. \label{e2-46}%
\end{equation}
Thus $|Y|+|Z|\leq C$ by (\ref{e2-14}), where $C>0$ depends on $\bar{\sigma}$,
$\varphi$, $h$ and $T$. By (\ref{e2-45}), we get%
\[
\mathbb{\tilde{E}}\left[  |K_{T}|^{2}\right]  \leq C\mathbb{\tilde{E}}\left[
\sup_{t\leq T}|Y_{t}|^{2}+1+\int_{0}^{T}|Z_{t}|^{2}d\langle B\rangle
_{t}\right]  \leq C,
\]
where $C>0$ depends on $\bar{\sigma}$, $\varphi$, $h$ and $T$. By Theorem
\ref{th2-5}, we know $\partial_{xx}^{2}u_{\varepsilon}\geq-C$ for
$\varepsilon \in(0,\bar{\sigma})$, where $C>0$ depends on $\bar{\sigma}$,
$\varphi$, $h$ and $T$. Thus%
\[
\left(  \partial_{xx}^{2}u_{\varepsilon}(s,B_{s})+2h(\tilde{Y}_{s}%
^{\varepsilon},\tilde{Z}_{s}^{\varepsilon})\right)  ^{-}\leq C\text{ for }%
s\in \lbrack0,T]\text{ and }\varepsilon \in(0,\bar{\sigma}),
\]
where $C>0$ depends on $\bar{\sigma}$, $\varphi$, $h$ and $T$. By
(\ref{e2-14}) and (\ref{e2-25}), we have $|\tilde{Y}^{\varepsilon}|+|\tilde
{Z}^{\varepsilon}|\leq C$ for $\varepsilon \in(0,\bar{\sigma})$ and%
\[
\mathbb{\tilde{E}}\left[  |L_{T}^{\varepsilon}|^{2}\right]  \leq
C\mathbb{\tilde{E}}\left[  \sup_{t\leq T}|\tilde{Y}_{t}^{\varepsilon}%
|^{2}+1+\int_{0}^{T}|\tilde{Z}_{t}^{\varepsilon}|^{2}d\langle B\rangle
_{t}\right]  \leq C\text{ for }\varepsilon \in(0,\bar{\sigma}),
\]
where $C>0$ depends on $\bar{\sigma}$, $\varphi$, $h$ and $T$.

Applying It\^{o}'s formula to $|\tilde{Y}_{t}^{\varepsilon}-Y_{t}|^{2}$ on
$[0,T]$, we obtain
\begin{align*}
\mathbb{\tilde{E}}\left[  \int_{0}^{T}|\tilde{Z}_{t}^{\varepsilon}-Z_{t}%
|^{2}d\langle B\rangle_{t}\right]   &  \leq C\mathbb{\tilde{E}}\left[
\int_{0}^{T}|\tilde{Y}_{t}^{\varepsilon}-Y_{t}|dt+(|L_{T}^{\varepsilon
}|+|K_{T}|)\sup_{t\leq T}|\tilde{Y}_{t}^{\varepsilon}-Y_{t}|\right] \\
&  \leq C\left(  \mathbb{\tilde{E}}\left[  \sup_{t\leq T}|\tilde{Y}%
_{t}^{\varepsilon}-Y_{t}|^{2}\right]  \right)  ^{1/2},
\end{align*}
where $C>0$ depends on $\bar{\sigma}$, $\varphi$, $h$ and $T$. By
(\ref{e2-46}) and $|Z|+|\tilde{Z}^{\varepsilon}|\leq C$ for $\varepsilon
\in(0,\bar{\sigma})$, we get%
\[
\lim_{\varepsilon \downarrow0}\mathbb{\tilde{E}}\left[  \left(  \int_{0}%
^{T}|\tilde{Z}_{t}^{\varepsilon}-Z_{t}|^{2}d\langle B\rangle_{t}\right)
^{p/2}\right]  =0\text{ for each }p>1.
\]
Thus $Z\in H_{G}^{2,p}(0,T;\langle B\rangle)$, and then $K_{T}\in L_{G}%
^{p}(\Omega_{T})$ by (\ref{e2-45}).
\end{proof}

Moreover, we extend the above result to the following two lemmas.

\begin{lemma}
\label{le2-7}Let $t_{1}\in \lbrack0,T)$, $\varphi \in C_{b.Lip}(\mathbb{R})$,
$h_{1}\in C_{0}^{\infty}(\mathbb{R})$ and $h_{2}\in C_{0}^{\infty}%
(\mathbb{R}^{2})$. Then, for each given $p>1$, $G$-BSDE%
\begin{equation}
Y_{t}=\varphi(B_{T}-B_{t_{1}})+\int_{t}^{T}h_{1}(Y_{s})ds+\int_{t}^{T}%
h_{2}(Y_{s},Z_{s})d\langle B\rangle_{s}-\int_{t}^{T}Z_{s}dB_{s}-(K_{T}-K_{t})
\label{e2-47}%
\end{equation}
has a unique $L^{p}$-solution $(Y,Z,K)$ in the $G$-expectation space.
Furthermore, $Y_{t}=u(t,B_{t}-B_{t_{1}})$ for $t\in \lbrack t_{1},T]$, where
$u(t,x)=Y_{t}^{t,x}$ and $(Y_{s}^{t,x})_{s\in \lbrack t,T]}$ satisfies the
following $G$-BSDE:%
\begin{equation}
Y_{s}^{t,x}=\varphi(x+B_{T}-B_{t})+\int_{s}^{T}h_{1}(Y_{r}^{t,x})dr+\int
_{s}^{T}h_{2}(Y_{r}^{t,x},Z_{r}^{t,x})d\langle B\rangle_{r}-\int_{s}^{T}%
Z_{r}^{t,x}dB_{r}-(K_{T}^{t,x}-K_{s}^{t,x}). \label{e2-48}%
\end{equation}

\end{lemma}

\begin{proof}
The uniqueness is due to (\ref{e2-2}) and (\ref{e2-8}) in Proposition
\ref{pro2-1}. For each given $p>1$, we can find a sequence $\varphi_{n}\in
C_{0}^{\infty}(\mathbb{R})$, $n\geq1$, such that $\mathbb{\hat{E}}\left[
|\varphi_{n}(B_{T}-B_{t_{1}})-\varphi(B_{T}-B_{t_{1}})|^{p+1}\right]
\rightarrow0$ as $n\rightarrow \infty$. Similar to the proof of Lemma
\ref{pro2-6}, the following $G$-BSDE%
\begin{equation}
Y_{t}^{n}=\varphi_{n}(B_{T}-B_{t_{1}})+\int_{t}^{T}h_{1}(Y_{s}^{n})ds+\int
_{t}^{T}h_{2}(Y_{s}^{n},Z_{s}^{n})d\langle B\rangle_{s}-\int_{t}^{T}Z_{s}%
^{n}dB_{s}-(K_{T}^{n}-K_{t}^{n}) \label{e2-49}%
\end{equation}
has a unique $L^{p}$-solution $(Y^{n},Z^{n},K^{n})$ in the $G$-expectation
space. By (\ref{e2-2}), (\ref{e2-8}) in Proposition \ref{pro2-1} and Theorem
\ref{th1-1}, we can easily deduce%
\[
\lim_{n,m\rightarrow \infty}\mathbb{\hat{E}}\left[  \sup_{t\leq T}|Y_{t}%
^{n}-Y_{t}^{m}|^{p}+\left(  \int_{0}^{T}|Z_{t}^{n}-Z_{t}^{m}|^{2}d\langle
B\rangle_{t}\right)  ^{p/2}+|K_{T}^{n}-K_{T}^{m}|^{p}\right]  =0.
\]
Thus there exist $Y\in S_{G}^{p}(0,T)$, $Z\in H_{G}^{2,p}(0,T;\langle
B\rangle)$ and a non-increasing $G$-martingale $K$ with $K_{0}=0$ and
$K_{T}\in L_{G}^{p}(\Omega_{T})$ such that%
\[
\lim_{n\rightarrow \infty}\mathbb{\hat{E}}\left[  \sup_{t\leq T}|Y_{t}%
^{n}-Y_{t}|^{p}+\left(  \int_{0}^{T}|Z_{t}^{n}-Z_{t}|^{2}d\langle B\rangle
_{t}\right)  ^{p/2}+|K_{T}^{n}-K_{T}|^{p}\right]  =0.
\]
From this we can easily get%
\[
\lim_{n\rightarrow \infty}\mathbb{\hat{E}}\left[  \sup_{t\leq T}\left(
\int_{t}^{T}|h_{1}(Y_{s}^{n})-h_{1}(Y_{s})|ds+\int_{t}^{T}|h_{2}(Y_{s}%
^{n},Z_{s}^{n})-h_{2}(Y_{s},Z_{s})|d\langle B\rangle_{s}+\left \vert \int
_{t}^{T}(Z_{s}^{n}-Z_{s})dB_{s}\right \vert \right)  ^{p}\right]  =0.
\]
Thus $(Y,Z,K)$ satisfies $G$-BSDE (\ref{e2-47}) by taking $n\rightarrow \infty$
in (\ref{e2-49}).

From the above proof, we know that $G$-BSDE (\ref{e2-48}) has a unique $L^{p}%
$-solution $(Y^{t,x},Z^{t,x},K^{t,x})$ and $Y_{t}^{t,x}\in \mathbb{R}$. By
(\ref{e2-2}) in Proposition \ref{pro2-1}, we obtain that%
\[
|u(t,x)-u(t,x^{\prime})|\leq C|x-x^{\prime}|\text{ and }|Y_{t}-u(t,x)|\leq
C|B_{t}-B_{t_{1}}-x|
\]
where the constant $C>0$ depends on $\bar{\sigma}$, $\varphi$, $h_{1}$,
$h_{2}$ and $T$. Thus we get $Y_{t}=u(t,B_{t}-B_{t_{1}})$ for $t\in \lbrack
t_{1},T]$.
\end{proof}

\begin{lemma}
\label{le2-8}Let $\xi \in Lip(\Omega_{T})$, $f(t,y)=\sum_{i=1}^{N_{1}}f_{t}%
^{i}h_{1}^{i}(y)$ and $g(t,y,z)=\sum_{j=1}^{N_{2}}g_{t}^{j}h_{2}^{j}(y,z)$
with $f^{i}$, $g^{j}\in M^{0}(0,T)$, $h_{1}^{i}\in C_{0}^{\infty}(\mathbb{R}%
)$, $h_{2}^{j}\in C_{0}^{\infty}(\mathbb{R}^{2})$, $i\leq N_{1}$, $j\leq
N_{2}$. Then $G$-BSDE (\ref{e2-1}) has a unique $L^{p}$-solution $(Y,Z,K)$ for
each given $p>1$.
\end{lemma}

\begin{proof}
The uniqueness is due to (\ref{e2-2}) and (\ref{e2-8}) in Proposition
\ref{pro2-1}. For the existence, we only prove the special case $\xi
=\varphi(B_{t_{1}},B_{T}-B_{t_{1}})$, $f(t,y)=0$ and $g(t,y,z)=(I_{[0,t_{1}%
)}(t)+\psi(B_{t_{1}})I_{[t_{1},T]}(t))h_{2}(y,z)$, the general case is similar.

By Lemma \ref{le2-7}, $G$-BSDE
\begin{equation}
Y_{t}^{x}=\varphi(x,B_{T}-B_{t_{1}})+\int_{t}^{T}\psi(x)h_{2}(Y_{s}^{x}%
,Z_{s}^{x})d\langle B\rangle_{s}-\int_{t}^{T}Z_{s}^{x}dB_{s}-(K_{T}^{x}%
-K_{t}^{x}) \label{e2-50}%
\end{equation}
has a unique $L^{p}$-solution $(Y^{x},Z^{x},K^{x})$ for each given $p>1$.
Furthermore, $Y_{t}^{x}=u(t,x,B_{t}-B_{t_{1}})$ for $t\in \lbrack t_{1},T]$,
where $u(t,x,x^{\prime})=Y_{t}^{t,x,x^{\prime}}$ and $(Y_{s}^{t,x,x^{\prime}%
})_{s\in \lbrack t,T]}$ satisfies the following $G$-BSDE:%
\[
Y_{s}^{t,x,x^{\prime}}=\varphi(x,x^{\prime}+B_{T}-B_{t})+\int_{s}^{T}%
\psi(x)h_{2}(Y_{r}^{t,x,x^{\prime}},Z_{r}^{t,x,x^{\prime}})d\langle
B\rangle_{r}-\int_{s}^{T}Z_{r}^{t,x,x^{\prime}}dB_{r}-(K_{T}^{t,x,x^{\prime}%
}-K_{s}^{t,x,x^{\prime}}).
\]
By (\ref{e2-2}) in Proposition \ref{pro2-1}, we obtain that, for $t\in \lbrack
t_{1},T]$, $x$, $x^{\prime}$, $\tilde{x}$, $\tilde{x}^{\prime}\in \mathbb{R}$,
\begin{equation}
|u(t,x,x^{\prime})|\leq C\text{ and }|u(t,x,x^{\prime})-u(t,\tilde{x}%
,\tilde{x}^{\prime})|\leq C(|x-\tilde{x}|+|x^{\prime}-\tilde{x}^{\prime}|)
\label{e2-51}%
\end{equation}
where the constant $C>0$ depends on $\bar{\sigma}$, $\varphi$, $\psi$, $h_{2}$
and $T$.

For each positive integer $n$, by partition of unity theorem, we can find
$l_{i}^{n}\in C_{0}^{\infty}(\mathbb{R})$, $i=1$,$\ldots$,$k_{n}$, such that%
\[
0\leq l_{i}^{n}\leq1\text{ and }\lambda(\text{supp}(l_{i}^{n}))\leq \frac{1}%
{n}\text{ for }i\leq k_{n}\text{, }I_{[-n,n]}(x)\leq \sum_{i=1}^{k_{n}}%
l_{i}^{n}(x)\leq1,
\]
where $\lambda(\cdot)$ is the Lebesgue measure. For $t\in \lbrack t_{1},T]$,
set%
\[
Y_{t}^{n}=\sum_{i=1}^{k_{n}}l_{i}^{n}(B_{t_{1}})Y_{t}^{x_{i}^{n}}\text{,
}Z_{t}^{n}=\sum_{i=1}^{k_{n}}l_{i}^{n}(B_{t_{1}})Z_{t}^{x_{i}^{n}}\text{,
}K_{t}^{n}=\sum_{i=1}^{k_{n}}l_{i}^{n}(B_{t_{1}})K_{t}^{x_{i}^{n}}\text{,}%
\]
where $l_{i}^{n}(x_{i}^{n})>0$. Then, by (\ref{e2-50}), we get that, for
$t\in \lbrack t_{1},T]$,
\begin{equation}
Y_{t}^{n}=Y_{T}^{n}+\int_{t}^{T}\sum_{i=1}^{k_{n}}l_{i}^{n}(B_{t_{1}}%
)\psi(x_{i}^{n})h_{2}(Y_{s}^{x_{i}^{n}},Z_{s}^{x_{i}^{n}})d\langle
B\rangle_{s}-\int_{t}^{T}Z_{s}^{n}dB_{s}-(K_{T}^{n}-K_{t}^{n}). \label{e2-52}%
\end{equation}
It follows from (\ref{e2-51}) that%
\begin{align*}
&  |Y_{t}^{n}-u(t,B_{t_{1}},B_{t}-B_{t_{1}})|\\
&  \leq \sum_{i=1}^{k_{n}}l_{i}^{n}(B_{t_{1}})|u(t,x_{i}^{n},B_{t}-B_{t_{1}%
})-u(t,B_{t_{1}},B_{t}-B_{t_{1}})|+\left(  1-\sum_{i=1}^{k_{n}}l_{i}%
^{n}(B_{t_{1}})\right)  |u(t,B_{t_{1}},B_{t}-B_{t_{1}})|\\
&  \leq \frac{C}{n}+\frac{C}{n}|B_{t_{1}}|,
\end{align*}
which implies%
\begin{equation}
\lim_{n,m\rightarrow \infty}\mathbb{\hat{E}}\left[  \sup_{t\in \lbrack t_{1}%
,T]}|Y_{t}^{n}-Y_{t}^{m}|^{p}\right]  =0. \label{e2-53}%
\end{equation}
Noting that $|\sum_{i=1}^{k_{n}}l_{i}^{n}(B_{t_{1}})\psi(x_{i}^{n})h_{2}%
(Y_{s}^{x_{i}^{n}},Z_{s}^{x_{i}^{n}})|\leq C$, where $C$ depends on $\psi$ and
$h_{2}$, we obtain by (\ref{e2-8}) in Proposition \ref{pro2-1} that
\begin{equation}
\lim_{n,m\rightarrow \infty}\mathbb{\hat{E}}\left[  \left(  \int_{t_{1}}%
^{T}|Z_{t}^{n}-Z_{t}^{m}|^{2}d\langle B\rangle_{t}\right)  ^{p/2}\right]  =0.
\label{e2-54}%
\end{equation}
It is easy to verify that%
\begin{align*}
&  \left \vert \sum_{i=1}^{k_{n}}l_{i}^{n}(B_{t_{1}})\psi(x_{i}^{n})h_{2}%
(Y_{s}^{x_{i}^{n}},Z_{s}^{x_{i}^{n}})-\psi(B_{t_{1}})h_{2}(Y_{s}^{n},Z_{s}%
^{n})\right \vert \\
&  \leq \frac{C}{n}(1+|B_{t_{1}}|)+C\sum_{i=1}^{k_{n}}l_{i}^{n}(B_{t_{1}%
})\left \vert h_{2}\left(  \sum_{j=1}^{k_{n}}l_{j}^{n}(B_{t_{1}})Y_{s}%
^{x_{i}^{n}},\sum_{j=1}^{k_{n}}l_{j}^{n}(B_{t_{1}})Z_{s}^{x_{i}^{n}}\right)
-h_{2}(Y_{s}^{n},Z_{s}^{n})\right \vert \\
&  \leq \frac{C}{n}(1+|B_{t_{1}}|)+C\sum_{i,j=1}^{k_{n}}l_{i}^{n}(B_{t_{1}%
})l_{j}^{n}(B_{t_{1}})\left(  \left \vert Y_{s}^{x_{i}^{n}}-Y_{s}^{x_{j}^{n}%
}\right \vert +\left \vert Z_{s}^{x_{i}^{n}}-Z_{s}^{x_{j}^{n}}\right \vert
\right)  .
\end{align*}
By (\ref{e2-51}), we know that $\left \vert Y_{s}^{x_{i}^{n}}-Y_{s}^{x_{j}^{n}%
}\right \vert \leq C|x_{i}^{n}-x_{j}^{n}|$. Similar to the proof of
(\ref{e2-8}) in Proposition \ref{pro2-1}, we deduce that, for each
$P\in \mathcal{P}$,%
\begin{align*}
&  E_{P}\left[  \left(  \int_{t_{1}}^{T}\left \vert Z_{t}^{x_{i}^{n}}%
-Z_{t}^{x_{j}^{n}}\right \vert ^{2}d\langle B\rangle_{t}\right)  ^{p/2}%
\Big{|}\mathcal{B}(\Omega_{t_{1}})\right] \\
&  \leq C\left \{  E_{P}\left[  \sup_{t\in \lbrack t_{1},T]}\left \vert
Y_{t}^{x_{i}^{n}}-Y_{t}^{x_{j}^{n}}\right \vert ^{p}\Big{|}\mathcal{B}%
(\Omega_{t_{1}})\right]  +\left(  E_{P}\left[  \sup_{t\in \lbrack t_{1}%
,T]}\left \vert Y_{t}^{x_{i}^{n}}-Y_{t}^{x_{j}^{n}}\right \vert ^{p}%
\Big{|}\mathcal{B}(\Omega_{t_{1}})\right]  \right)  ^{1/2}\right \} \\
&  \leq C\left(  |x_{i}^{n}-x_{j}^{n}|^{p}+|x_{i}^{n}-x_{j}^{n}|^{p/2}\right)
.
\end{align*}
Noting that $l_{i}^{n}(B_{t_{1}})l_{j}^{n}(B_{t_{1}})|x_{i}^{n}-x_{j}^{n}|=0$
if $|x_{i}^{n}-x_{j}^{n}|>\frac{2}{n}$, we obtain%
\begin{equation}
\lim_{n\rightarrow \infty}\sup_{P\in \mathcal{P}}E_{P}\left[  \left(
\int_{t_{1}}^{T}\left \vert \sum_{i=1}^{k_{n}}l_{i}^{n}(B_{t_{1}})\psi
(x_{i}^{n})h_{2}(Y_{s}^{x_{i}^{n}},Z_{s}^{x_{i}^{n}})-\psi(B_{t_{1}}%
)h_{2}(Y_{s}^{n},Z_{s}^{n})\right \vert d\langle B\rangle_{s}\right)
^{p}\right]  =0. \label{e2-55}%
\end{equation}
By (\ref{e2-52}), (\ref{e2-53}), (\ref{e2-54}) and (\ref{e2-55}), we get
$\lim_{n,m\rightarrow \infty}\mathbb{\hat{E}}\left[  |(K_{T}^{n}-K_{t_{1}}%
^{n})-(K_{T}^{m}-K_{t_{1}}^{m})|^{p}\right]  =0$. Thus there exist $Y\in
S_{G}^{p}(t_{1},T)$, $Z\in H_{G}^{2,p}(t_{1},T;\langle B\rangle)$ and a
non-increasing $K$ with $K_{t_{1}}=0$ and $K_{T}\in L_{G}^{p}(\Omega_{T})$
such that
\[
Y_{t}=\varphi(B_{t_{1}},B_{T}-B_{t_{1}})+\int_{t}^{T}\psi(B_{t_{1}}%
)h_{2}(Y_{s},Z_{s})d\langle B\rangle_{s}-\int_{t}^{T}Z_{s}dB_{s}-(K_{T}%
-K_{t})\text{ for }t\in \lbrack t_{1},T].
\]

In the following, we prove that $K$ is a $G$-martingale. For each positive
integer $n$, set%
\[
\tilde{l}_{i}^{n}(x)=I_{[-n+\frac{i}{n},-n+\frac{i+1}{n})}(x)\text{ for
}i=0,\ldots,2n^{2}-1,\text{ }\tilde{l}_{2n^{2}}^{n}(x)=I_{[-n,n)^{c}}(x)
\]
and%
\[
\tilde{Y}_{t}^{n}=\sum_{i=0}^{2n^{2}}\tilde{l}_{i}^{n}(B_{t_{1}}%
)Y_{t}^{-n+\frac{i}{n}}\text{, }\tilde{Z}_{t}^{n}=\sum_{i=0}^{2n^{2}}\tilde
{l}_{i}^{n}(B_{t_{1}})Z_{t}^{-n+\frac{i}{n}}\text{, }\tilde{K}_{t}^{n}%
=\sum_{i=0}^{2n^{2}}\tilde{l}_{i}^{n}(B_{t_{1}})K_{t}^{-n+\frac{i}{n}}\text{.}%
\]
Then, for $t\in \lbrack t_{1},T]$,
\[
\tilde{Y}_{t}^{n}=\tilde{Y}_{T}^{n}+\int_{t}^{T}\sum_{i=0}^{2n^{2}}\tilde
{l}_{i}^{n}(B_{t_{1}})\psi \left(  -n+\frac{i}{n}\right)  h_{2}(\tilde{Y}%
_{s}^{n},\tilde{Z}_{s}^{n})d\langle B\rangle_{s}-\int_{t}^{T}\tilde{Z}_{s}%
^{n}dB_{s}-(\tilde{K}_{T}^{n}-\tilde{K}_{t}^{n}).
\]
Similar to the above proof, we have $\lim_{n\rightarrow \infty}\mathbb{\hat{E}%
}\left[  \sup_{t\in \lbrack t_{1},T]}|\tilde{Y}_{t}^{n}-Y_{t}|^{p}\right]  =0$,
which implies
\[
\lim_{n\rightarrow \infty}\mathbb{\hat{E}}\left[  |(\tilde{K}_{T}^{n}-\tilde
{K}_{t_{1}}^{n})-(K_{T}-K_{t_{1}})|^{p}\right]  =0.
\]
By Proposition 2.5 in \cite{HJPS1}, we know that, for $t\in \lbrack t_{1},T]$,
$\mathbb{\hat{E}}_{t}\left[  \tilde{K}_{T}^{n}-\tilde{K}_{t}^{n}\right]  =0$
and%
\[
\mathbb{\hat{E}}\left[  |\mathbb{\hat{E}}_{t}\left[  K_{T}-K_{t}\right]
|\right]  =\mathbb{\hat{E}}\left[  |\mathbb{\hat{E}}_{t}\left[  K_{T}%
-K_{t}\right]  -\mathbb{\hat{E}}_{t}[\tilde{K}_{T}^{n}-\tilde{K}_{t}%
^{n}]|\right]  \leq \mathbb{\hat{E}}\left[  |(K_{T}-K_{t})-(\tilde{K}_{T}%
^{n}-\tilde{K}_{t}^{n})|\right]  ,
\]
which implies $\mathbb{\hat{E}}_{t}\left[  K_{T}\right]  =K_{t}$ by letting
$n\rightarrow \infty$. Thus we obtain an $L^{p}$-solution $(Y,Z,K)$ on
$[t_{1},T]$. Noting that $Y_{t_{1}}=$ $u(t_{1},B_{t_{1}},0)$, we obtain the
desired result by applying Lemma \ref{le2-7} to find an $L^{p}$-solution on
$[0,t_{1}]$.
\end{proof}

Now, we give the following existence and uniqueness result of $G$-BSDE
(\ref{e2-1}).

\begin{theorem}
\label{th2-2}Suppose that $\xi$, $f$ and $g$ satisfy (H1) and (H2). Then
$G$-BSDE (\ref{e2-1}) has a unique $L^{p}$-solution $(Y,Z,K)$ for each given
$p\in(1,\bar{p})$.
\end{theorem}

\begin{proof}
The uniqueness is due to (\ref{e2-2}) and (\ref{e2-8}) in Proposition
\ref{pro2-1}. For each positive integer $n$, by partition of unity theorem, we
can find $h_{i}^{n}\in C_{0}^{\infty}(\mathbb{R})$, $\tilde{h}_{j}^{n}\in
C_{0}^{\infty}(\mathbb{R}^{2})$, $i\leq k_{n}$, $j\leq \tilde{k}_{n}$, such
that%
\[
0\leq h_{i}^{n}\leq1\text{ and }\lambda(\text{supp}(h_{i}^{n}))\leq \frac{1}%
{n}\text{ for }i\leq k_{n}\text{, }I_{[-n,n]}(y)\leq \sum_{i=1}^{k_{n}}%
h_{i}^{n}(y)\leq1,
\]%
\[
0\leq \tilde{h}_{j}^{n}\leq1\text{ and }\lambda(\text{supp}(\tilde{h}_{j}%
^{n}))\leq \frac{1}{n}\text{ for }j\leq \tilde{k}_{n}\text{, }I_{[-n,n]\times
\lbrack-n,n]}(y,z)\leq \sum_{j=1}^{\tilde{k}_{n}}\tilde{h}_{j}^{n}(y,z)\leq1.
\]
For each $N>0$, set $\tilde{f}(t,y)=f(t,y)-f(t,0)$, $\tilde{g}%
(t,y,z)=g(t,y,z)-g(t,0,0)$,
\[
\tilde{f}^{N}(t,y)=(\tilde{f}(t,y)\wedge N)\vee(-N),\text{ }\tilde{g}%
^{N}(t,y,z)=(\tilde{g}(t,y,z)\wedge N)\vee(-N),
\]%
\[
f^{N}(t,y)=f(t,0)+\tilde{f}^{N}(t,y)\text{, }g^{N}(t,y,z)=g(t,0,0)+\tilde
{g}^{N}(t,y,z),
\]%
\[
f_{n}^{N}(t,y)=f(t,0)+\sum_{i=1}^{k_{n}}\tilde{f}^{N}(t,y_{i}^{n})h_{i}%
^{n}(y),\text{ }g_{n}^{N}(t,y,z)=g(t,0,0)+\sum_{j=1}^{\tilde{k}_{n}}\tilde
{g}^{N}(t,\tilde{y}_{j}^{n},\tilde{z}_{j}^{n})\tilde{h}_{j}^{n}(y,z),
\]
where $h_{i}^{n}(y_{i}^{n})>0$, $\tilde{h}_{j}^{n}(\tilde{y}_{j}^{n},\tilde
{z}_{j}^{n})>0$ for $i\leq k_{n}$, $j\leq \tilde{k}_{n}$. By Proposition
\ref{pro2-1} and Lemma \ref{le2-8}, we can easily deduce that $G$-BSDE%
\begin{equation}
Y_{t}^{N,n}=\xi+\int_{t}^{T}f_{n}^{N}(s,Y_{s}^{N,n})ds+\int_{t}^{T}g_{n}%
^{N}(s,Y_{s}^{N,n},Z_{s}^{N,n})d\langle B\rangle_{s}-\int_{t}^{T}Z_{s}%
^{N,n}dB_{s}-(K_{T}^{N,n}-K_{t}^{N,n}) \label{e2-56}%
\end{equation}
has a unique $L^{p}$-solution $(Y^{N,n},Z^{N,n},K^{N,n})$ for each given
$p\in(1,\bar{p})$. Noting that%
\[
f_{n}^{N}(s,Y_{s}^{N,n})=f^{N}(s,Y_{s}^{N,n})+\hat{f}_{n}^{N}(s)\text{ and
}g_{n}^{N}(s,Y_{s}^{N,n},Z_{s}^{N,n})=g^{N}(s,Y_{s}^{N,n},Z_{s}^{N,n})+\hat
{g}_{n}^{N}(s),
\]
where $|\hat{f}_{n}^{N}(s)|=|f_{n}^{N}(s,Y_{s}^{N,n})-f^{N}(s,Y_{s}%
^{N,n})|\leq(\frac{L}{n}+\frac{N}{n}|Y_{s}^{N,n}|)\wedge(2N)$, $|\hat{g}%
_{n}^{N}(s)|\leq \lbrack \frac{L}{n}+\frac{N}{n}(|Y_{s}^{N,n}|+|Z_{s}%
^{N,n}|)]\wedge(2N)$. By Proposition \ref{pro2-1}, we can easily deduce that,
for each given $p\in(1,\bar{p})$,
\[
\lim_{n,m\rightarrow \infty}\mathbb{\hat{E}}\left[  \sup_{t\in \lbrack
0,T]}|Y_{t}^{N,n}-Y_{t}^{N,m}|^{p}+\left(  \int_{0}^{T}|Z_{t}^{N,n}%
-Z_{t}^{N,m}|^{2}d\langle B\rangle_{t}\right)  ^{p/2}+|K_{T}^{N,n}-K_{T}%
^{N,m}|^{p}\right]  =0,
\]
which implies that $G$-BSDE%
\begin{equation}
Y_{t}^{N}=\xi+\int_{t}^{T}f^{N}(s,Y_{s}^{N})ds+\int_{t}^{T}g^{N}(s,Y_{s}%
^{N},Z_{s}^{N})d\langle B\rangle_{s}-\int_{t}^{T}Z_{s}^{N}dB_{s}-(K_{T}%
^{N}-K_{t}^{N}) \label{e2-57}%
\end{equation}
has a unique $L^{p}$-solution $(Y^{N},Z^{N},K^{N})$ for each given
$p\in(1,\bar{p})$. By (\ref{e2-6}), (\ref{e2-7}) in Proposition \ref{pro2-1}
and Theorem \ref{th1-1}, we obtain that, for each $p\in(1,\bar{p})$,%
\begin{equation}
\sup_{N>0}\mathbb{\hat{E}}\left[  \sup_{t\in \lbrack0,T]}|Y_{t}^{N}%
|^{p}+\left(  \int_{0}^{T}|Z_{t}^{N}|^{2}d\langle B\rangle_{t}\right)
^{p/2}+|K_{T}^{N}|^{p}\right]  \leq C, \label{e2-58}%
\end{equation}
where the constant $C>0$ depends on $p$, $\bar{p}$, $\bar{\sigma}$, $L$ and
$T$. For each fixed $p\in(1,\bar{p})$, we have
\[
|f^{N_{1}}(s,Y_{s}^{N_{1}})-f^{N_{2}}(s,Y_{s}^{N_{1}})|\leq(N_{1}\wedge
N_{2})^{-\delta}|\tilde{f}(s,Y_{s}^{N_{1}})|^{1+\delta}\leq L^{1+\delta}%
(N_{1}\wedge N_{2})^{-\delta}|Y_{s}^{N_{1}}|^{1+\delta},
\]%
\[
|g^{N_{1}}(s,Y_{s}^{N_{1}},Z_{s}^{N_{1}})-g^{N_{2}}(s,Y_{s}^{N_{1}}%
,Z_{s}^{N_{1}})|\leq L^{1+\delta}(N_{1}\wedge N_{2})^{-\delta}(|Y_{s}^{N_{1}%
}|+|Z_{s}^{N_{1}}|)^{1+\delta},
\]
where $\delta=[\frac{1}{2}(\frac{\bar{p}}{p}-1)]\wedge1$. Thus, by
(\ref{e2-2}), (\ref{e2-8}) in Proposition \ref{pro2-1}, (\ref{e2-58}) and
Theorem \ref{th1-1}, we get%
\[
\lim_{N_{1},N_{2}\rightarrow \infty}\mathbb{\hat{E}}\left[  \sup_{t\in
\lbrack0,T]}|Y_{t}^{N_{1}}-Y_{t}^{N_{2}}|^{p}+\left(  \int_{0}^{T}%
|Z_{t}^{N_{1}}-Z_{t}^{N_{2}}|^{2}d\langle B\rangle_{t}\right)  ^{p/2}%
+|K_{T}^{N_{1}}-K_{T}^{N_{2}}|^{p}\right]  =0,
\]
which implies the desired result by letting $N\rightarrow \infty$ in
(\ref{e2-57}).
\end{proof}

The following example shows that $f$ can not contain $z$ in $G$-BSDE
(\ref{e2-1}).

\begin{example}
Let $B$ be a $1$-dimensional $G$-Brownian motion with $G(a):=\frac{1}{2}%
\bar{\sigma}^{2}a^{+}$ for $a\in \mathbb{R}$. For each $n\geq1$, we know that
$((n^{-1}+\langle B\rangle_{s})^{-1/5})_{s\in \lbrack0,T]}\in H_{G}%
^{2,p}(0,T;\langle B\rangle)$ for each $p>1$. Since%
\[
\left \vert (n^{-1}+\langle B\rangle_{s})^{-1/5}-(\langle B\rangle_{s}%
)^{-1/5}\right \vert \leq(\langle B\rangle_{s})^{-2/5}\left \vert (n^{-1}%
+\langle B\rangle_{s})^{1/5}-(\langle B\rangle_{s})^{1/5}\right \vert \leq
n^{-1/5}(\langle B\rangle_{s})^{-2/5},
\]
we have
\[
\int_{0}^{T}|(n^{-1}+\langle B\rangle_{s})^{-1/5}-(\langle B\rangle
_{s})^{-1/5}|^{2}d\langle B\rangle_{s}\leq n^{-2/5}\int_{0}^{T}(\langle
B\rangle_{s})^{-4/5}d\langle B\rangle_{s}=5n^{-2/5}(\langle B\rangle
_{T})^{1/5}.
\]
Thus $((\langle B\rangle_{s})^{-1/5})_{s\in \lbrack0,T]}\in H_{G}%
^{2,p}(0,T;\langle B\rangle)$ for each $p>1$, which implies $\int_{0}%
^{T}(\langle B\rangle_{s})^{-1/5}dB_{s}\in L_{G}^{p}(\Omega_{T})$ for each
$p>1$. Consider the following linear $G$-BSDE:%
\begin{equation}
Y_{t}=\int_{0}^{T}(\langle B\rangle_{s})^{-1/5}dB_{s}+\int_{t}^{T}Z_{s}%
ds-\int_{t}^{T}Z_{s}dB_{s}-(K_{T}-K_{t}), \label{e2-59}%
\end{equation}
we assert that, for each given $p>1$, the above $G$-BSDE has no $L^{p}%
$-solution $(Y,Z,K)$. Otherwise, there exists an $L^{p}$-solution $(Y,Z,K)$
for some $p>1$.

For each $\varepsilon>0$, we introduce the following $\tilde{G}^{\varepsilon}%
$-expectation $\mathbb{\tilde{E}}^{\varepsilon}$. Set $\tilde{\Omega}%
_{T}=C_{0}([0,T];\mathbb{R}^{2})$ and the canonical process is denoted by
$(B,\tilde{B})$. For each $A\in \mathbb{S}_{2}$, define%
\[
\tilde{G}^{\varepsilon}\left(  A\right)  =\frac{1}{2}\sup_{\varepsilon^{2}\leq
v\leq \bar{\sigma}^{2}}\mathrm{tr}\left[  A\left(
\begin{array}
[c]{cc}%
v & 1\\
1 & \varepsilon^{-2}%
\end{array}
\right)  \right]  .
\]
By Proposition 3.1.5 in Peng \cite{P2019}, we know that $\varepsilon \tilde{B}$
is a classical $1$-dimensional standard Brownian motion under $\mathbb{\tilde
{E}}^{\varepsilon}$ and $\mathbb{\tilde{E}}^{\varepsilon}|_{Lip(\Omega_{T}%
)}\leq \mathbb{\hat{E}}$. Thus $G$-BSDE (\ref{e2-59}) still holds under
$\mathbb{\tilde{E}}^{\varepsilon}$. Similar to (\ref{new-e2-3}), we know that
$\langle B,\tilde{B}\rangle_{t}=t$ and $\langle \tilde{B}\rangle_{t}%
=\varepsilon^{-2}t$ under $\mathbb{\tilde{E}}^{\varepsilon}$. Consider the
following $G$-SDE under $\mathbb{\tilde{E}}^{\varepsilon}$:%
\[
dX_{t}=X_{t}d\tilde{B}_{t}\text{, }X_{0}=1.
\]
The solution is $X_{t}=\exp(\tilde{B}_{t}-2^{-1}\varepsilon^{-2}t)>0$.
Applying It\^{o}'s formula to $X_{t}Y_{t}$ on $[0,T]$ under $\mathbb{\tilde
{E}}^{\varepsilon}$, we get%
\[
X_{T}Y_{T}=Y_{0}+\int_{0}^{T}X_{t}Z_{t}dB_{t}+\int_{0}^{T}X_{t}Y_{t}d\tilde
{B}_{t}+\int_{0}^{T}X_{t}dK_{t}.
\]
Since $\int_{0}^{T}X_{t}dK_{t}\leq0$, we obtain%
\[
Y_{0}\geq \mathbb{\tilde{E}}^{\varepsilon}[X_{T}Y_{T}]=\mathbb{\tilde{E}%
}^{\varepsilon}\left[  X_{T}\int_{0}^{T}(\langle B\rangle_{s})^{-1/5}%
dB_{s}\right]  \text{ for each }\varepsilon>0.
\]
Applying It\^{o}'s formula to $X_{t}\int_{0}^{t}(\langle B\rangle_{s}%
)^{-1/5}dB_{s}$ on $[0,T]$ under $\mathbb{\tilde{E}}^{\varepsilon}$, we deduce%
\[
Y_{0}\geq \mathbb{\tilde{E}}^{\varepsilon}\left[  X_{T}\int_{0}^{T}(\langle
B\rangle_{s})^{-1/5}dB_{s}\right]  =\mathbb{\tilde{E}}^{\varepsilon}\left[
\int_{0}^{T}X_{s}(\langle B\rangle_{s})^{-1/5}ds\right]  \text{ for each
}\varepsilon>0.
\]
Let $E^{\varepsilon}$ be the linear $\bar{G}^{\varepsilon}$-expectation with%
\[
\bar{G}^{\varepsilon}\left(  A\right)  =\frac{1}{2}\mathrm{tr}\left[  A\left(
\begin{array}
[c]{cc}%
\varepsilon^{2} & 1\\
1 & \varepsilon^{-2}%
\end{array}
\right)  \right]  \text{ for }A\in \mathbb{S}_{2}.
\]
Since $\bar{G}^{\varepsilon}\leq \tilde{G}^{\varepsilon}$, we know that
$E^{\varepsilon}\leq \mathbb{\tilde{E}}^{\varepsilon}$. By Proposition 3.1.5 in
Peng \cite{P2019}, we know that $\varepsilon^{-1}B$ and $\varepsilon \tilde{B}$
are two classical $1$-dimensional standard Brownian motion under
$E^{\varepsilon}$. Then we get%
\[
Y_{0}\geq E^{\varepsilon}\left[  X_{T}\int_{0}^{T}(\langle B\rangle
_{s})^{-1/5}dB_{s}\right]  =\frac{5}{4}T^{4/5}\varepsilon^{-2/5}\text{ for
each }\varepsilon>0,
\]
which contradicts to $Y_{0}\in \mathbb{R}$. Thus, for each given $p>1$,
$G$-BSDE (\ref{e2-59}) has no $L^{p}$-solution $(Y,Z,K)$.
\end{example}

Finally, we give the following existence and uniqueness result of $G$-BSDE
(\ref{new-e2-4}).

\begin{theorem}
Suppose that $\xi$, $f$, $g_{ij}$, $g_{l}$, $i$, $j\leq d^{\prime}$,
$d^{\prime}<l\leq d$, satisfy (H1) and (H2). Then $G$-BSDE (\ref{new-e2-4})
has a unique $L^{p}$-solution $(Y,Z,K)$ for each given $p\in(1,\bar{p})$.
\end{theorem}

\begin{proof}
The proof of this theorem is similar to Theorem \ref{th2-2}, we omit it.
\end{proof}

\section{Application to the regularity of fully nonlinear PDEs}

For simplicity of representation, we only consider $1$-dimensional
$G$-Brownian motion with $G(a)=\frac{1}{2}\bar{\sigma}^{2}a^{+}$, the methods
still hold for the $d$-dimensional $G$-Brownian motion with $G(\cdot)$ given
in (\ref{new-e2-2}). For each fixed $t\in \lbrack0,T]$ and $\xi \in \cap_{p\geq
2}L_{G}^{p}(\Omega_{t})$, consider the following $G$-FBSDE:%
\begin{equation}
dX_{s}^{t,\xi}=b(s,X_{s}^{t,\xi})ds+h(s,X_{s}^{t,\xi})d\langle B\rangle
_{s}+\sigma(s,X_{s}^{t,\xi})dB_{s},\text{ }X_{t}^{t,\xi}=\xi,\text{ }%
s\in \lbrack t,T], \label{e3-1}%
\end{equation}%
\begin{equation}%
\begin{array}
[c]{rl}%
Y_{s}^{t,\xi}= & \varphi(X_{T}^{t,\xi})+\int_{s}^{T}f(r,X_{r}^{t,\xi}%
,Y_{r}^{t,\xi})dr+\int_{s}^{T}g(r,X_{r}^{t,\xi},Y_{r}^{t,\xi},Z_{r}^{t,\xi
})d\langle B\rangle_{r}\\
& -\int_{s}^{T}Z_{r}^{t,\xi}dB_{r}-(K_{T}^{t,\xi}-K_{s}^{t,\xi}),
\end{array}
\label{e3-2}%
\end{equation}
where $b$, $h$, $\sigma:[0,T]\times \mathbb{R}\rightarrow \mathbb{R}$,
$\varphi:\mathbb{R}\rightarrow \mathbb{R}$, $f:[0,T]\times \mathbb{R}%
^{2}\rightarrow \mathbb{R}$, $g:[0,T]\times \mathbb{R}^{3}\rightarrow \mathbb{R}$
satisfy the following conditions:

\begin{description}
\item[(A1)] $b$, $h$, $\sigma$, $f$, $g$ are continuous in $(s,x,y,z)$.

\item[(A2)] There exist a constant $L_{1}>0$ and a positive integer $m$ such
that for any $s\in \lbrack0,T]$, $x$, $x^{\prime}$, $y$, $y^{\prime}$, $z$,
$z^{\prime}\in \mathbb{R}$,%
\[%
\begin{array}
[c]{l}%
|b(s,x)-b(s,x^{\prime})|+|h(s,x)-h(s,x^{\prime})|+|\sigma(s,x)-\sigma
(s,x^{\prime})|\leq L_{1}|x-x^{\prime}|,\\
|\varphi(x)-\varphi(x^{\prime})|\leq L_{1}(1+|x|^{m}+|x^{\prime}%
|^{m})|x-x^{\prime}|,\\
|f(s,x,y)-f(s,x^{\prime},y^{\prime})|+|g(s,x,y,z)-g(s,x^{\prime},y^{\prime
},z^{\prime})|\\
\leq L_{1}[(1+|x|^{m}+|x^{\prime}|^{m})|x-x^{\prime}|+|y-y^{\prime
}|+|z-z^{\prime}|].
\end{array}
\]

\end{description}

Under the assumptions (A1) and (A2), for each $p\geq2$, SDE (\ref{e3-1}) has a
unique solution $(X_{s}^{t,\xi})_{s\in \lbrack t,T]}\in S_{G}^{p}(t,T)$ and
$G$-BSDE (\ref{e3-2}) has a unique $L^{p}$-solution $(Y_{s}^{t,\xi}%
,Z_{s}^{t,\xi},K_{s}^{t,\xi})_{s\in \lbrack t,T]}$ with $K_{t}^{t,\xi}=0$. The
following standard estimates of SDE can be found in Chapter 5 in Peng
\cite{P2019}.

\begin{proposition}
\label{pro3-1}Suppose that (A1) and (A2) hold. Let $\xi$, $\xi^{\prime}\in
\cap_{p\geq2}L_{G}^{p}(\Omega_{t})$ with $t<T$. Then, for each $p\geq2$ and
$\delta \in \lbrack0,T-t]$, we have%
\[
\mathbb{\hat{E}}_{t}\left[  \left \vert X_{t+\delta}^{t,\xi}-X_{t+\delta
}^{t,\xi^{\prime}}\right \vert ^{p}\right]  \leq C|\xi-\xi^{\prime}|^{p}\text{
and }\mathbb{\hat{E}}_{t}\left[  \left \vert X_{t+\delta}^{t,\xi}\right \vert
^{p}\right]  \leq C(1+|\xi|^{p}),
\]
where the constant $C>0$ depends on $L_{1}$, $\bar{\sigma}$, $p$ and $T$.
\end{proposition}

Set $\xi=x\in \mathbb{R}$, define%
\begin{equation}
u(t,x)=Y_{t}^{t,x}\text{ for }(t,x)\in \lbrack0,T]\times \mathbb{R}.
\label{e3-3}%
\end{equation}
Since $(B_{t+r}-B_{t})_{r\geq0}$ is still a $G$-Brownian motion, we have
$Y_{t}^{t,x}\in \mathbb{R}$.

\begin{proposition}
\label{pro3-2}Suppose that (A1) and (A2) hold. Then

\begin{description}
\item[(i)] For each $(t,x)\in \lbrack0,T)\times \mathbb{R}$, we have
$Y_{s}^{t,x}=u(s,X_{s}^{t,x})$ for $s\in \lbrack t,T]$.

\item[(ii)] $u(\cdot,\cdot)$ is the unique viscosity solution of the following
fully nonlinear PDE:%
\[
\left \{
\begin{array}
[c]{l}%
\partial_{t}u+G(\sigma^{2}(t,x)\partial_{xx}^{2}u+2h(t,x)\partial
_{x}u+2g(t,x,u,\sigma(t,x)\partial_{x}u))\\
+b(t,x)\partial_{x}u+f(t,x,u)=0,\\
u(T,x)=\varphi(x).
\end{array}
\right.
\]

\end{description}
\end{proposition}

\begin{proof}
The proof is the same as Theorems 4.4 and 4.5 in \cite{HJPS}, we omit it.
\end{proof}

In the following, we discuss the regularity properties of $u(\cdot,\cdot)$.
First, we study $\partial_{x}u(t,x)$. For each $(t,x)\in \lbrack0,T)\times
\mathbb{R}$ and $\Delta \in \lbrack-1,1]$, by Proposition \ref{pro3-1}, we have,
for each $p\geq2$,%
\begin{equation}
\sup_{s\in \lbrack t,T]}\mathbb{\hat{E}}\left[  \left \vert X_{s}^{t,x+\Delta
}-X_{s}^{t,x}\right \vert ^{p}\right]  \leq C|\Delta|^{p}\text{ and }\sup
_{s\in \lbrack t,T]}\mathbb{\hat{E}}\left[  \left \vert X_{s}^{t,x}\right \vert
^{p}\right]  \leq C(1+|x|^{p}), \label{e3-4}%
\end{equation}
where $C>0$ depends on $L_{1}$, $\bar{\sigma}$, $p$ and $T$. It follows from
Proposition \ref{pro2-1}, Theorem \ref{th1-1} and (\ref{e3-4}) that, for each
$p\geq2$,%
\begin{equation}
\mathbb{\hat{E}}\left[  \sup_{s\in \lbrack t,T]}\left \vert Y_{s}^{t,x+\Delta
}-Y_{s}^{t,x}\right \vert ^{p}\right]  \leq C(1+|x|^{mp})|\Delta|^{p},
\label{e3-5}%
\end{equation}
where $C>0$ depends on $L_{1}$, $\bar{\sigma}$, $p$ and $T$.

Let $\mathcal{P}$ be a weakly compact and convex set of probability measures
on $(\Omega_{T},\mathcal{B}(\Omega_{T}))$ satisfying%
\[
\mathbb{\hat{E}}\left[  \xi \right]  =\sup_{P\in \mathcal{P}}E_{P}[\xi]\text{
for }\xi \in L_{G}^{1}(\Omega_{T}).
\]
For each $(t,x)\in \lbrack0,T)\times \mathbb{R}$, set%
\[
\mathcal{P}_{t,x}=\{P\in \mathcal{P}:E_{P}[K_{T}^{t,x}]=0\}.
\]
Similar to the proof of Proposition \ref{pro2-3}, we obtain that, for each
$p\geq2$,%
\begin{equation}
E_{P}\left[  \left(  \int_{t}^{T}\left \vert Z_{s}^{t,x+\Delta}-Z_{s}%
^{t,x}\right \vert ^{2}d\langle B\rangle_{s}\right)  ^{p/2}+\left \vert
K_{T}^{t,x+\Delta}\right \vert ^{p}\right]  \leq C(1+|x|^{mp})|\Delta
|^{p}\text{ for }P\in \mathcal{P}_{t,x}, \label{e3-6}%
\end{equation}%
\begin{equation}
E_{P^{\Delta}}\left[  \left(  \int_{t}^{T}\left \vert Z_{s}^{t,x+\Delta}%
-Z_{s}^{t,x}\right \vert ^{2}d\langle B\rangle_{s}\right)  ^{p/2}+\left \vert
K_{T}^{t,x}\right \vert ^{p}\right]  \leq C(1+|x|^{mp})|\Delta|^{p}\text{ for
}P^{\Delta}\in \mathcal{P}_{t,x+\Delta}. \label{e3-7}%
\end{equation}

In order to obtain $\partial_{x}u(t,x)$, we need the following assumption.

\begin{description}
\item[(A3)] $b_{x}^{\prime}$, $h_{x}^{\prime}$, $\sigma_{x}^{\prime}$,
$\varphi^{\prime}$, $f_{x}^{\prime}$, $f_{y}^{\prime}$, $g_{x}^{\prime}$,
$g_{y}^{\prime}$, $g_{z}^{\prime}$ are continuous in $(s,x,y,z)$.
\end{description}

\begin{remark}
Under the assumptions (A2) and (A3), we can easily deduce that, for any
$s\in \lbrack0,T]$, $x$, $y$, $z\in \mathbb{R}$,%
\[%
\begin{array}
[c]{l}%
|b_{x}^{\prime}(s,x)|+|h_{x}^{\prime}(s,x)|+|\sigma_{x}^{\prime}(s,x)|\leq
L_{1}\text{, }|\varphi^{\prime}(x)|\leq L_{1}(1+2|x|^{m})\text{, }%
|g_{z}^{\prime}(s,x,y,z)|\leq L_{1},\\
|f_{x}^{\prime}(s,x,y)|+|g_{x}^{\prime}(s,x,y,z)|\leq L_{1}(1+2|x|^{m})\text{,
}|f_{y}^{\prime}(s,x,y)|+|g_{y}^{\prime}(s,x,y,z)|\leq L_{1}.
\end{array}
\]

\end{remark}

\begin{lemma}
\label{le3-3}Suppose that (A1)-(A3) hold. Then, for each $(t,x)\in
\lbrack0,T)\times \mathbb{R}$ and $p\geq2$, we have%
\begin{equation}
\lim_{\Delta \rightarrow0}\sup_{s\in \lbrack t,T]}\mathbb{\hat{E}}\left[
\left \vert \frac{X_{s}^{t,x+\Delta}-X_{s}^{t,x}}{\Delta}-\hat{X}_{s}%
^{t,x}\right \vert ^{p}\right]  =0, \label{e3-8}%
\end{equation}
where $(\hat{X}_{s}^{t,x})_{s\in \lbrack t,T]}$ is the solution of the
following $G$-SDE:%
\begin{equation}
d\hat{X}_{s}^{t,x}=b_{x}^{\prime}(s,X_{s}^{t,x})\hat{X}_{s}^{t,x}%
ds+h_{x}^{\prime}(s,X_{s}^{t,x})\hat{X}_{s}^{t,x}d\langle B\rangle_{s}%
+\sigma_{x}^{\prime}(s,X_{s}^{t,x})\hat{X}_{s}^{t,x}dB_{s}\text{, }\hat{X}%
_{t}^{t,x}=1. \label{e3-9}%
\end{equation}

\end{lemma}

\begin{proof}
Set $\hat{X}_{s}^{\Delta}=$ $X_{s}^{t,x+\Delta}-X_{s}^{t,x}$, $\tilde{X}%
_{s}^{\Delta}=$ $\hat{X}_{s}^{\Delta}-\hat{X}_{s}^{t,x}\Delta$ for
$s\in \lbrack t,T]$, we have%
\[
d\tilde{X}_{s}^{\Delta}=(b_{x}^{\prime}(s)\tilde{X}_{s}^{\Delta}+\tilde
{b}(s))ds+(h_{x}^{\prime}(s)\tilde{X}_{s}^{\Delta}+\tilde{h}(s))d\langle
B\rangle_{s}+(\sigma_{x}^{\prime}(s)\tilde{X}_{s}^{\Delta}+\tilde{\sigma
}(s))dB_{s}\text{, }\tilde{X}_{t}^{\Delta}=1,
\]
where $b_{x}^{\prime}(s)=b_{x}^{\prime}(s,X_{s}^{t,x})$,%
\begin{align*}
\tilde{b}(s)  &  =b(s,X_{s}^{t,x+\Delta})-b(s,X_{s}^{t,x})-b_{x}^{\prime
}(s,X_{s}^{t,x})\hat{X}_{s}^{\Delta}\\
&  =\hat{X}_{s}^{\Delta}\int_{0}^{1}\left[  b_{x}^{\prime}(s,X_{s}%
^{t,x}+\theta \hat{X}_{s}^{\Delta})-b_{x}^{\prime}(s,X_{s}^{t,x})\right]
d\theta,
\end{align*}
similar for $h_{x}^{\prime}(s)$, $\tilde{h}(s)$, $\sigma_{x}^{\prime}(s)$ and
$\tilde{\sigma}(s)$. By standard estimates of SDE, we get%
\begin{align*}
\sup_{s\in \lbrack t,T]}\mathbb{\hat{E}}\left[  \left \vert \tilde{X}%
_{s}^{\Delta}\right \vert ^{p}\right]   &  \leq C\mathbb{\hat{E}}\left[
\left(  \int_{t}^{T}|\tilde{b}(s)|ds\right)  ^{p}+\left(  \int_{t}^{T}%
|\tilde{h}(s)|d\langle B\rangle_{s}\right)  ^{p}+\left(  \int_{t}^{T}%
|\tilde{\sigma}(s)|^{2}d\langle B\rangle_{s}\right)  ^{p/2}\right] \\
&  \leq C\int_{t}^{T}\mathbb{\hat{E}}[|\tilde{b}(s)|^{p}+|\tilde{h}%
(s)|^{p}+|\tilde{\sigma}(s)|^{p}]ds,
\end{align*}
where $C>0$ depends on $L_{1}$, $\bar{\sigma}$, $p$ and $T$. By (\ref{e3-4})
and H\"{o}lder's inequality, we obtain%
\begin{equation}
\mathbb{\hat{E}}[|\tilde{b}(s)|^{p}]\leq C|\Delta|^{p}\left(  \mathbb{\hat{E}%
}\left[  \left(  \int_{0}^{1}|b_{x}^{\prime}(s,X_{s}^{t,x}+\theta \hat{X}%
_{s}^{\Delta})-b_{x}^{\prime}(s,X_{s}^{t,x})|d\theta \right)  ^{2p}\right]
\right)  ^{1/2}, \label{e3-10}%
\end{equation}
where $C>0$ depends on $L_{1}$, $\bar{\sigma}$, $p$ and $T$. For each $N>0$
and $\varepsilon>0$, define%
\begin{equation}
\omega_{N}(\varepsilon)=\sup \{|b_{x}^{\prime}(r,x_{1})-b_{x}^{\prime}%
(r,x_{2})|:r\in \lbrack0,T],\text{ }|x_{1}|\leq N,\text{ }|x_{1}-x_{2}%
|\leq \varepsilon \}. \label{e3-11}%
\end{equation}
Under assumption (A3), we know that $\omega_{N}(\varepsilon)\rightarrow0$ as
$\varepsilon \downarrow0$. Noting that%
\begin{equation}%
\begin{array}
[c]{l}%
|b_{x}^{\prime}(s,X_{s}^{t,x}+\theta \hat{X}_{s}^{\Delta})-b_{x}^{\prime
}(s,X_{s}^{t,x})|\\
\leq|b_{x}^{\prime}(s,X_{s}^{t,x}+\theta \hat{X}_{s}^{\Delta})-b_{x}^{\prime
}(s,X_{s}^{t,x})|I_{\{|\hat{X}_{s}^{\Delta}|\leq \varepsilon \}}+2L_{1}%
I_{\{|\hat{X}_{s}^{\Delta}|>\varepsilon \}}\\
\leq|b_{x}^{\prime}(s,X_{s}^{t,x}+\theta \hat{X}_{s}^{\Delta})-b_{x}^{\prime
}(s,X_{s}^{t,x})|I_{\{|\hat{X}_{s}^{\Delta}|\leq \varepsilon,|X_{s}^{t,x}|\leq
N\}}+2L_{1}I_{\{|X_{s}^{t,x}|>N\}}+2L_{1}I_{\{|\hat{X}_{s}^{\Delta
}|>\varepsilon \}}\\
\leq \omega_{N}(\varepsilon)+2L_{1}(|X_{s}^{t,x}|/N+|\hat{X}_{s}^{\Delta
}|/\varepsilon),
\end{array}
\label{e3-12}%
\end{equation}
we obtain by (\ref{e3-4}), (\ref{e3-10}) and (\ref{e3-12}) that%
\[
\mathbb{\hat{E}}[|\tilde{b}(s)|^{p}]\leq C|\Delta|^{p}\left(  |\omega
_{N}(\varepsilon)|^{p}+\frac{1+|x|^{p}}{N^{p}}+\frac{|\Delta|^{p}}%
{\varepsilon^{p}}\right)  ,
\]
where $C>0$ depends on $L_{1}$, $\bar{\sigma}$, $p$ and $T$. Thus%
\[
\underset{\Delta \rightarrow0}{\lim \sup}\frac{1}{|\Delta|^{p}}\int_{t}%
^{T}\mathbb{\hat{E}}[|\tilde{b}(s)|^{p}]ds\leq C\left(  |\omega_{N}%
(\varepsilon)|^{p}+\frac{1+|x|^{p}}{N^{p}}\right)  ,
\]
which implies $\lim_{\Delta \rightarrow0}\frac{1}{|\Delta|^{p}}\int_{t}%
^{T}\mathbb{\hat{E}}[|\tilde{b}(s)|^{p}]ds=0$ by letting $\varepsilon
\downarrow0$ first and then $N\rightarrow \infty$. Similarly, we can obtain%
\[
\lim_{\Delta \rightarrow0}\frac{1}{|\Delta|^{p}}\int_{t}^{T}\mathbb{\hat{E}%
}[|\tilde{h}(s)|^{p}+|\tilde{\sigma}(s)|^{p}]ds=0,
\]
which implies the desired result.
\end{proof}

\begin{theorem}
\label{th3-4}Suppose that (A1)-(A3) hold. Then, for each $(t,x)\in
\lbrack0,T)\times \mathbb{R}$, we have%
\begin{equation}
\partial_{x+}u(t,x)=\sup_{P\in \mathcal{P}_{t,x}}E_{P}\left[  \varphi^{\prime
}(X_{T}^{t,x})\hat{X}_{T}^{t,x}\Gamma_{T}^{t,x}+\int_{t}^{T}f_{x}^{\prime
}(s)\hat{X}_{s}^{t,x}\Gamma_{s}^{t,x}ds+\int_{t}^{T}g_{x}^{\prime}(s)\hat
{X}_{s}^{t,x}\Gamma_{s}^{t,x}d\langle B\rangle_{s}\right]  , \label{e3-13}%
\end{equation}%
\begin{equation}
\partial_{x-}u(t,x)=\inf_{P\in \mathcal{P}_{t,x}}E_{P}\left[  \varphi^{\prime
}(X_{T}^{t,x})\hat{X}_{T}^{t,x}\Gamma_{T}^{t,x}+\int_{t}^{T}f_{x}^{\prime
}(s)\hat{X}_{s}^{t,x}\Gamma_{s}^{t,x}ds+\int_{t}^{T}g_{x}^{\prime}(s)\hat
{X}_{s}^{t,x}\Gamma_{s}^{t,x}d\langle B\rangle_{s}\right]  , \label{e3-14}%
\end{equation}
where $(\hat{X}_{s}^{t,x})_{s\in \lbrack t,T]}$ satisfies (\ref{e3-9}),
$(\Gamma_{s}^{t,x})_{s\in \lbrack t,T]}$ satisfies the following $G$-SDE:%
\begin{equation}
d\Gamma_{s}^{t,x}=f_{y}^{\prime}(s)\Gamma_{s}^{t,x}ds+g_{y}^{\prime}%
(s)\Gamma_{s}^{t,x}d\langle B\rangle_{s}+g_{z}^{\prime}(s)\Gamma_{s}%
^{t,x}dB_{s},\text{ }\Gamma_{t}^{t,x}=1, \label{new-e3-14}%
\end{equation}
$g_{x}^{\prime}(s)=g_{x}^{\prime}(s,X_{s}^{t,x},Y_{s}^{t,x},Z_{s}^{t,x})$,
similar for $g_{y}^{\prime}(s)$, $g_{z}^{\prime}(s)$, $f_{x}^{\prime}(s)$ and
$f_{y}^{\prime}(s)$.
\end{theorem}

\begin{proof}
Set $\hat{X}_{s}^{\Delta}=$ $X_{s}^{t,x+\Delta}-X_{s}^{t,x}$, $\hat{Y}%
_{s}^{\Delta}=Y_{s}^{t,x+\Delta}-Y_{s}^{t,x}$, $\hat{Z}_{s}^{\Delta}%
=Z_{s}^{t,x+\Delta}-Z_{s}^{t,x}$ for $\Delta>0$ and $s\in \lbrack t,T]$. For
each $P\in \mathcal{P}_{t,x}$, we have%
\[%
\begin{array}
[c]{rl}%
\hat{Y}_{s}^{\Delta}= & \varphi^{\prime}(X_{T}^{t,x})\hat{X}_{T}^{t,x}%
\Delta+\tilde{\varphi}(T)+\int_{s}^{T}[f_{x}^{\prime}(r)\hat{X}_{r}%
^{t,x}\Delta+f_{y}^{\prime}(r)\hat{Y}_{r}^{\Delta}+\tilde{f}(r)]dr\\
& +\int_{s}^{T}[g_{x}^{\prime}(r)\hat{X}_{r}^{t,x}\Delta+g_{y}^{\prime}%
(r)\hat{Y}_{r}^{\Delta}+g_{z}^{\prime}(r)\hat{Z}_{r}^{\Delta}+\tilde
{g}(r)]d\langle B\rangle_{r}-\int_{s}^{T}\hat{Z}_{r}^{\Delta}dB_{r}-\int
_{s}^{T}dK_{r}^{t,x+\Delta},
\end{array}
\]
where $\tilde{g}(r)=g(r,X_{r}^{t,x+\Delta},Y_{r}^{t,x+\Delta},Z_{r}%
^{t,x+\Delta})-g(r,X_{r}^{t,x},Y_{r}^{t,x},Z_{r}^{t,x})-g_{x}^{\prime}%
(r)\hat{X}_{r}^{t,x}\Delta-g_{y}^{\prime}(r)\hat{Y}_{r}^{\Delta}-g_{z}%
^{\prime}(r)\hat{Z}_{r}^{\Delta}$, similar for $\tilde{\varphi}(T)$ and
$\tilde{f}(r)$. Applying It\^{o}'s formula to $\hat{Y}_{s}^{\Delta}\Gamma
_{s}^{t,x}$ on $[t,T]$ under $P$, we obtain%
\begin{equation}%
\begin{array}
[c]{rl}%
\Delta^{-1}\hat{Y}_{t}^{\Delta}= & E_{P}\left[  \varphi^{\prime}(X_{T}%
^{t,x})\hat{X}_{T}^{t,x}\Gamma_{T}^{t,x}+\int_{t}^{T}f_{x}^{\prime}(s)\hat
{X}_{s}^{t,x}\Gamma_{s}^{t,x}ds+\int_{t}^{T}g_{x}^{\prime}(s)\hat{X}_{s}%
^{t,x}\Gamma_{s}^{t,x}d\langle B\rangle_{s}\right] \\
& +\Delta^{-1}E_{P}\left[  \tilde{\varphi}(T)\Gamma_{T}^{t,x}+\int_{t}%
^{T}\tilde{f}(s)\Gamma_{s}^{t,x}ds+\int_{t}^{T}\tilde{g}(s)\Gamma_{s}%
^{t,x}d\langle B\rangle_{s}-\int_{t}^{T}\Gamma_{s}^{t,x}dK_{s}^{t,x+\Delta
}\right]  .
\end{array}
\label{e3-15}%
\end{equation}
Noting that $\tilde{\varphi}(T)=\varphi^{\prime}(X_{T}^{t,x})(\hat{X}%
_{T}^{\Delta}-\hat{X}_{T}^{t,x}\Delta)+\hat{X}_{T}^{\Delta}\int_{0}%
^{1}[\varphi^{\prime}(X_{T}^{t,x}+\theta \hat{X}_{T}^{\Delta})-\varphi^{\prime
}(X_{T}^{t,x})]d\theta$, similar for $\tilde{f}(s)$ and $\tilde{g}(s)$, by
(\ref{e3-4}), (\ref{e3-5}), (\ref{e3-6}), (\ref{e3-8}) and using the method in
(\ref{e3-12}), we get%
\begin{equation}
\lim_{\Delta \downarrow0}\Delta^{-1}E_{P}\left[  \tilde{\varphi}(T)\Gamma
_{T}^{t,x}+\int_{t}^{T}\tilde{f}(s)\Gamma_{s}^{t,x}ds+\int_{t}^{T}\tilde
{g}(s)\Gamma_{s}^{t,x}d\langle B\rangle_{s}\right]  =0. \label{e3-16}%
\end{equation}
Since $\Delta>0$, $\Gamma_{s}^{t,x}\geq0$ and $dK_{s}^{t,x+\Delta}\leq0$, we
deduce by (\ref{e3-15}) and (\ref{e3-16}) that%
\begin{equation}
\underset{\Delta \downarrow0}{\lim \inf}\frac{\hat{Y}_{t}^{\Delta}}{\Delta}%
\geq \sup_{P\in \mathcal{P}_{t,x}}E_{P}\left[  \varphi^{\prime}(X_{T}^{t,x}%
)\hat{X}_{T}^{t,x}\Gamma_{T}^{t,x}+\int_{t}^{T}f_{x}^{\prime}(s)\hat{X}%
_{s}^{t,x}\Gamma_{s}^{t,x}ds+\int_{t}^{T}g_{x}^{\prime}(s)\hat{X}_{s}%
^{t,x}\Gamma_{s}^{t,x}d\langle B\rangle_{s}\right]  . \label{e3-17}%
\end{equation}

For each $P^{\Delta}\in \mathcal{P}_{t,x+\Delta}$ for $\Delta>0$, similar to
(\ref{e3-15}), we have%
\begin{equation}%
\begin{array}
[c]{rl}%
\Delta^{-1}\hat{Y}_{t}^{\Delta}= & E_{P^{\Delta}}\left[  \varphi^{\prime
}(X_{T}^{t,x})\hat{X}_{T}^{t,x}\Gamma_{T}^{t,x}+\int_{t}^{T}f_{x}^{\prime
}(s)\hat{X}_{s}^{t,x}\Gamma_{s}^{t,x}ds+\int_{t}^{T}g_{x}^{\prime}(s)\hat
{X}_{s}^{t,x}\Gamma_{s}^{t,x}d\langle B\rangle_{s}\right] \\
& +\Delta^{-1}E_{P^{\Delta}}\left[  \tilde{\varphi}(T)\Gamma_{T}^{t,x}%
+\int_{t}^{T}\tilde{f}(s)\Gamma_{s}^{t,x}ds+\int_{t}^{T}\tilde{g}(s)\Gamma
_{s}^{t,x}d\langle B\rangle_{s}+\int_{t}^{T}\Gamma_{s}^{t,x}dK_{s}%
^{t,x}\right]  .
\end{array}
\label{e3-18}%
\end{equation}
Similar to (\ref{e3-16}), we get%
\begin{equation}
\lim_{\Delta \downarrow0}\Delta^{-1}E_{P^{\Delta}}\left[  \tilde{\varphi
}(T)\Gamma_{T}^{t,x}+\int_{t}^{T}\tilde{f}(s)\Gamma_{s}^{t,x}ds+\int_{t}%
^{T}\tilde{g}(s)\Gamma_{s}^{t,x}d\langle B\rangle_{s}\right]  =0.
\label{e3-19}%
\end{equation}
Since $\mathcal{P}$ is weakly compact, for any sequence $\Delta_{j}%
\downarrow0$, we can find a subsequence $\Delta_{i}\downarrow0$ such that
$P^{\Delta_{i}}$ converges weakly to $P^{\ast}\in \mathcal{P}$. By Proposition
\ref{pro2-1} and (\ref{e3-5}), we have $\mathbb{\hat{E}}\left[  |K_{T}%
^{t,x+\Delta}-K_{T}^{t,x}|\right]  \rightarrow0$ as $\Delta \downarrow0$. Due
to%
\[
|E_{P^{\ast}}[K_{T}^{t,x}]|=|E_{P^{\ast}}[K_{T}^{t,x}]-E_{P^{\Delta_{i}}%
}[K_{T}^{t,x+\Delta_{i}}]|\leq|E_{P^{\ast}}[K_{T}^{t,x}]-E_{P^{\Delta_{i}}%
}[K_{T}^{t,x}]|+\mathbb{\hat{E}}\left[  |K_{T}^{t,x+\Delta_{i}}-K_{T}%
^{t,x}|\right]
\]
and $E_{P^{\Delta_{i}}}[K_{T}^{t,x}]\rightarrow E_{P^{\ast}}[K_{T}^{t,x}]$ as
$\Delta_{i}\downarrow0$, we get $E_{P^{\ast}}[K_{T}^{t,x}]=0$, which implies
$P^{\ast}\in \mathcal{P}_{t,x}$. Noting that $\int_{t}^{T}\Gamma_{s}%
^{t,x}dK_{s}^{t,x}\leq0$ and%
\[
\varphi^{\prime}(X_{T}^{t,x})\hat{X}_{T}^{t,x}\Gamma_{T}^{t,x}+\int_{t}%
^{T}f_{x}^{\prime}(s)\hat{X}_{s}^{t,x}\Gamma_{s}^{t,x}ds+\int_{t}^{T}%
g_{x}^{\prime}(s)\hat{X}_{s}^{t,x}\Gamma_{s}^{t,x}d\langle B\rangle_{s}\in
L_{G}^{1}(\Omega_{T}),
\]
we deduce by (\ref{e3-18}) and (\ref{e3-19}) that%
\begin{equation}
\underset{\Delta \downarrow0}{\lim \sup}\frac{\hat{Y}_{t}^{\Delta}}{\Delta}%
\leq \sup_{P\in \mathcal{P}_{t,x}}E_{P}\left[  \varphi^{\prime}(X_{T}^{t,x}%
)\hat{X}_{T}^{t,x}\Gamma_{T}^{t,x}+\int_{t}^{T}f_{x}^{\prime}(s)\hat{X}%
_{s}^{t,x}\Gamma_{s}^{t,x}ds+\int_{t}^{T}g_{x}^{\prime}(s)\hat{X}_{s}%
^{t,x}\Gamma_{s}^{t,x}d\langle B\rangle_{s}\right]  . \label{e3-20}%
\end{equation}
Thus we obtain (\ref{e3-13}) by (\ref{e3-17}) and (\ref{e3-20}). Similarly, we
can get (\ref{e3-14}).
\end{proof}

Now we study $\partial_{t}u(t,x)$. For each $(t,x)\in(0,T)\times \mathbb{R}$
and $|\Delta|<t\wedge(T-t)$, noting that%
\[
\sqrt{\frac{T-t}{T-t-\Delta}}\left(  B_{\frac{T-t-\Delta}{T-t}(s-t)+t+\Delta
}-B_{t+\Delta}\right)  _{s\in \lbrack t,T]}%
\]
and $(B_{s}-B_{t})_{s\in \lbrack t,T]}$ have the same distribution, we obtain
$u(t+\Delta,x)=\bar{Y}_{t}^{t,x,\Delta}$, where $(\bar{X}^{t,x,\Delta}$,
$\bar{Y}^{t,x,\Delta}$, $\bar{Z}^{t,x,\Delta}$, $\bar{K}^{t,x,\Delta})$
satisfies the following $G$-FBSDE:%
\[%
\begin{array}
[c]{rl}%
\bar{X}_{s}^{t,x,\Delta}= & x+\frac{T-t-\Delta}{T-t}\left[  \int_{t}%
^{s}b(r+\frac{T-r}{T-t}\Delta,\bar{X}_{r}^{t,x,\Delta})dr+\int_{t}%
^{s}h(r+\frac{T-r}{T-t}\Delta,\bar{X}_{r}^{t,x,\Delta})d\langle B\rangle
_{r}\right] \\
& +\sqrt{\frac{T-t-\Delta}{T-t}}\int_{t}^{s}\sigma(r+\frac{T-r}{T-t}%
\Delta,\bar{X}_{r}^{t,x,\Delta})dB_{r},
\end{array}
\]%
\[%
\begin{array}
[c]{rl}%
\bar{Y}_{s}^{t,x,\Delta}= & \varphi(\bar{X}_{T}^{t,x,\Delta})+\frac
{T-t-\Delta}{T-t}\int_{s}^{T}g(r+\frac{T-r}{T-t}\Delta,\bar{X}_{r}%
^{t,x,\Delta},\bar{Y}_{r}^{t,x,\Delta},\sqrt{\frac{T-t}{T-t-\Delta}}\bar
{Z}_{r}^{t,x,\Delta})d\langle B\rangle_{r}\\
& +\frac{T-t-\Delta}{T-t}\int_{s}^{T}f(r+\frac{T-r}{T-t}\Delta,\bar{X}%
_{r}^{t,x,\Delta},\bar{Y}_{r}^{t,x,\Delta})dr-\int_{s}^{T}\bar{Z}%
_{r}^{t,x,\Delta}dB_{r}-(\bar{K}_{T}^{t,x,\Delta}-\bar{K}_{s}^{t,x,\Delta}).
\end{array}
\]

In order to obtain $\partial_{t}u(t,x)$, we need the following assumption.

\begin{description}
\item[(A4)] $b_{t}^{\prime}$, $h_{t}^{\prime}$, $\sigma_{t}^{\prime}$,
$f_{t}^{\prime}$, $g_{t}^{\prime}$ are continuous in $(s,x,y,z)$, and there
exist a constant $L_{2}>0$ and a positive integer $m_{1}$ such that for any
$s\in \lbrack0,T]$, $x$, $y$, $z\in \mathbb{R}$,%
\[
|b_{t}^{\prime}(s,x)|+|h_{t}^{\prime}(s,x)|+|\sigma_{t}^{\prime}%
(s,x)|+|f_{t}^{\prime}(s,x,y)|+|g_{t}^{\prime}(s,x,y,z)|\leq L_{2}%
(1+|x|^{m_{1}}+|y|^{m_{1}}+|z|^{2}).
\]

\end{description}

\begin{lemma}
\label{le3-5}Suppose that (A1)-(A4) hold. Then, for each $(t,x)\in
(0,T)\times \mathbb{R}$ and $p\geq2$, we have%
\[
\lim_{\Delta \rightarrow0}\sup_{s\in \lbrack t,T]}\mathbb{\hat{E}}\left[
\left \vert \frac{\bar{X}_{s}^{t,x,\Delta}-X_{s}^{t,x}}{\Delta}-\bar{X}%
_{s}^{t,x}\right \vert ^{p}\right]  =0,
\]
where $(\bar{X}_{s}^{t,x})_{s\in \lbrack t,T]}$ is the solution of the
following $G$-SDE:%
\begin{equation}%
\begin{array}
[c]{rl}%
\bar{X}_{s}^{t,x}= & \int_{t}^{s}\left[  b_{x}^{\prime}(r,X_{r}^{t,x})\bar
{X}_{r}^{t,x}+\frac{T-r}{T-t}b_{t}^{\prime}(r,X_{r}^{t,x})-\frac{1}%
{T-t}b(r,X_{r}^{t,x})\right]  dr\\
& +\int_{t}^{s}\left[  h_{x}^{\prime}(r,X_{r}^{t,x})\bar{X}_{r}^{t,x}%
+\frac{T-r}{T-t}h_{t}^{\prime}(r,X_{r}^{t,x})-\frac{1}{T-t}h(r,X_{r}%
^{t,x})\right]  d\langle B\rangle_{r}\\
& +\int_{t}^{s}\left[  \sigma_{x}^{\prime}(r,X_{r}^{t,x})\bar{X}_{r}%
^{t,x}+\frac{T-r}{T-t}\sigma_{t}^{\prime}(r,X_{r}^{t,x})-\frac{1}%
{2(T-t)}\sigma(r,X_{r}^{t,x})\right]  dB_{r}.
\end{array}
\label{e3-21}%
\end{equation}

\end{lemma}

\begin{proof}
The proof is similar to Lemma \ref{le3-3}, we omit it.
\end{proof}

\begin{theorem}
\label{th3-6}Suppose that (A1)-(A4) hold. Then, for each $(t,x)\in
(0,T)\times \mathbb{R}$, we have%
\begin{align*}
\partial_{t+}u(t,x)  &  =\sup_{P\in \mathcal{P}_{t,x}}E_{P}\left[
\varphi^{\prime}(X_{T}^{t,x})\bar{X}_{T}^{t,x}\Gamma_{T}^{t,x}+\int_{t}%
^{T}\left(  f_{x}^{\prime}(s)\bar{X}_{s}^{t,x}+\frac{T-s}{T-t}f_{t}^{\prime
}(s)-\frac{1}{T-t}f(s)\right)  \Gamma_{s}^{t,x}ds\right. \\
&  \left.  \text{ \  \  \  \ }+\int_{t}^{T}\left(  \frac{g_{z}^{\prime}%
(s)Z_{s}^{t,x}}{2(T-t)}+g_{x}^{\prime}(s)\bar{X}_{s}^{t,x}+\frac{T-s}%
{T-t}g_{t}^{\prime}(s)-\frac{1}{T-t}g(s)\right)  \Gamma_{s}^{t,x}d\langle
B\rangle_{s}\right]  ,
\end{align*}%
\begin{align*}
\partial_{t-}u(t,x)  &  =\inf_{P\in \mathcal{P}_{t,x}}E_{P}\left[
\varphi^{\prime}(X_{T}^{t,x})\bar{X}_{T}^{t,x}\Gamma_{T}^{t,x}+\int_{t}%
^{T}\left(  f_{x}^{\prime}(s)\bar{X}_{s}^{t,x}+\frac{T-s}{T-t}f_{t}^{\prime
}(s)-\frac{1}{T-t}f(s)\right)  \Gamma_{s}^{t,x}ds\right. \\
&  \left.  \text{ \  \  \  \ }+\int_{t}^{T}\left(  \frac{g_{z}^{\prime}%
(s)Z_{s}^{t,x}}{2(T-t)}+g_{x}^{\prime}(s)\bar{X}_{s}^{t,x}+\frac{T-s}%
{T-t}g_{t}^{\prime}(s)-\frac{1}{T-t}g(s)\right)  \Gamma_{s}^{t,x}d\langle
B\rangle_{s}\right]  ,
\end{align*}
where $(\bar{X}_{s}^{t,x})_{s\in \lbrack t,T]}$ satisfies (\ref{e3-21}),
$(\Gamma_{s}^{t,x})_{s\in \lbrack t,T]}$ satisfies (\ref{new-e3-14}),
$f_{t}^{\prime}(s)=f_{t}^{\prime}(s,X_{s}^{t,x},Y_{s}^{t,x})$, similar for
$f(s)$, $f_{x}^{\prime}(s)$, $g(s)$, $g_{x}^{\prime}(s)$, $g_{z}^{\prime}(s)$
and $g_{t}^{\prime}(s)$.
\end{theorem}

\begin{proof}
The proof is similar to Theorem \ref{th3-4}, we omit it.
\end{proof}

The following theorem gives the condition for $\partial_{x+}u(t,x)=\partial
_{x-}u(t,x)$.

\begin{theorem}
\label{th3-7}Suppose that (A1)-(A4) hold. If $\sigma(t,x)\not =0$ for some
$(t,x)\in(0,T)\times \mathbb{R}$, then $\partial_{x+}u(t,x)=\partial
_{x-}u(t,x)$.
\end{theorem}

\begin{proof}
We first sketch the properties of $u$, which is the same as in the proof of
Theorem 4.5 in \cite{HJPS}. By Propositions \ref{pro2-1} and \ref{pro3-1}, we
can get that, for $s\in \lbrack0,T]$, $x_{1}$, $x_{2}\in \mathbb{R}$, $p\geq2$,%
\begin{equation}
|u(s,x_{1})-u(s,x_{2})|\leq C(1+|x_{1}|^{m}+|x_{2}|^{m})|x_{1}-x_{2}|\text{,
}|u(s,x_{1})|\leq C(1+|x_{1}|^{m+1}), \label{e3-23}%
\end{equation}%
\begin{equation}
\mathbb{\hat{E}}\left[  \sup_{s\leq r\leq T}|Y_{r}^{s,x_{1}}|^{p}+\left(
\int_{s}^{T}|Z_{r}^{s,x_{1}}|^{2}d\langle B\rangle_{r}\right)  ^{p/2}%
+|K_{T}^{s,x_{1}}|^{p}\right]  \leq C(1+|x_{1}|^{(m+1)p}), \label{e3-24}%
\end{equation}
where $C>0$ depends on $L_{1}$, $\bar{\sigma}$, $p$ and $T$. For each $0\leq
t_{1}<t_{2}\leq T$ and $x_{1}\in \mathbb{R}$, by (i) of Proposition
\ref{pro3-2}, we know%
\begin{equation}
u(t_{1},x_{1})=\mathbb{\hat{E}}\left[  u(t_{2},X_{t_{2}}^{t_{1},x_{1}}%
)+\int_{t_{1}}^{t_{2}}f(r,X_{r}^{t_{1},x_{1}},Y_{r}^{t_{1},x_{1}}%
)dr+\int_{t_{1}}^{t_{2}}g(r,X_{r}^{t_{1},x_{1}},Y_{r}^{t_{1},x_{1}}%
,Z_{r}^{t_{1},x_{1}})d\langle B\rangle_{r}\right]  . \label{e3-22}%
\end{equation}
It follows from (\ref{e3-4}), (\ref{e3-23}), (\ref{e3-24}), (\ref{e3-22}) and
H\"{o}lder's inequality, we obtain%
\begin{equation}
|u(t_{1},x_{1})-u(t_{2},x_{1})|\leq C(1+|x_{1}|^{m+1})\sqrt{t_{2}-t_{1}},
\label{new-e3-25}%
\end{equation}
where $C>0$ depends on $L_{1}$, $\bar{\sigma}$ and $T$.

We then take $t_{1}=t-\delta$ with $\delta \in(0,t)$, $t_{2}=t$ and $x_{1}=x$
in (\ref{e3-22}). By Theorem \ref{th3-6}, we know that
\begin{equation}
\lim_{\delta \downarrow0}\delta^{-1}(u(t-\delta,x)-u(t,x))=-\partial
_{t-}u(t,x)\in \mathbb{R}. \label{e3-25}%
\end{equation}
In the following, we will prove that%
\begin{equation}
\mathbb{\hat{E}}\left[  |u(t,X_{t}^{t-\delta,x})-u(t,x+\sigma(t,x)(B_{t}%
-B_{t-\delta}))|\right]  \leq C\delta, \label{e3-26}%
\end{equation}%
\begin{equation}
\mathbb{\hat{E}}\left[  \int_{t-\delta}^{t}|f(r,X_{r}^{t-\delta,x}%
,Y_{r}^{t-\delta,x})|dr+\int_{t-\delta}^{t}|g(r,X_{r}^{t-\delta,x}%
,Y_{r}^{t-\delta,x},Z_{r}^{t-\delta,x})|d\langle B\rangle_{r}\right]  \leq
C\delta, \label{e3-27}%
\end{equation}%
\begin{equation}
\lim_{\delta \downarrow0}\delta^{-1}\mathbb{\hat{E}}\left[  u(t,x+\sigma
(t,x)(B_{t}-B_{t-\delta}))-u(t,x)\right]  =\infty \text{ if }\partial
_{x+}u(t,x)>\partial_{x-}u(t,x), \label{e3-28}%
\end{equation}
where the constant $C>0$ depends on $x$, $L_{1}$, $L_{2}$, $m$, $m_{1}$,
$\bar{\sigma}$ and $T$. If (\ref{e3-26}), (\ref{e3-27}) and (\ref{e3-28})
hold, we can get $\partial_{x+}u(t,x)=\partial_{x-}u(t,x)$ by (\ref{e3-25}).

Noting that%
\[
\mathbb{\hat{E}}\left[  \int_{t-\delta}^{t}|\sigma(r,X_{r}^{t-\delta
,x})-\sigma(t,x)|^{2}d\langle B\rangle_{r}\right]  \leq C\int_{t-\delta}%
^{t}\mathbb{\hat{E}}[|X_{r}^{t-\delta,x}-x|^{2}]dr+C\delta^{3}\leq C\delta
^{2},
\]
we get (\ref{e3-26}) by (\ref{e3-23}). By (i) of Proposition \ref{pro3-2}, we
know $Y_{r}^{t-\delta,x}=u(r,X_{r}^{t-\delta,x})$. Then we get%
\[%
\begin{array}
[c]{rl}%
Y_{s}^{t-\delta,x}-u(t,x)= & u(t,X_{t}^{t-\delta,x})-u(t,x)+\int_{s}%
^{t}g(r,X_{r}^{t-\delta,x},u(r,X_{r}^{t-\delta,x}),Z_{r}^{t-\delta,x})d\langle
B\rangle_{r}\\
& +\int_{s}^{t}f(r,X_{r}^{t-\delta,x},u(r,X_{r}^{t-\delta,x}))dr-\int_{s}%
^{t}Z_{r}^{t-\delta,x}dB_{r}-(K_{t}^{t-\delta,x}-K_{s}^{t-\delta,x}).
\end{array}
\]
By (\ref{e2-7}) in Proposition \ref{pro2-1}, (\ref{e3-23}) and
(\ref{new-e3-25}), we obtain%
\begin{align*}
\mathbb{\hat{E}}\left[  \int_{t-\delta}^{t}|Z_{r}^{t-\delta,x}|^{2}d\langle
B\rangle_{r}\right]   &  \leq C\mathbb{\hat{E}}\left[  \sup_{s\in \lbrack
t-\delta,t]}|u(s,X_{s}^{t-\delta,x})-u(t,x)|^{2}\right]  +C\delta^{2}\\
&  \leq C\mathbb{\hat{E}}\left[  \sup_{s\in \lbrack t-\delta,t]}|u(s,X_{s}%
^{t-\delta,x})-u(s,x)|^{2}\right]  +C\delta \\
&  \leq C\delta.
\end{align*}
Then we can easily get (\ref{e3-27}) by H\"{o}lder's inequality.

Now we prove (\ref{e3-28}). Set $\xi_{\delta}=\sigma(t,x)(B_{t}-B_{t-\delta}%
)$, we have%
\begin{align*}
\frac{u(t,x+\xi_{\delta})-u(t,x)}{\delta}=  &  \frac{[u(t,x+\xi_{\delta
})-u(t,x)-\partial_{x+}u(t,x)\xi_{\delta}]I_{\{ \xi_{\delta}>0\}}%
+\partial_{x+}u(t,x)\xi_{\delta}^{+}}{\delta}\\
&  +\frac{[u(t,x+\xi_{\delta})-u(t,x)-\partial_{x-}u(t,x)\xi_{\delta}]I_{\{
\xi_{\delta}<0\}}-\partial_{x-}u(t,x)\xi_{\delta}^{-}}{\delta}.
\end{align*}
If $\partial_{x+}u(t,x)>\partial_{x-}u(t,x)$, then there exists an $l>0$ such
that%
\[
|u(t,x+x^{\prime})-u(t,x)-\partial_{x+}u(t,x)x^{\prime}|\leq \frac{\gamma}%
{4}x^{\prime}\text{ for }x^{\prime}\in \lbrack0,l],
\]%
\[
|u(t,x+x^{\prime})-u(t,x)-\partial_{x-}u(t,x)x^{\prime}|\leq-\frac{\gamma}%
{4}x^{\prime}\text{ for }x^{\prime}\in \lbrack-l,0],
\]
where $\gamma=\partial_{x+}u(t,x)-\partial_{x-}u(t,x)$. Then, by
(\ref{e3-23}), we obtain%
\begin{align*}
&  \frac{|u(t,x+\xi_{\delta})-u(t,x)-\partial_{x+}u(t,x)\xi_{\delta}|I_{\{
\xi_{\delta}>0\}}}{\delta}\\
&  \leq C(1+|\xi_{\delta}|^{m})\frac{|\xi_{\delta}|}{\delta}I_{\{ \xi_{\delta
}>l\}}+\frac{\gamma}{4}\frac{\xi_{\delta}}{\delta}I_{\{0<\xi_{\delta}\leq
l\}}\\
&  \leq C(1+|\xi_{\delta}|^{m})\frac{|\xi_{\delta}|^{3}}{\delta l^{2}}%
+\frac{\gamma}{4}\frac{\xi_{\delta}^{+}}{\delta},
\end{align*}
where the constant $C>0$ depends on $x$, $L_{1}$, $m$, $\bar{\sigma}$ and $T$.
Similarly, we have%
\[
\frac{|u(t,x+\xi_{\delta})-u(t,x)-\partial_{x-}u(t,x)\xi_{\delta}|I_{\{
\xi_{\delta}<0\}}}{\delta}\leq C(1+|\xi_{\delta}|^{m})\frac{|\xi_{\delta}%
|^{3}}{\delta l^{2}}+\frac{\gamma}{4}\frac{\xi_{\delta}^{-}}{\delta}.
\]
Noting that $\partial_{x+}u(t,x)\xi_{\delta}^{+}-\partial_{x-}u(t,x)\xi
_{\delta}^{-}=\frac{\gamma}{2}|\xi_{\delta}|+\frac{1}{2}[\gamma+2\partial
_{x-}u(t,x)]\xi_{\delta}$ we get%
\[
\frac{u(t,x+\xi_{\delta})-u(t,x)}{\delta}\geq \frac{\gamma}{4}\frac
{|\xi_{\delta}|}{\delta}+\frac{1}{2}[\gamma+2\partial_{x-}u(t,x)]\frac
{\xi_{\delta}}{\delta}-2C(1+|\xi_{\delta}|^{m})\frac{|\xi_{\delta}|^{3}%
}{\delta l^{2}}.
\]
Since $\mathbb{\hat{E}}[\xi_{\delta}]=\mathbb{\hat{E}}[-\xi_{\delta}]=0$,
$\mathbb{\hat{E}}[|\xi_{\delta}|]=|\sigma(t,x)|\mathbb{\hat{E}}[|B_{1}%
|]\sqrt{\delta}$, $\mathbb{\hat{E}}[|\xi_{\delta}|^{6}]=|\sigma(t,x)|^{6}%
\mathbb{\hat{E}}[|B_{1}|^{6}]\delta^{3}$ and
\[
\delta^{-1}\mathbb{\hat{E}}\left[  u(t,x+\xi_{\delta})-u(t,x)\right]
\geq \delta^{-1}\left(  \frac{\gamma}{4}\mathbb{\hat{E}}\left[  |\xi_{\delta
}|\right]  -\frac{2C}{l^{2}}\sqrt{\mathbb{\hat{E}}[(1+|\xi_{\delta}|^{m}%
)^{2}]\mathbb{\hat{E}}[|\xi_{\delta}|^{6}]}\right)  ,
\]
we obtain (\ref{e3-28}). The proof is completed.
\end{proof}

Finally, we study $\partial_{xx}^{2}u(t,x)$. We need the following assumption.

\begin{description}
\item[(A5)] $b_{xx}^{\prime \prime}$, $h_{xx}^{\prime \prime}$, $\sigma
_{xx}^{\prime \prime}$, $f_{xx}^{\prime \prime}$, $f_{xy}^{\prime \prime}$,
$f_{yy}^{\prime \prime}$, $g_{xx}^{\prime \prime}$, $g_{xy}^{\prime \prime}$,
$g_{xz}^{\prime \prime}$, $g_{yy}^{\prime \prime}$, $g_{yz}^{\prime \prime}$,
$g_{zz}^{\prime \prime}$ are continuous in $(s,x,y,z)$ and bounded by a
constant $L_{3}>0$.
\end{description}

\begin{theorem}
Suppose that (A1)-(A3) and (A5) hold. Then, for each $(t,x)\in \lbrack
0,T)\times \mathbb{R}$, we have%
\begin{equation}
\Delta^{-1}\left[  \partial_{x-}u(t,x+\Delta)-\partial_{x+}u(t,x)\right]
\geq-C(1+|x|^{2m})\text{ for }\Delta \in(0,1], \label{e3-29}%
\end{equation}%
\begin{equation}
\Delta^{-1}\left[  \partial_{x+}u(t,x+\Delta)-\partial_{x-}u(t,x)\right]
\geq-C(1+|x|^{2m})\text{ for }\Delta \in \lbrack-1,0), \label{e3-30}%
\end{equation}
where the constant $C>0$ depends on $L_{1}$, $L_{3}$, $\bar{\sigma}$ and $T$.
\end{theorem}

\begin{proof}
By the definition of $\mathcal{P}_{t,x}$, it is easy to verify that
$\mathcal{P}_{t,x}$ is weakly compact. Then we can choose a $P\in
\mathcal{P}_{t,x}$ such that%
\[
\partial_{x+}u(t,x)=E_{P}\left[  \varphi^{\prime}(X_{T}^{t,x})\hat{X}%
_{T}^{t,x}\Gamma_{T}^{t,x}+\int_{t}^{T}f_{x}^{\prime}(s)\hat{X}_{s}%
^{t,x}\Gamma_{s}^{t,x}ds+\int_{t}^{T}g_{x}^{\prime}(s)\hat{X}_{s}^{t,x}%
\Gamma_{s}^{t,x}d\langle B\rangle_{s}\right]
\]
in (\ref{e3-13}). Using the same notations as in the proof of Theorem
\ref{th3-4}, for $\Delta \in(0,1]$, we get by (\ref{e3-15}) that%
\[
\hat{Y}_{t}^{\Delta}\geq \Delta \partial_{x+}u(t,x)+E_{P}\left[  \tilde{\varphi
}(T)\Gamma_{T}^{t,x}+\int_{t}^{T}\tilde{f}(s)\Gamma_{s}^{t,x}ds+\int_{t}%
^{T}\tilde{g}(s)\Gamma_{s}^{t,x}d\langle B\rangle_{s}\right]  .
\]
Under the assumption (A5), it is easy to check that%
\[
|\tilde{\varphi}(T)|\leq C(1+|X_{T}^{t,x}|^{m})|\hat{X}_{T}^{\Delta}-\hat
{X}_{T}^{t,x}\Delta|+C|\hat{X}_{T}^{\Delta}|^{2}\text{, }|\tilde{f}(s)|\leq
C(1+|X_{s}^{t,x}|^{m})|\hat{X}_{s}^{\Delta}-\hat{X}_{s}^{t,x}\Delta
|+C(|\hat{X}_{s}^{\Delta}|^{2}+|\hat{Y}_{s}^{\Delta}|^{2}),
\]%
\[
|\tilde{g}(s)|\leq C(1+|X_{s}^{t,x}|^{m})|\hat{X}_{s}^{\Delta}-\hat{X}%
_{s}^{t,x}\Delta|+C(|\hat{X}_{s}^{\Delta}|^{2}+|\hat{Y}_{s}^{\Delta}%
|^{2}+|\hat{Z}_{s}^{\Delta}|^{2}),
\]
where $C>0$ depends on $L_{1}$ and $L_{3}$. We can also get%
\[
\sup_{s\in \lbrack t,T]}\mathbb{\hat{E}}\left[  \left \vert \hat{X}_{s}^{\Delta
}-\hat{X}_{s}^{t,x}\Delta \right \vert ^{p}\right]  \leq C\Delta^{2p}%
\]
for $p\geq2$ in the proof of Lemma \ref{le3-3} by $|\tilde{b}(s)|+|\tilde
{h}(s)|+|\tilde{\sigma}(s)|\leq C|\hat{X}_{s}^{\Delta}|^{2}$, where $C>0$
depends on $L_{1}$, $L_{3}$, $\bar{\sigma}$, $p$ and $T$. It follows from
(\ref{e3-4}), (\ref{e3-5}) and (\ref{e3-6}) that%
\begin{equation}
\left \vert E_{P}\left[  \tilde{\varphi}(T)\Gamma_{T}^{t,x}+\int_{t}^{T}%
\tilde{f}(s)\Gamma_{s}^{t,x}ds+\int_{t}^{T}\tilde{g}(s)\Gamma_{s}%
^{t,x}d\langle B\rangle_{s}\right]  \right \vert \leq C(1+|x|^{2m})\Delta^{2},
\label{e3-32}%
\end{equation}
where $C>0$ depends on $L_{1}$, $L_{3}$, $\bar{\sigma}$ and $T$. Thus%
\begin{equation}
\hat{Y}_{t}^{\Delta}\geq \Delta \partial_{x+}u(t,x)-C(1+|x|^{2m})\Delta^{2}.
\label{e3-31}%
\end{equation}
We can also choose a $P^{\Delta}\in \mathcal{P}_{t,x+\Delta}$ such that%
\begin{align*}
\partial_{x-}u(t,x+\Delta)=  &  E_{P^{\Delta}}\left[  \varphi^{\prime}%
(X_{T}^{t,x+\Delta})\hat{X}_{T}^{t,x+\Delta}\Gamma_{T}^{t,x+\Delta}+\int
_{t}^{T}f_{x}^{\prime}(s,X_{s}^{t,x+\Delta},Y_{s}^{t,x+\Delta})\hat{X}%
_{s}^{t,x+\Delta}\Gamma_{s}^{t,x+\Delta}ds\right. \\
&  \left.  +\int_{t}^{T}g_{x}^{\prime}(s,X_{s}^{t,x+\Delta},Y_{s}^{t,x+\Delta
},Z_{s}^{t,x+\Delta})\hat{X}_{s}^{t,x+\Delta}\Gamma_{s}^{t,x+\Delta}d\langle
B\rangle_{s}\right]  .
\end{align*}
Applying It\^{o}'s formula to $\hat{Y}_{s}^{\Delta}\Gamma_{s}^{t,x+\Delta}$ on
$[t,T]$ under $P^{\Delta}$, similar to (\ref{e3-15}) and the analysis of
(\ref{e3-32}), we can get%
\begin{equation}
\hat{Y}_{t}^{\Delta}\leq \Delta \partial_{x-}u(t,x+\Delta)+C(1+|x|^{2m}%
)\Delta^{2}, \label{e3-33}%
\end{equation}
where $C>0$ depends on $L_{1}$, $L_{3}$, $\bar{\sigma}$ and $T$. Then we
obtain (\ref{e3-29}) by (\ref{e3-31}) and (\ref{e3-33}). Similarly, we can
deduce (\ref{e3-30}).
\end{proof}

\begin{remark}
We can get similar estimates under the assumption
\[
|b_{xx}^{\prime \prime}(s,x)|\leq L_{4}(1+|x|^{m_{2}})
\]
for positive constant $L_{4}$ and positive integer $m_{2}$, similar for the
second derivatives of $h$, $\sigma$, $f$ and $g$.
\end{remark}


\bigskip

\begin{thebibliography}{99}                                                                                               %


\bibitem {CIP}M.G. Crandall, H. Ishii, P.L. Lions, User's guide to viscosity
solutions of second order partial differential equations, Bull. Amer. Math.
Soc., 27 (1992), 1-67.

\bibitem {DHP11}L. Denis, M. Hu, S. Peng, Function spaces and capacity related
to a sublinear expectation: application to $G$-Brownian motion paths,
Potential Anal., 34 (2011), 139-161.

\bibitem {DenisMartini2006}L. Denis, C. Martini, A theoretical framework for
the pricing of contingent claims in the presence of model uncertainty, Ann.
Appl. Probab., 16 (2006), 827-852.

\bibitem {DK}L. Denis, K. Kervarec, Optimal investment under model uncertainty
in non-dominated models, SIAM J. Control Optim., 51 (2013), 1803-1822.

\bibitem {HJ0}M. Hu, S. Ji, Stochastic maximum principle for stochastic
recursive optimal control problem under volatility ambiguity, SIAM J. Control
Optim., 54 (2016), 918-945.

\bibitem {HJ1}M. Hu, S. Ji, Dynamic programming principle for stochastic
recursive optimal control problem driven by a G-Brownian motion, Stochastic
Process. Appl., 127 (2017), 107-134.

\bibitem {HJPS1}M. Hu, S. Ji, S. Peng, Y. Song, Backward stochastic
differential equations driven by G-Brownian motion, Stochastic Process. Appl.,
124 (2014), 759-784.

\bibitem {HJPS}M. Hu, S. Ji, S. Peng, Y. Song, Comparison theorem, Feynman-Kac
formula and Girsanov transformation for BSDEs driven by G-Brownian motion,
Stochastic Process. Appl., 124 (2014), 1170-1195.

\bibitem {HJL}M. Hu, X. Ji, G. Liu, On the strong Markov property for
stochastic differential equations driven by G-Brownian motion, Stochastic
Process. Appl., 131 (2021), 417-453.

\bibitem {HP09}M. Hu, S. Peng, On representation theorem of G-expectations and
paths of $G$-Brownian motion, Acta Math. Appl. Sin. Engl. Ser., 25 (2009), 539-546.

\bibitem {HWZ}M. Hu, F. Wang, G. Zheng, Quasi-continuous random variables and
processes under the $G$-expectation framework, Stochastic Process. Appl., 126
(2016), 2367-2387.

\bibitem {HYH}Y. Hu, Y. Lin, A. S. Hima, Quadratic backward stochastic
differential equations driven by G-Brownian motion: Discrete solutions and
approximation, Stochastic Process. Appl., 128 (2018), 3724-3750.

\bibitem {Kr}N. V. Krylov, Nonlinear Parabolic and Elliptic Equations of the
Second Order, Reidel Publishing Company (1987).

\bibitem {LPH}H. Li, S. Peng, A. S. Hima, Reflected solutions of backward
stochastic differential equations driven by G-Brownian motion, Sci. China
Math., 61 (2018), 1-26.

\bibitem {LRT}Y. Lin, Z. Ren, N. Touzi, J. Yang, Second order backward SDE
with random terminal time, Electron. J. Probab., 25 (2020), 1-43.

\bibitem {MPZ}A. Matoussi, D. Possamai, C. Zhou, Robust Utility maximization
in non-dominated models with 2BSDEs, Math. Finance, 25 (2015), 258-287.

\bibitem {P-P92}E. Pardoux, S. Peng, Backward stochastic differential
equations and quasilinear parabolic partial differential equations, in:
Stochastic Partial Differential Equations and Their Applications, Proc. IFIP,
in: LNCIS, vol. 176, 1992, 200-217.

\bibitem {Peng2004}S. Peng, Filtration consistent nonlinear expectations and
evaluations of contingent claims, Acta Math. Appl. Sin., 20(2) (2004), 1-24.

\bibitem {Peng2005}S. Peng, Nonlinear expectations and nonlinear Markov
chains, Chin. Ann. Math., 26B(2) (2005), 159-184.

\bibitem {P07a}S. Peng, $G$-expectation, $G$-Brownian Motion and Related
Stochastic Calculus of It\^{o} type, Stochastic analysis and applications,
Abel Symp., Vol. 2, Springer, Berlin, 2007, 541-567.

\bibitem {P08a}S. Peng, Multi-dimensional $G$-Brownian motion and related
stochastic calculus under $G$-expectation, Stochastic Process. Appl., 118
(2008), 2223-2253.

\bibitem {P2019}S. Peng, Nonlinear Expectations and Stochastic Calculus under
Uncertainty, Springer (2019).

\bibitem {PZ}T. Pham, J. Zhang, Two Person Zero-sum Game in Weak Formulation
and Path Dependent Bellman-Isaacs Equation, SIAM J. Control Optim., 52 (2014), 2090-2121.

\bibitem {STZ}H. M. Soner, N. Touzi, J. Zhang, Martingale Representation
Theorem under G-expectation, Stochastic Process. Appl., 121 (2011), 265-287.

\bibitem {STZ11}H. M. Soner, N. Touzi, J. Zhang, Wellposedness of Second Order
Backward SDEs, Probab. Theory Related Fields, 153 (2012), 149-190.

\bibitem {Song11}Y. Song, Some properties on G-evaluation and its applications
to G-martingale decomposition, Sci. China Math., 54(2) (2011), 287-300.
\end{thebibliography}
\end{document}